\documentclass[a4paper, reqno, 11pt]{amsart}
\usepackage[utf8]{inputenc}

\usepackage{amsthm}
\usepackage{amssymb}
\usepackage{exscale}
\usepackage[mathscr]{euscript}
\usepackage{url}
\usepackage[all,ps,tips,tpic]{xy}
\usepackage{graphicx}
\usepackage{color}
\newenvironment{altenumerate}
   {\begin{list}
      {\textup{(\theenumi)} }
      {\usecounter{enumi}
       \setlength{\labelwidth}{0pt}
       \setlength{\labelsep}{2pt}
       \setlength{\leftmargin}{0pt}
       \setlength{\itemsep}{\the\smallskipamount}
       \renewcommand{\theenumi}{\roman{enumi}}
      }}
   {\end{list}}

\newtheorem{lem}{Lemma}[section]

\newtheorem{definition}[lem]{Definition}
\newtheorem{defprop}[lem]{Definition/Proposition}

\newtheorem{cor}[lem]{Corollary}
\newtheorem{thm}[lem]{Theorem}
\newtheorem{prop}[lem]{Proposition}

\newtheorem{conj}[lem]{Conjecture}
\newtheorem{claim}[lem]{Claim}

\theoremstyle{remark}
\newtheorem{rem}[lem]{Remark}
\newtheorem{example}[lem]{Example}

\DeclareMathOperator{\Hom}{Hom}
\DeclareMathOperator{\Ext}{Ext}

\DeclareMathOperator{\Tor}{Tor}

\DeclareMathOperator{\im}{im}

\DeclareMathOperator{\supp}{supp}
\DeclareMathOperator{\Spa}{Spa}
\DeclareMathOperator{\Spec}{Spec}
\DeclareMathOperator{\Sp}{Sp}
\DeclareMathOperator{\Spf}{Spf}
\DeclareMathOperator{\alHom}{alHom}
\DeclareMathOperator{\End}{End}
\DeclareMathOperator{\Gal}{Gal}

\newcommand{\tr}{\operatorname*{tr}}
\newcommand{\GL}{\mathrm{GL}}
\newcommand{\ad}{\mathrm{ad}}
\newcommand{\perf}{\mathrm{perf}}
\newcommand{\spec}{\mathrm{sp}}
\newcommand{\Berk}{\mathrm{Berk}}
\newcommand{\fet}{\mathrm{f\acute{e}t}}
\newcommand{\et}{\mathrm{\acute{e}t}}
\newcommand{\Perf}{\mathrm{Perf}}

\newcommand{\Fil}{\mathrm{Fil}}
\newcommand{\gr}{\mathrm{gr}}

\newdir{ >}{{}*!/-5pt/@{>}}
\SelectTips{cm}{10}

\begin{document}

\title{Perfectoid Spaces}
\author{Peter Scholze}
\begin{abstract}
We introduce a certain class of so-called perfectoid rings and spaces, which give a natural framework for Faltings' almost purity theorem, and for which there is a natural tilting operation which exchanges characteristic $0$ and characteristic $p$. We deduce the weight-monodromy conjecture in certain cases by reduction to equal characteristic.
\end{abstract}

\date{\today}
\maketitle
\tableofcontents
\pagebreak

\section{Introduction}

In commutative algebra and algebraic geometry, some of the most subtle problems arise in the context of mixed characteristic, i.e. over local fields such as $\mathbb{Q}_p$ which are of characteristic $0$, but whose residue field $\mathbb{F}_p$ is of characteristic $p$. The aim of this paper is to establish a general framework for reducing certain problems about mixed characteristic rings to problems about rings in characteristic $p$. We will use this framework to establish a generalization of Faltings's almost purity theorem, and new results on Deligne's weight-monodromy conjecture.

The basic result which we want to put into a larger context is the following canonical isomorphism of Galois groups, due to Fontaine and Wintenberger, \cite{FontaineWintenberger}. A special case is the following result.

\begin{thm} The absolute Galois groups of $\mathbb{Q}_p(p^{1/p^\infty})$ and $\mathbb{F}_p((t))$ are canonically isomorphic.
\end{thm}

In other words, after adjoining all $p$-power roots of $p$ to a mixed characteristic field, it looks like an equal characteristic ring in some way. Let us first explain how one can prove this theorem. Let $K$ be the completion of $\mathbb{Q}_p(p^{1/p^\infty})$ and let $K^\flat$ be the completion of $\mathbb{F}_p((t))(t^{1/p^\infty})$; it is enough to prove that the absolute Galois groups of $K$ and $K^\flat$ are isomorphic. Let us first explain the relation between $K$ and $K^\flat$, which in vague terms consists in replacing the prime number $p$ by a formal variable $t$. Let $K^\circ$ and $K^{\flat \circ}$ be the subrings of integral elements. Then
\[
K^\circ / p = \mathbb{Z}_p[p^{1/p^\infty}]/p \cong \mathbb{F}_p[t^{1/p^\infty}]/t = K^{\flat \circ}/t\ ,
\]
where the middle isomorphism sends $p^{1/p^n}$ to $t^{1/p^n}$. Using it, one can define a continuous multiplicative, but nonadditive, map $K^\flat\rightarrow K$, $x\mapsto x^\sharp$, which sends $t$ to $p$. On $K^{\flat\circ}$, it is given by sending $x$ to $\lim_{n\rightarrow\infty} y_n^{p^n}$, where $y_n\in K^\circ$ is any lift of the image of $x^{1/p^n}$ in $K^{\flat\circ}/t = K^\circ/p$. Then one has an identification
\[
K^\flat = \varprojlim_{x\mapsto x^p} K\ , x\mapsto (x^\sharp,(x^{1/p})^\sharp,\ldots)\ .
\]
In order to prove the theorem, one has to construct a canonical finite extension $L^\sharp$ of $K$ for any finite extension $L$ of $K^\flat$. There is the following description. Say $L$ is the splitting field of a polynomial $X^d + a_{d-1}X^{d-1}+\ldots+a_0$, which is also the splitting field of $X^d + a_{d-1}^{1/p^n} X^{d-1}+\ldots+a_0^{1/p^n}$ for all $n\geq 0$. Then $L^\sharp$ can be defined as the splitting field of $X^d + (a_{d-1}^{1/p^n})^\sharp X^{d-1}+\ldots+(a_0^{1/p^n})^\sharp$ for $n$ large enough: these fields stabilize as $n\rightarrow \infty$.

In fact, the same ideas work in greater generality.

\begin{definition} A perfectoid field is a complete topological field $K$ whose topology is induced by a nondiscrete valuation of rank $1$, such that the Frobenius $\Phi$ is surjective on $K^\circ / p$.
\end{definition}

Here $K^\circ\subset K$ denotes the set of powerbounded elements. Generalizing the example above, a construction of Fontaine associates to any perfectoid field $K$ another perfectoid field $K^\flat$ of characteristic $p$, whose underlying multiplicative monoid can be described as
\[
K^\flat = \varprojlim_{x\mapsto x^p} K\ .
\]
The theorem above generalizes to the following result.

\begin{thm}\label{GeneralFW} The absolute Galois groups of $K$ and $K^\flat$ are canonically isomorphic.
\end{thm}

Our aim is to generalize this to a comparison of geometric objects over $K$ with geometric objects over $K^\flat$. The basic claim is the following.

\begin{claim} The affine line $\mathbb{A}^1_{K^\flat}$ `is equal to' the inverse limit $\varprojlim_{T\mapsto T^p} \mathbb{A}^1_K$, where $T$ is the coordinate on $\mathbb{A}^1$.
\end{claim}

One way in which this is correct is the observation that it is true on $K^\flat$-, resp. $K$-, valued points. Moreover, for any finite extension $L$ of $K$ corresponding to an extension $L^\flat$ of $K^\flat$, we have the same relation
\[
L^\flat = \varprojlim_{x\mapsto x^p} L\ .
\]
Looking at the example above, we see that the explicit description of the map between $\mathbb{A}^1_{K^\flat}$ and $\varprojlim_{T\mapsto T^p} \mathbb{A}^1_K$ involves a limit procedure. For this reason, a formalization of this isomorphism has to be of an analytic nature, and we have to use some kind of rigid-analytic geometry over $K$. We choose to work with Huber's language of adic spaces, which reinterprets rigid-analytic varieties as certain locally ringed topological spaces. In particular, any variety $X$ over $K$ has an associated adic space $X^\ad$ over $K$, which in turn has an underlying topological space $|X^\ad|$.

\begin{thm}\label{ThmHomeomA1} There is a homeomorphism of topological spaces
\[
|(\mathbb{A}^1_{K^\flat})^\ad|\cong \varprojlim_{T\mapsto T^p} |(\mathbb{A}^1_K)^\ad|\ .
\]
\end{thm}

Note that both sides of this isomorphism can be regarded as locally ringed topological spaces. It is natural to ask whether one can compare the structure sheaves on both sides. There is the obvious obstacle that the left-hand side has a sheaf of characteristic $p$ rings, whereas the right-hand side has a sheaf of characteristic $0$ rings. Fontaine's functors make it possible to translate between the two worlds. There is the following result.

\begin{definition} Let $K$ be a perfectoid field. A perfectoid $K$-algebra is a Banach $K$-algebra $R$ such that the set of powerbounded elements $R^\circ\subset R$ is bounded, and such that the Frobenius $\Phi$ is surjective on $R^\circ / p$.
\end{definition}

\begin{thm} There is natural equivalence of categories, called the tilting equivalence, between the category of perfectoid $K$-algebras and the category of perfectoid $K^\flat$-algebras. Here a perfectoid $K$-algebra $R$ is sent to the perfectoid $K^\flat$-algebra
\[
R^\flat = \varprojlim_{x\mapsto x^p} R\ .
\]
\end{thm}

We note in particular that for perfectoid $K$-algebras $R$, we still have a map $R^\flat\rightarrow R$, $f\mapsto f^\sharp$. An example of a perfectoid $K$-algebra is the algebra $R=K\langle T^{1/p^\infty}\rangle$ for which $R^\circ = K^\circ\langle T^{1/p^\infty}\rangle$ is the $p$-adic completion of $K[T^{1/p^\infty}]$. This is the completion of an algebra that appears on the right-hand side of Theorem \ref{ThmHomeomA1}. Its tilt is given by $R^\flat = K^\flat\langle T^{1/p^\infty}\rangle$, which is the completed perfection of an algebra that appears on the left-hand side of Theorem \ref{ThmHomeomA1}.

Now an affinoid perfectoid space is associated to a perfectoid affinoid $K$-algebra, which is a pair $(R,R^+)$, where $R$ is a perfectoid $K$-algebra, and $R^+\subset R^\circ$ is open and integrally closed (and often $R^+=R^\circ$). There is a natural way to form the tilt $(R^\flat,R^{\flat +})$. To such a pair $(R,R^+)$, Huber, \cite{HuberContVal}, associates a space $X=\Spa(R,R^+)$ of equivalence classes of continuous valuations $R\rightarrow \Gamma\cup \{0\}$, $f\mapsto |f(x)|$, which are $\leq 1$ on $R^+$. The topology on this space is generated by so-called rational subsets. Moreover, Huber defines presheaves $\mathcal{O}_X$ and $\mathcal{O}_X^+$ on $X$, whose global sections are $R$, resp. $R^+$. 

\begin{thm} Let $(R,R^+)$ be a perfectoid affinoid $K$-algebra, and let $X=\Spa(R,R^+)$, $X^\flat = \Spa(R^\flat,R^{\flat +})$.
\begin{altenumerate}
 \item[{\rm (i)}] There is a homeomorphism $X\cong X^\flat$, given by mapping $x\in X$ to the valuation $x^\flat \in X^\flat$ defined by $|f(x^\flat)| = |f^\sharp(x)|$. This homeomorphism identifies rational subsets.
\item[{\rm (ii)}] For any rational subset $U\subset X$ with tilt $U^ \flat\subset X^ \flat$, the pair $(\mathcal{O}_X(U),\mathcal{O}_X^+(U))$ is a perfectoid affinoid $K$-algebra with tilt $(\mathcal{O}_{X^\flat}(U^ \flat),\mathcal{O}_{X^\flat}^+(U^\flat))$.
\item[{\rm (iii)}] The presheaves $\mathcal{O}_X$, $\mathcal{O}_X^+$ are sheaves.
\item[{\rm (iv)}] The cohomology group $H^i(X,\mathcal{O}_X^+)$ is $\mathfrak{m}$-torsion for $i>0$.
\end{altenumerate}
\end{thm}

Here $\mathfrak{m}\subset K^\circ$ is the subset of topologically nilpotent elements. Part (iv) implies that $H^i(X,\mathcal{O}_X)=0$ for $i>0$, which gives Tate's acyclicity theorem in the context of perfectoid spaces. However, it says that this statement about the generic fibre extends {\it almost} to the integral level, in the language of Faltings's so-called almost mathematics. In fact, this is a general property of perfectoid objects: Many statements that are true on the generic fibre are automatically almost true on the integral level.

Using the theorem, one can define general perfectoid spaces by gluing affinoid perfectoid spaces $X=\Spa(R,R^+)$. We arrive at the following theorem.

\begin{thm} The category of perfectoid spaces over $K$ and the category of perfectoid spaces over $K^\flat$ are equivalent.
\end{thm}

We denote the tilting functor by $X\mapsto X^\flat$. Our next aim is to define an \'{e}tale topos of perfectoid spaces. This necessitates a generalization of Faltings's almost purity theorem, cf. \cite{FaltingsPadicHodgeTheory}, \cite{FaltingsAlmostEtale}.

\begin{thm} Let $R$ be a perfectoid $K$-algebra. Let $S/R$ be finite \'{e}tale. Then $S$ is a perfectoid $K$-algebra, and $S^\circ$ is almost finite \'{e}tale over $R^\circ$.
\end{thm}

In fact, as for perfectoid fields, it is easy to construct a fully faithful functor from the category of finite \'{e}tale $R^\flat$-algebras to finite \'{e}tale $R$-algebras, and the problem becomes to show that this functor is essentially surjective. But locally on $X=\Spa(R,R^+)$, the functor is essentially surjective by the result for perfectoid fields; one deduces the general case by a gluing argument.

Using this theorem, one proves the following theorem. Here, $X_\et$ denotes the \'{e}tale site of a perfectoid space $X$, and we denote by $X_\et^\sim$ the associated topos.

\begin{thm} Let $X$ be a perfectoid space over $K$ with tilt $X^\flat$ over $K^\flat$. Then tilting induces an equivalence of sites $X_\et\cong X^\flat_\et$.
\end{thm}

As a concrete application of this theorem, we have the following result. Here, we use the \'{e}tale topoi of adic spaces, which are the same as the \'{e}tale topoi of the corresponding rigid-analytic variety. In particular, the same theorem holds for rigid-analytic varieties.

\begin{thm}\label{EtTopoiPn} The \'{e}tale topos $(\mathbb{P}^{n,\ad}_{K^\flat})^\sim_\et$ is equivalent to the inverse limit $\varprojlim_\varphi (\mathbb{P}^{n,\ad}_K)^\sim_\et$.
\end{thm}

Here, one has to interpret the latter as the inverse limit of a fibred topos in an obvious way, and $\varphi$ is the map given on coordinates by $\varphi(x_0:\ldots:x_n) = (x_0^p:\ldots:x_n^p)$. The same theorem stays true for proper toric varieties without change. We note that the theorem gives rise to a projection map
\[
\pi: \mathbb{P}^n_{K^\flat}\rightarrow \mathbb{P}^n_K
\]
defined on topological spaces and \'{e}tale topoi of adic spaces, and which is given on coordinates by $\pi(x_0:\ldots:x_n) = (x_0^\sharp:\ldots:x_n^\sharp)$. In particular, we see again that this isomorphism is of a deeply analytic and transcendental nature.

We note that $(\mathbb{P}^n_{K^\flat})^\ad$ is itself not a perfectoid space, but $\varprojlim_\varphi (\mathbb{P}^n_{K^\flat})^\ad$ is, where $\varphi: \mathbb{P}^n_{K^\flat}\rightarrow \mathbb{P}^n_{K^\flat}$ denotes again the $p$-th power map on coordinates. However, $\varphi$ is purely inseparable and hence induces an isomorphism on topological spaces and \'{e}tale topoi, which is the reason that we have not written this inverse limit in Theorem \ref{ThmHomeomA1} and Theorem \ref{EtTopoiPn}.

Finally, we apply these results to the weight-monodromy conjecture. Let us recall its formulation. Let $k$ be a local field whose residue field is of characteristic $p$, let $G_k=\Gal(\bar{k}/k)$, and let $q$ be the cardinality of the residue field of $k$. For any finite-dimensional $\bar{\mathbb{Q}}_\ell$-representation $V$ of $G_k$, we have the monodromy operator $N:V\rightarrow V(-1)$ induced from the action of the $\ell$-adic inertia subgroup. It induces the monodromy filtration $\Fil_i^N V\subset V$, $i\in \mathbb{Z}$, characterized by the property that $N(\Fil_i^N V)\subset \Fil_{i-2}^N V(-1)$ for all $i\in \mathbb{Z}$ and $\gr_i^N V\cong \gr_{-i}^N V(-i)$ via $N^i$ for all $i\geq 0$.

\begin{conj}[Deligne, \cite{DeligneICM}] Let $X$ be a proper smooth variety over $k$, and let $V=H^i(X_{\bar{k}},\bar{\mathbb{Q}}_\ell)$. Then for all $j\in \mathbb{Z}$ and for any geometric Frobenius $\Phi\in G_k$, all eigenvalues of $\Phi$ on $\gr_j^N V$ are Weil numbers of weight $i+j$, i.e. algebraic numbers $\alpha$ such that $|\alpha|=q^{(i+j)/2}$ for all complex absolute values.
\end{conj}

Deligne, \cite{DeligneWeil2}, proved this conjecture if $k$ is of characteristic $p$, and the situation is already defined over a curve. The general weight-monodromy conjecture over fields $k$ of characteristic $p$ can be deduced from this case, as done by Terasoma, \cite{Terasoma}, and by Ito, \cite{Ito}.

In mixed characteristic, the conjecture is wide open. Introducing what is now called the Rapoport-Zink spectral sequence, Rapoport and Zink, \cite{RapoportZink}, have proved the conjecture when $X$ has dimension at most $2$ and $X$ has semistable reduction. They also show that in general it would follow from a suitable form of the standard conjectures, the main point being that a certain linear pairing on cohomology groups should be nondegenerate. Using de Jong's alterations, \cite{deJongAlterations}, one can reduce the general case to the case of semistable reduction, and in particular the case of dimension at most $2$ follows. Apart from that, other special cases are known. Notably, the case of varieties which admit $p$-adic uniformization by Drinfeld's upper half-space is proved by Ito, \cite{ItoUpperHalfPlane}, by pushing through the argument of Rapoport-Zink in this special case, making use of the special nature of the components of the special fibre, which are explicit rational varieties.

On the other hand, there is a large amount of activity that uses automorphic arguments to prove results in cases of certain Shimura varieties, notably those of type $U(1,n-1)$ used in the book of Harris-Taylor \cite{HarrisTaylor}. Let us only mention the work of Taylor and Yoshida, \cite{TaylorYoshida}, later completed by Shin, \cite{Shin}, and Caraiani, \cite{Caraiani}, as well as the independent work of Boyer, \cite{Boyer}, \cite{Boyer2}. Boyer's results were used by Dat, \cite{Dat}, to handle the case of varieties which admit uniformization by a covering of Drinfeld's upper half-space, thereby generalizing Ito's result.

Our last main theorem is the following.

\begin{thm} Let $k$ be a local field of characteristic $0$. Let $X$ be a geometrically connected proper smooth variety over $k$ such that $X$ is a set-theoretic complete intersection in a projective smooth toric variety. Then the weight-monodromy conjecture is true for $X$.
\end{thm}

Let us give a short sketch of the proof for a smooth hypersurface in $X\subset \mathbb{P}^n$, which is already a new result. We have the projection
\[
\pi: \mathbb{P}^n_{K^\flat}\rightarrow \mathbb{P}^n_K\ ,
\]
and we can look at the preimage $\pi^{-1}(X)$. One has an injective map $H^i(X)\rightarrow H^i(\pi^{-1}(X))$, and if $\pi^{-1}(X)$ were an algebraic variety, then one could deduce the result from Deligne's theorem in equal characteristic. However, the map $\pi$ is highly transcendental, and $\pi^{-1}(X)$ will not be given by equations. In general, it will look like some sort of fractal, have infinite-dimensional cohomology, and will have infinite degree in the sense that it will meet a line in infinitely many points. As an easy example, let
\[
X=\{x_0+x_1+x_2=0\}\subset (\mathbb{P}^2_K)^\ad\ .
\]
Then the homeomorphism
\[
|(\mathbb{P}^2_{K^\flat})^\ad|\cong \varprojlim_\varphi |(\mathbb{P}^2_K)^\ad|
\]
means that $\pi^{-1}(X)$ is topologically the inverse limit of the subvarieties
\[
X_n=\{x_0^{p^n}+x_1^{p^n}+x_2^{p^n}=0\}\subset (\mathbb{P}^2_K)^\ad\ .
\]
However, we have the following crucial approximation lemma.

\begin{lem} Let $\tilde{X}\subset (\mathbb{P}^n_K)^\ad$ be a small open neighborhood of the hypersurface $X$. Then there is a hypersurface $Y\subset \pi^{-1}(\tilde{X})$.
\end{lem}

The proof of this lemma is by an explicit approximation algorithm for the homogeneous polynomial defining $X$, and is the reason that we have to restrict to complete intersections. Using a result of Huber, one finds some $\tilde{X}$ such that $H^i(X) = H^i(\tilde{X})$, and hence gets a map $H^i(X)=H^i(\tilde{X})\rightarrow H^i(Y)$. As before, one checks that it is injective and concludes.

After the results of this paper were first announced, Kiran Kedlaya informed us that he had obtained related results in joint work with Ruochuan Liu, \cite{KedlayaLiu}. In particular, in our terminology, they prove that for any perfectoid $K$-algebra $R$ with tilt $R^\flat$, there is an equivalence between the finite \'{e}tale $R$-algebras and the finite \'{e}tale $R^\flat$-algebras. However, the tilting equivalence, the generalization of Faltings's almost purity theorem and the application to the weight-monodromy conjecture were not observed by them. This led to an exchange of ideas, with the following two influences on this paper. In the first version of this work, Theorem \ref{GeneralFW} was proved using a version of Faltings's almost purity theorem for fields, proved in \cite{GabberRamero}, Chapter 6, using ramification theory. Kedlaya observed that one could instead reduce to the case where $K^\flat$ is algebraically closed, which gives a more elementary proof of the theorem. We include both arguments here. Secondly, a certain finiteness condition on the perfectoid $K$-algebra was imposed at some places in the first version, a condition close to the notion of p-finiteness introduced below; in most applications known to the author, this condition is satisfied. Kedlaya made us aware of the possibility to deduce the general case by a simple limit argument.

{\bf Acknowledgments.} First, I want to express my deep gratitude to my advisor M. Rapoport, who suggested that I should think about the weight-monodromy conjecture, and in particular suggested that it might be possible to reduce it to the case of equal characteristic after a highly ramified base change. Next, I want to thank Gerd Faltings for a crucial remark on a first version of this paper. Moreover, I wish to thank all participants of the ARGOS seminar on perfectoid spaces at the University of Bonn in the summer term 2011, for working through an early version of this manuscript and the large number of suggestions for improvements. The same applies to Lorenzo Ramero, whom I also want to thank for his very careful reading of the manuscript. Moreover, I thank Roland Huber for answering my questions on adic spaces. Further thanks go to Ahmed Abbes, Bhargav Bhatt, Pierre Colmez, Laurent Fargues, Jean-Marc Fontaine, Ofer Gabber, Luc Illusie, Adrian Iovita, Kiran Kedlaya, Gerard Laumon, Ruochuan Liu, Wieslawa Niziol, Arthur Ogus, Martin Olsson, Bernd Sturmfels and Jared Weinstein for helpful discussions. Finally, I want to heartily thank the organizers of the CAGA lecture series at the IHES for their invitation. This work is the author's PhD thesis at the University of Bonn, which was supported by the Hausdorff Center for Mathematics, and the thesis was finished while the author was a Clay Research Fellow. He wants to thank both institutions for their support. Moreover, parts of it were written while visiting Harvard University, the Universit\'{e} Paris-Sud at Orsay, and the IHES, and the author wants to thank these institutions for their hospitality.

\section{Adic spaces}

Throughout this paper, we make use of Huber's theory of adic spaces. For this reason, we recall some basic definitions and statements about adic spaces over nonarchimedean local fields. We also compare Huber's theory to the more classical language of rigid-analytic geometry, and to the theory of Berkovich's analytic spaces. The material of this section can be found in \cite{Huber}, \cite{HuberDefAdic} and \cite{HuberContVal}.

\begin{definition} A nonarchimedean field is a topological field $k$ whose topology is induced by a nontrivial valuation of rank $1$.
\end{definition}

In particular, $k$ admits a norm $|\cdot|:k\rightarrow \mathbb{R}_{\geq 0}$, and it is easy to see that $|\cdot|$ is unique up to automorphisms $x\mapsto x^\alpha$, $0<\alpha<\infty$, of $\mathbb{R}_{\geq 0}$.

Throughout, we fix a nonarchimedean field $k$. Replacing $k$ by its completion will not change the theory, so we may and do assume that $k$ is complete.

The idea of rigid-analytic geometry, and the closely related theories of Berkovich's analytic spaces and Huber's adic spaces, is to have a nonarchimedean analogue of the notion of complex analytic spaces over $\mathbb{C}$. In particular, there should be a functor
\[
\{\mathrm{varieties}/k\}\rightarrow \{\mathrm{adic\ spaces}/k\}\ :\ X\mapsto X^{\ad}\ ,
\]
sending any variety over $k$ to its analytification $X^{\ad}$. Moreover, it should be possible to define subspaces of $X^{\ad}$ by inequalities: For any $f\in \Gamma(X,\mathcal{O}_X)$, the subset
\[
\{x\in X^{\ad}\mid |f(x)|\leq 1\}
\]
should make sense. In particular, any point $x\in X^{\ad}$ should give rise to a valuation function $f\mapsto |f(x)|$. In classical rigid-analytic geometry, one considers only the maximal points of the scheme $X$. Each of them gives a map $\Gamma(X,\mathcal{O}_X)\rightarrow k^\prime$ for some finite extension $k^\prime$ of $k$; composing with the unique extension of the absolute value of $k$ to $k^\prime$ gives a valuation on $\Gamma(X,\mathcal{O}_X)$. In Berkovich's theory, one considers norm maps $\Gamma(X,\mathcal{O}_X)\rightarrow \mathbb{R}_{\geq 0}$ inducing a fixed norm map on $k$. Equivalently, one considers valuations of rank $1$ on $\Gamma(X,\mathcal{O}_X)$. In Huber's theory, one allows also valuations of higher rank.

\begin{definition} Let $R$ be some ring. A valuation on $R$ is given by a multiplicative map $|\cdot|: R\rightarrow \Gamma\cup \{0\}$, where $\Gamma$ is some totally ordered abelian group, written multiplicatively, such that $|0|=0$, $|1|=1$ and $|x+y|\leq \max(|x|,|y|)$ for all $x,y\in R$.

If $R$ is a topological ring, then a valuation $|\cdot|$ on $R$ is said to be continuous if for all $\gamma\in \Gamma$, the subset $\{x\in R\mid |x|<\gamma\}\subset R$ is open.
\end{definition}

\begin{rem} The term valuation is somewhat unfortunate: If $\Gamma = \mathbb{R}_{>0}$, then this would usually be called a seminorm, and the term valuation would be used for (a constant multiple of) the map $x\mapsto -\log |x|$. On the other hand, the term higher-rank norm is much less commonly used than the term higher-rank valuation. For this reason, we stick with Huber's terminology.
\end{rem}

\begin{rem} Recall that a valuation ring is an integral domain $R$ such that for any $x\neq 0$ in the fraction field $K$ of $R$, at least one of $x$ and $x^{-1}$ is in $R$. Any valuation $|\cdot|$ on a field $K$ gives rise to the valuation subring $R=\{x\mid |x|\leq 1\}$. Conversely, a valuation ring $R$ gives rise to a valuation on $K$ with values in $\Gamma = K^\times/R^\times$, ordered by saying that $x\leq y$ if $x=yz$ for some $z\in R$. With respect to a suitable notion of equivalence of valuations defined below, this induces a bijective correspondence between valuation subrings of $K$ and valuations on $K$.
\end{rem}

If $|\cdot|: R\rightarrow \Gamma\cup \{0\}$ is a valuation on $R$, let $\Gamma_{|\cdot|}\subset \Gamma$ denote the subgroup generated by all $|x|$, $x\in R$, which are nonzero. The set $\supp(|\cdot|) = \{x\in R\mid |x|=0\}$ is a prime ideal of $R$ called the support of $|\cdot|$. Let $K$ be the quotient field of $R/\supp(|\cdot|)$. Then the valuation factors as a composite $R\rightarrow K\rightarrow \Gamma\cup \{0\}$. Let $R(|\cdot|)\subset K$ be the valuation subring, i.e. $R(|\cdot|) = \{x\in K\mid |x|\leq 1\}$.

\begin{definition} Two valuations $|\cdot|$, $|\cdot|^\prime$ are called equivalent if the following equivalent conditions are satisfied.
\begin{altenumerate}
\item[{\rm (i)}] There is an isomorphism of totally ordered groups $\alpha: \Gamma_{|\cdot|}\cong \Gamma_{|\cdot|^\prime}$ such that $|\cdot|^\prime = \alpha\circ |\cdot|$.
\item[{\rm (ii)}] The supports $\supp(|\cdot|) = \supp(|\cdot|^\prime)$ and valuation rings $R(|\cdot|) = R(|\cdot|^\prime)$ agree.
\item[{\rm (iii)}] For all $a,b\in R$, $|a|\geq |b|$ if and only if $|a|^\prime\geq |b|^\prime$.
\end{altenumerate}
\end{definition}

In \cite{HuberContVal}, Huber defines spaces of (continuous) valuations in great generality. Let us specialize to the case of interest to us.

\begin{definition}
\begin{altenumerate}
\item[{\rm (i)}] A Tate $k$-algebra is a topological $k$-algebra $R$ for which there exists a subring $R_0\subset R$ such that $aR_0$, $a\in k^\times$, forms a basis of open neighborhoods of $0$. A subset $M\subset R$ is called bounded if $M\subset aR_0$ for some $a\in k^\times$. An element $x\in R$ is called power-bounded if $\{x^n\mid n\geq 0\}\subset R$ is bounded. Let $R^\circ\subset R$ denote the subring of powerbounded elements.
\item[{\rm (ii)}] An affinoid $k$-algebra is a pair $(R,R^+)$ consisting of a Tate $k$-algebra $R$ and an open and integrally closed subring $R^+\subset R^\circ$.
\item[{\rm (iii)}] An affinoid $k$-algebra $(R,R^+)$ is said to be of topologically finite type (tft for short) if $R$ is a quotient of $k\langle T_1,\ldots,T_n\rangle$ for some $n$, and $R^+=R^\circ$.
\end{altenumerate}
\end{definition}

Here,
\[
k\langle T_1,\ldots,T_n\rangle = \{\sum_{i_1,\ldots,i_n\geq 0} x_{i_1,\ldots,i_n} T_1^{i_1}\cdots T_n^{i_n}\mid x_{i_1,\ldots,i_n}\in k, x_{i_1,\ldots,i_n}\rightarrow 0\}
\]
is the ring of convergent power series on the ball given by $|T_1|,\ldots,|T_n|\leq 1$. Often, only affinoid $k$-algebras of tft are considered; however, this paper will show that other classes of affinoid $k$-algebras are of interest as well. We also note that any Tate $k$-algebra $R$, resp. affinoid $k$-algebra $(R,R^+)$, admits the completion $\hat{R}$, resp. $(\hat{R},\hat{R}^+)$, which is again a Tate, resp. affinoid, $k$-algebra. Everything depends only on the completion, so one may assume that $(R,R^+)$ is complete in the following.

\begin{definition} Let $(R,R^+)$ be an affinoid $k$-algebra. Let
\[
X=\Spa(R,R^+) = \{|\cdot|: R\rightarrow \Gamma\cup \{0\}\ \mathrm{continuous\ valuation}\mid \forall f\in R^+: |f|\leq 1\}/\cong\ .
\]
For any $x\in X$, write $f\mapsto |f(x)|$ for the corresponding valuation on $R$. We equip $X$ with the topology which has the open subsets
\[
U(\frac{f_1,\ldots,f_n}g) = \{x\in X\mid \forall i: |f_i(x)|\leq |g(x)|\}\ ,
\]
called rational subsets, as basis for the topology, where $f_1,\ldots,f_n\in R$ generate $R$ as an ideal and $g\in R$.
\end{definition}

\begin{rem}\label{RemRationalSubset} Let $\varpi\in k$ be topologically nilpotent, i.e. $|\varpi|<1$. Then to $f_1,\ldots,f_n$ one can add $f_{n+1}=\varpi^N$ for some big integer $N$ without changing the rational subspace. Indeed, there are elements $h_1,\ldots,h_n\in R$ such that $\sum h_i f_i = 1$. Multiplying by $\varpi^N$ for $N$ sufficiently large, we have $\varpi^N h_i\in R^+$, as $R^+\subset R$ is open. Now for any $x\in U(\frac{f_1,\ldots,f_n}g)$, we have
\[
|\varpi^N(x)| = |\sum (\varpi^N h_i)(x) f_i(x)|\leq \max |(\varpi^N h_i)(x)| |f_i(x)|\leq |g(x)|\ ,
\]
as desired. In particular, we see that on rational subsets, $|g(x)|$ is nonzero, and bounded from below.
\end{rem}

The topological spaces $\Spa(R,R^+)$ have some special properties reminiscent of the properties of $\Spec(A)$ for a ring $A$. In fact, let us recall the following result of Hochster, \cite{Hochster}.

\begin{defprop} A topological space $X$ is called spectral if it satisfies the following equivalent properties.
\begin{altenumerate}
\item[{\rm (i)}] There is some ring $A$ such that $X\cong \Spec(A)$.
\item[{\rm (ii)}] One can write $X$ as an inverse limit of finite $T_0$ spaces.
\item[{\rm (iii)}] The space $X$ is quasicompact, has a basis of quasicompact open subsets stable under finite intersections, and every irreducible closed subset has a unique generic point.
\end{altenumerate}
\end{defprop}

In particular, spectral spaces are quasicompact, quasiseparated and $T_0$. Recall that a topological space $X$ is called quasiseparated if the intersection of any two quasicompact open subsets is again quasicompact. In the following we will often abbreviate quasicompact, resp. quasiseparated, as qc, resp. qs.

\begin{prop}[{\cite[Theorem 3.5]{HuberContVal}}] For any affinoid $k$-algebra $(R,R^+)$, the space $\Spa(R,R^+)$ is spectral. The rational subsets form a basis of quasicompact open subsets stable under finite intersections.
\end{prop}

\begin{prop}[{\cite[Proposition 3.9]{HuberContVal}}]\label{SpaCompletion} Let $(R,R^+)$ be an affinoid $k$-algebra with completion $(\hat{R},\hat{R}^+)$. Then $\Spa(R,R^+)\cong \Spa(\hat{R},\hat{R}^+)$, identifying rational subsets.
\end{prop}

Moreover, the space $\Spa(R,R^+)$ is large enough to capture important properties.

\begin{prop} Let $(R,R^+)$ be an affinoid $k$-algebra, $X=\Spa(R,R^+)$.
\begin{altenumerate}
\item[{\rm (i)}] If $X=\emptyset$, then $\hat{R}=0$.
\item[{\rm (ii)}] Let $f\in R$ be such that $|f(x)|\neq 0$ for all $x\in X$. If $R$ is complete, then $f$ is invertible.
\item[{\rm (iii)}] Let $f\in R$ be such that $|f(x)|\leq 1$ for all $x\in X$. Then $f\in R^+$.
\end{altenumerate}
\end{prop}

\begin{proof} Part (i) is \cite{HuberContVal}, Proposition 3.6 (i). Part (ii) is \cite{HuberDefAdic}, Lemma 1.4, and part (iii) follows from \cite{HuberContVal}, Lemma 3.3 (i).
\end{proof}

We want to endow $X=\Spa(R,R^+)$ with a structure sheaf $\mathcal{O}_X$. The construction is as follows.

\begin{definition}\label{DefinitionPresheaf} Let $(R,R^+)$ be an affinoid $k$-algebra, and let $U=U(\frac{f_1,\ldots,f_n}g)\subset X=\Spa(R,R^+)$ be a rational subset. Choose some $R_0\subset R$ such that $aR_0$, $a\in k^\times$, is a basis of open neighborhoods of $0$ in $R$. Consider the subalgebra $R[\frac{f_1}g,\ldots,\frac{f_n}g]$ of $R[g^{-1}]$, and equip it with the topology making $aR_0[\frac{f_1}g,\ldots,\frac{f_n}g]$, $a\in k^\times$, a basis of open neighborhoods of $0$. Let $B\subset R[\frac{f_1}g,\ldots,\frac{f_n}g]$ be the integral closure of $R^+[\frac{f_1}g,\ldots,\frac{f_n}g]$ in $R[\frac{f_1}g,\ldots,\frac{f_n}g]$. Then $(R[\frac{f_1}g,\ldots,\frac{f_n}g],B)$ is an affinoid $k$-algebra. Let $(R\langle \frac{f_1}g,\ldots,\frac{f_n}g\rangle,\hat{B})$ be its completion.
\end{definition}

Obviously,
\[
\Spa(R\langle \frac{f_1}g,\ldots,\frac{f_n}g\rangle,\hat{B})\rightarrow \Spa(R,R^+)
\]
factors over the open subset $U\subset X$.

\begin{prop}[{\cite[Proposition 1.3]{HuberDefAdic}}]\label{UnivPropertyPresheaf} In the situation of the definition, the following universal property is satisfied. For every complete affinoid $k$-algebra $(S,S^+)$ with a map $(R,R^+)\rightarrow (S,S^+)$ such that the induced map $\Spa(S,S^+)\rightarrow \Spa(R,R^+)$ factors over $U$, there is a unique map
\[
(R\langle \frac{f_1}g,\ldots,\frac{f_n}g\rangle,\hat{B})\rightarrow (S,S^+)
\]
making the obvious diagram commute.

In particular, $(R\langle \frac{f_1}g,\ldots,\frac{f_n}g\rangle,\hat{B})$ depends only on $U$. Define
\[
(\mathcal{O}_X(U),\mathcal{O}_X^+(U)) = (R\langle \frac{f_1}g,\ldots,\frac{f_n}g\rangle,\hat{B})\ .
\]
For example, $(\mathcal{O}_X(X),\mathcal{O}_X^+(X))$ is the completion of $(R,R^+)$.
\end{prop}

The idea is that since $f_1,\ldots,f_n$ generate $R$, not all $|f_i(x)|=0$ for $x\in U$; in particular, $|g(x)|\neq 0$ for all $x\in U$. This implies that $g$ is invertible in $S$ in the situation of the proposition. Moreover, $|(\frac{f_i}g)(x)|\leq 1$ for all $x\in U$, which means that $\frac{f_i}g\in S^+$.

We define presheaves $\mathcal{O}_X$ and $\mathcal{O}_X^+$ on $X$ as above on rational subsets, and for general open $U\subset X$ by requiring
\[
\mathcal{O}_X(W) = \varprojlim_{U\subset W\ \mathrm{rational}} \mathcal{O}_X(U)\ ,
\]
and similarly for $\mathcal{O}_X^+$.

\begin{prop}[{\cite[Lemma 1.5, Proposition 1.6]{HuberDefAdic}}] For any $x\in X$, the valuation $f\mapsto |f(x)|$ extends to the stalk $\mathcal{O}_{X,x}$, and
\[
\mathcal{O}_{X,x}^+ = \{f\in \mathcal{O}_{X,x}\mid |f(x)|\leq 1\}\ .
\]
The ring $\mathcal{O}_{X,x}$ is a local ring with maximal ideal given by $\{f\mid |f(x)|=0\}$. The ring $\mathcal{O}_{X,x}^+$ is a local ring with maximal ideal given by $\{f\mid |f(x)|<1\}$. Moreover, for any open subset $U\subset X$,
\[
\mathcal{O}_X^+(U) = \{f\in \mathcal{O}_X(U)\mid \forall x\in U: |f(x)|\leq 1\}\ .
\]

If $U\subset X$ is rational, then $U\cong \Spa(\mathcal{O}_X(U),\mathcal{O}_X^+(U))$ compatible with rational subsets, the presheaves $\mathcal{O}_X$ and $\mathcal{O}_X^+$, and the valuations at all $x\in U$.
\end{prop}

Unfortunately, it is not in general known\footnote{and probably not true} that $\mathcal{O}_X$ is sheaf. We note that the proposition ensures that $\mathcal{O}_X^+$ is a sheaf if $\mathcal{O}_X$ is. The basic problem is that completion behaves in general badly for nonnoetherian rings.

\begin{definition} A Tate $k$-algebra $R$ is called strongly noetherian if
\[
R\langle T_1,\ldots,T_n\rangle = \{\sum_{i_1,\ldots,i_n\geq 0} x_{i_1,\ldots,i_n} T_1^{i_1}\cdots T_n^{i_n}\mid x_{i_1,\ldots,i_n}\in \hat{R}, x_{i_1,\ldots,i_n}\rightarrow 0\}
\]
is noetherian for all $n\geq 0$.
\end{definition}

For example, if $R$ is of tft, then $R$ is strongly noetherian.

\begin{thm}[{\cite[Theorem 2.2]{HuberDefAdic}}] If $(R,R^+)$ is an affinoid $k$-algebra such that $R$ is strongly noetherian, then $\mathcal{O}_X$ is a sheaf.
\end{thm}

Later, we will show that this theorem is true under the assumption that $R$ is a perfectoid $k$-algebra. We define perfectoid $k$-algebras later; let us only remark that except for trivial examples, they are huge and in particular not strongly noetherian.

Finally, we can define the category of adic spaces over $k$. Namely, consider the category $(V)$ of triples $(X,\mathcal{O}_X,(|\cdot(x)|\mid x\in X))$ consisting of a locally ringed topological space $(X,\mathcal{O}_X)$, where $\mathcal{O}_X$ is a sheaf of complete topological $k$-algebras, and a continuous valuation $f\mapsto |f(x)|$ on $\mathcal{O}_{X,x}$ for every $x\in X$. Morphisms are given by morphisms of locally ringed topological spaces which are continuous $k$-algebra morphisms on $\mathcal{O}_X$, and compatible with the valuations in the obvious sense.

Any affinoid $k$-algebra $(R,R^+)$ for which $\mathcal{O}_X$ is a sheaf gives rise to such a triple $(X,\mathcal{O}_X,(|\cdot(x)|\mid x\in X))$. Call an object of $(V)$ isomorphic to such a triple an affinoid adic space.

\begin{definition} An adic space over $k$ is an object $(X,\mathcal{O}_X,(|\cdot(x)|\mid x\in X))$ of $(V)$ that is locally on $X$ an affinoid adic space.
\end{definition}

\begin{prop}[{\cite[Proposition 2.1 (ii)]{HuberDefAdic}}]\label{MapsToAffinoid} For any affinoid $k$-algebra $(R,R^+)$ with $X=\Spa(R,R^+)$ such that $\mathcal{O}_X$ is a sheaf, and any adic space $Y$ over $k$, we have
\[
\Hom(Y,X) = \Hom((\hat{R},\hat{R}^+),(\mathcal{O}_Y(Y),\mathcal{O}_Y^+(Y)))\ .
\]
Here, the latter set denotes the set of continuous $k$-algebra morphisms $\hat{R}\rightarrow \mathcal{O}_Y(Y)$ such that $\hat{R}^+$ is mapped into $\mathcal{O}_Y^+(Y)$.
\end{prop}

In particular, the category of complete affinoid $k$-algebras for which the structure presheaf is a sheaf is equivalent to the category of affinoid adic spaces over $k$.

For the rest of this section, let us discuss an example, and explain the relation to rigid-analytic varieties, \cite{TateRigidAnalytic}, and Berkovich's analytic spaces, \cite{BerkovichSpectralTheory}.

\begin{example} Assume that $k$ is complete and algebraically closed. Let $R=k\langle T\rangle$, $R^+=R^\circ = k^\circ\langle T\rangle$ the subspace of power series with coefficients in $k^\circ$. Then $(R,R^+)$ is an affinoid $k$-algebra of tft. We want to describe the topological space $X=\Spa(R,R^+)$. For convenience, let us fix the norm $|\cdot|: k\rightarrow \mathbb{R}_{\geq 0}$. Then there are in general points of $5$ different types, of which the first four are already present in Berkovich's theory.

\includegraphics[scale = 0.5, viewport = 0 200 0 560]{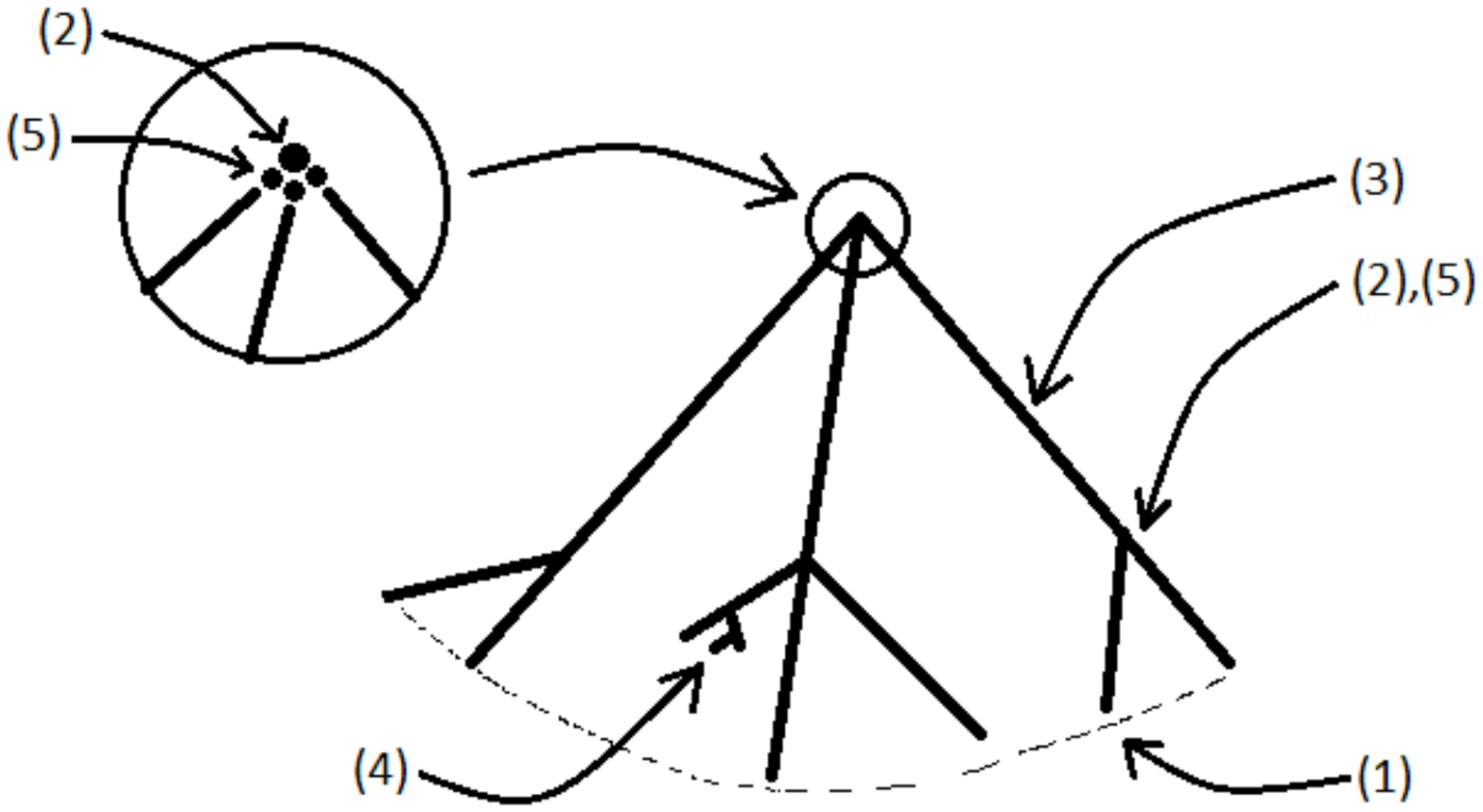}

\begin{altenumerate}
\item[{\rm (1)}] The classical points: Let $x\in k^\circ$, i.e. $x\in k$ with $|x|\leq 1$. Then for any $f\in k\langle T\rangle$, we can evaluate $f$ at $x$ to get a map $R\rightarrow k$, $f=\sum a_n T^n\mapsto \sum a_n x^n$. Composing with the norm on $k$, one gets a valuation $f\mapsto |f(x)|$ on $R$, which is obviously continuous and $\leq 1$ for all $f\in R^+$.
\item[{\rm (2), (3)}] The rays of the tree: Let $0\leq r\leq 1$ be some real number, and $x\in k^\circ$. Then
\[
f=\sum a_n (T-x)^n\mapsto \sup |a_n|r^n = \sup_{y\in k^\circ : |y-x|\leq r} |f(y)|
\]
defines another continuous valuation on $R$ which is $\leq 1$ for all $f\in R^+$. It depends only on the disk $D(x,r)=\{y\in k^\circ\mid |y-x|\leq r\}$. If $r=0$, then it agrees with the classical point corresponding to $x$. For $r=1$, the disk $D(x,1)$ is independent of $x\in k^\circ$, and the corresponding valuation is called the Gau\ss point.

If $r\in |k^\times|$, then the point is said to be of type (2), otherwise of type (3). Note that a branching occurs at a point corresponding to the disk $D(x,r)$ if and only if $r\in |k^\times|$, i.e. a branching occurs precisely at the points of type (2).
\item[{\rm (4)}] Dead ends of the tree: Let $D_1\supset D_2\supset \ldots$ be a sequence of disks with $\bigcap D_i=\emptyset$. Such families exist if $k$ is not spherically complete, e.g. if $k=\mathbb{C}_p$. Then
\[
f\mapsto \inf_i \sup_{x\in D_i} |f(x)|
\]
defines a valuation on $R$, which again is $\leq 1$ for all $f\in R^+$.
\item[{\rm (5)}] Finally, there are some valuations of rank $2$ which are only seen in the adic space. Let us first give an example, before giving the general classification. Consider the totally ordered abelian group $\Gamma = \mathbb{R}_{>0}\times \gamma^{\mathbb{Z}}$, where we require that $r<\gamma<1$ for all real numbers $r<1$. It is easily seen that there is a unique such ordering. Then
\[
f=\sum a_n (T-x)^n\mapsto \max |a_n| \gamma^n
\]
defines a rank-$2$-valuation on $R$. This is similar to cases (2), (3), but with the variable $r$ infinitesimally close to $1$. One may check that this point only depends on the disc $D(x,<1)=\{y\in k^\circ\mid |y-x|<1\}$.

Similarly, take any $x\in k^\circ$, some real number $0<r<1$ and choose a sign $?\in \{<,>\}$. Consider the totally ordered abelian group $\Gamma_{?r} = \mathbb{R}_{>0}\times \gamma^{\mathbb{Z}}$, where $r^\prime < \gamma < r$ for all real numbers $r^\prime < r$ if $?=<$, and $r^\prime>\gamma>r$ for all real numbers $r^\prime>r$ if $?=>$. Then
\[
f=\sum a_n (T-x)^n\mapsto \max |a_n| \gamma^n
\]
defines a rank $2$-valuation on $R$. If $?=<$, then it depends only on $D(x,<r) = \{y\in k^\circ\mid |y-x|<r\}$. If $?=>$, then it depends only on $D(x,r)$.

One checks that if $r\not\in |k^\times|$, then these points are all equivalent to the corresponding point of type (3). However, at each branching point, i.e. point of type (2), this gives exactly one additional point for each ray starting from this point.
\end{altenumerate}

All points except those of type (2) are closed. Let $\kappa$ be the residue field of $k$. Then the closure of the Gau\ss point is exactly the Gau\ss point together with the points of type (5) around it, and is homeomorphic to $\mathbb{A}^1_\kappa$, with the Gau\ss point as the generic point. At the other points of type (2), one gets $\mathbb{P}^1_\kappa$.

We note that in case (5), one could define a similar valuation
\[
f=\sum a_n T^n \mapsto \max |a_n| \gamma^n\ ,
\]
with $\gamma$ having the property that $r>\gamma>1$ for all $r>1$. This valuation would still be continuous, but it takes the value $\gamma>1$ on $T\in R^+$. This shows the relevance of the requirement $|f(x)|\leq 1$ for all $f\in R^+$, which is automatic for rank-$1$-valuations.
\end{example}

\begin{thm}[{\cite[(1.1.11)]{Huber}}] There is a fully faithful functor
\[
r: \{\mathrm{rigid-analytic\ varieties}/k\}\rightarrow \{\mathrm{adic\ spaces}/k\}\ :\ X\mapsto X^{\ad}
\]
sending $\Sp(R)$ to $\Spa(R,R^+)$ for any affinoid $k$-algebra $(R,R^+)$ of tft. It induces an equivalence
\[
\{\mathrm{qs\ rigid-analytic\ varieties}/k\}\cong \{\mathrm{qs\ adic\ spaces\ locally\ of\ finite\ type}/k\}\ ,
\]
where an adic space over $k$ is called locally of finite type if it is locally of the form $\Spa(R,R^+)$, where $(R,R^+)$ is of tft. Let $X$ be a rigid-analytic variety over $k$ with corresponding adic space $X^\ad$. As any classical point defines an adic point, we have $X\subset X^\ad$. If $X$ is quasiseparated, then mapping a quasicompact open subset $U\subset X^\ad$ to $U\cap X$ defines a bijection
\[
\{\mathrm{qc\ admissible\ opens\ in\ }X\}\cong \{\mathrm{qc\ opens\ in\ }X^\ad\}\ ,
\]
the inverse of which is denoted $U\mapsto \tilde{U}$. Under this bijection a family of quasicompact admissible opens $U_i\subset X$ forms an admissible cover if and only if the corresponding subsets $\tilde{U}_i\subset X^\ad$ cover $X^\ad$.

In particular, for any rigid-analytic variety $X$, the topos of sheaves on the Grothendieck site associated to $X$ is equivalent to the category of sheaves on the sober topological space $X^\ad$.
\end{thm}

We recall that for abstract reasons, there is up to equivalence at most one sober topological space with the last property in the theorem. This gives the topological space underlying the adic space a natural interpretation.

In the example $X=\Sp(k\langle T\rangle)$ discussed above, a typical example of a non-admissible cover is the cover of $\{x\mid |x|\leq 1\}$ as
\[
\{x\mid |x|\leq 1\} = \{x\mid |x|=1\}\cup \bigcup_{r<1} \{x\mid |x|\leq r\}\ .
\]
In the adic world, one can see this non-admissibility as being caused by the point of type (5) which gives $|x|$ a value $\gamma<1$ bigger than any $r<1$.

\includegraphics[scale = 0.5, viewport = 50 320 0 560]{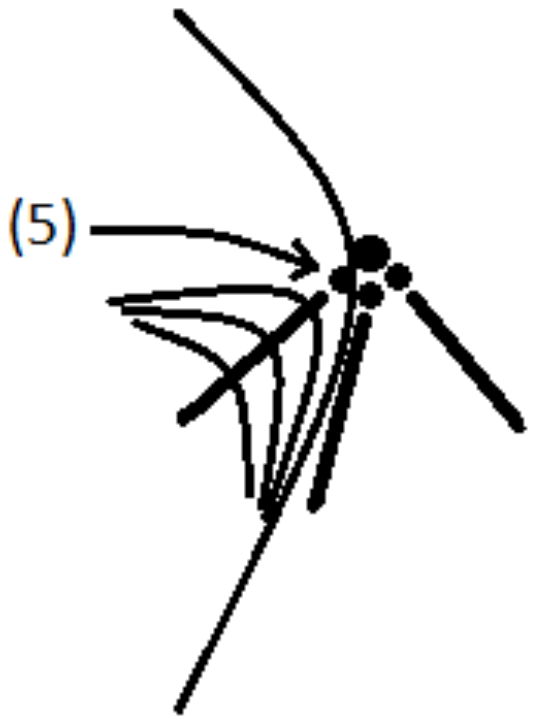}

Moreover, adic spaces behave well with respect to formal models. In fact, one can define adic spaces in greater generality so as to include locally noetherian formal schemes as a full subcategory, but we will not need this more general theory here.

\begin{thm} Let $\mathfrak{X}$ be some admissible formal scheme over $k^\circ$, let $X$ be its generic fibre in the sense of Raynaud, and let $X^\ad$ be the associated adic space. Then there is a continuous specialization map
\[
\spec: X^\ad\rightarrow \mathfrak{X}\ ,
\]
extending to a morphism of locally ringed topological spaces $(X^\ad,\mathcal{O}_{X^\ad}^+)\rightarrow (\mathfrak{X},\mathcal{O}_{\mathfrak{X}})$.

Now assume that $X$ is a fixed quasicompact quasiseparated adic space locally of finite type over $k$. By Raynaud, there exist formal models $\mathfrak{X}$ for $X$, unique up to admissible blowup. Then there is a homeomorphism
\[
X\cong \varprojlim_{\mathfrak{X}} \mathfrak{X}\ ,
\]
where $\mathfrak{X}$ runs over formal models of $X$, extending to an isomorphism of locally ringed topological spaces $(X,\mathcal{O}_X^+)\cong \varprojlim_{\mathfrak{X}} (\mathfrak{X},\mathcal{O}_{\mathfrak{X}})$, where the right-hand side is the inverse limit in the category of locally ringed topological spaces.
\end{thm}

\begin{proof} It is an easy exercise to deduce this from the previous theorem and the results of \cite{BoschLuetkebohmert}, Section 4, and \cite{HuberDefAdic}, Section 4.
\end{proof}

In the example, one can start with $\mathfrak{X} = \Spf(k^\circ\langle T\rangle)$ as a formal model; this gives $\mathbb{A}^1_\kappa$ as underlying topological space. After that, one can perform iterated blowups at closed points. This introduces additional $\mathbb{P}^1_\kappa$'s; the strict transform of each component survives in the inverse limit and gives the closure of a point of type (2). Note that the point of type (2) is given as the preimage of the generic point of the component in the formal model.

We note that in order to get continuity of $\spec$, it is necessary to use nonstrict equalities in the definition of open subsets.

Now, let us state the following theorem about the comparison of Berkovich's analytic spaces and Huber's adic spaces. For this, we need to recall the following definition:

\begin{definition} An adic space $X$ over $k$ is called taut if it is quasiseparated and for every quasicompact open subset $U\subset X$, the closure $\overline{U}$ of $U$ in $X$ is still quasicompact.
\end{definition}

Most natural adic spaces are taut, e.g. all affinoid adic spaces, or more generally all qcqs adic spaces, and also all adic spaces associated to separated schemes of finite type over $k$. However, in recent work of Hellmann, \cite{Hellmann}, studying an analogue of Rapoport-Zink period domains in the context of deformations of Galois representations, it was found that the weakly admissible locus is in general a nontaut adic space.

We note that one can also define taut rigid-analytic varieties, and that one gets an equivalence of categories between the category of taut rigid-analytic varieties over $k$ and taut adic spaces locally of finite type over $k$. Hence the first equivalence in the following theorem could be stated without reference to adic spaces.

\begin{thm}[{\cite[Proposition 8.3.1, Lemma 8.1.8]{Huber}}] There is an equivalence of categories
\[\begin{aligned}
\{\mathrm{hausdorff\ strictly\ }&k\mathrm{-analytic\ Berkovich\ spaces}\}\\
&\cong \{\mathrm{taut\ adic\ spaces\ locally\ of\ finite\ type}/k\}\ ,
\end{aligned}\]
sending $\mathcal{M}(R)$ to $\Spa(R,R^+)$ for any affinoid $k$-algebra $(R,R^+)$ of tft.

Let $X^\Berk$ map to $X^\ad$ under this equivalence. Then there is an injective map of sets $X^\Berk\rightarrow X^\ad$, whose image is precisely the subset of rank-$1$-valuations. This map is in general not continuous. It admits a continuous retraction $X^\ad\rightarrow X^\Berk$, which identifies $X^\Berk$ with the maximal hausdorff quotient of $X^\ad$.
\end{thm}

In the example above, the image of the map $X^\Berk\rightarrow X^\ad$ consists of the points of type (1) - (4). The retraction $X^\ad\rightarrow X^\Berk$ contracts each point of type (2) with all points of type (5) around it, mapping them to the corresponding point of type (2) in $X^\Berk$. For any map from $X^\ad$ to a hausdorff topological space, any point of type (2) will have the same image as the points of type (5) around it, as they lie in its topological closure, which verifies the last assertion of the theorem in this case.

Let us end this section by describing in more detail the fibres of the map $X^\ad\rightarrow X^\Berk$. In fact, this discussion is valid even for adic spaces which are not related to Berkovich spaces. As the following discussion is local, we restrict to the affinoid case.

Let $(R,R^+)$ be an affinoid $k$-algebra, and let $X=\Spa(R,R^+)$. We need not assume that $\mathcal{O}_X$ is a sheaf in the following. For any $x\in X$, we let $k(x)$ be the residue field of $\mathcal{O}_{X,x}$, and $k(x)^+\subset k(x)$ be the image of $\mathcal{O}_{X,x}^+$. We have the following crucial property, surprising at first sight.

\begin{prop} Let $\varpi\in k$ be topologically nilpotent. Then the $\varpi$-adic completion of $\mathcal{O}_{X,x}^+$ is equal to the $\varpi$-adic completion $\widehat{k(x)}^+$ of $k(x)^+$.
\end{prop}

\begin{proof} It is enough to note that kernel of the map $\mathcal{O}_{X,x}^+\rightarrow k(x)^+$, which is also the kernel of the map $\mathcal{O}_{X,x}\rightarrow k(x)$, is $\varpi$-divisible.
\end{proof}

\begin{definition} An affinoid field is pair $(K,K^+)$ consisting of a nonarchimedean field $K$ and an open valuation subring $K^+\subset K^\circ$.
\end{definition}

In other words, an affinoid field is given by a nonarchimedean field $K$ equipped with a continuous valuation (up to equivalence). In the situation above, $(k(x),k(x)^+)$ is an affinoid field. The completion of an affinoid field is again an affinoid field. Also note that affinoid fields for which $k\subset K$ are affinoid $k$-algebras. The following description of points is immediate.

\begin{prop} Let $(R,R^+)$ be an affinoid $k$-algebra. The points of $\Spa(R,R^+)$ are in bijection with maps $(R,R^+)\rightarrow (K,K^+)$ to complete affinoid fields $(K,K^+)$ such that the quotient field of the image of $R$ in $K$ is dense.
\end{prop}

\begin{definition} For two points $x,y$ in some topological space $X$, we say that $x$ specializes to $y$ (or $y$ generalizes to $x$), written $x\succ y$ (or $y\prec x$), if $y$ lies in the closure $\overline{\{x\}}$ of $x$.
\end{definition}

\begin{prop}[{\cite[(1.1.6) - (1.1.10)]{Huber}}] Let $(R,R^+)$ be an affinoid $k$-algebra, and let $x,y\in X=\Spa(R,R^+)$ correspond to maps $(R,R^+)\rightarrow (K,K^+)$, resp. $(R,R^+)\rightarrow (L,L^+)$. Then $x\succ y$ if and only if $K\cong L$ as topological $R$-algebras and $L^+\subset K^+$.

For any point $y\in X$, the set $\{x\mid x\succ y\}$ of generalizations of $y$ is a totally ordered chain of length exactly the rank of the valuation corresponding to $y$.
\end{prop}

Note that in particular, for a given complete nonarchimedean field $K$ with a map $R\rightarrow K$, there is the point $x_0$ corresponding to $(K,K^\circ)$. This corresponds to the unique continuous rank-$1$-valuation on $K$. The point $x_0$ specializes to any other point for the same $K$.

\section{Perfectoid fields}

\begin{definition} A perfectoid field is a complete nonarchimedean field $K$ of residue characteristic $p>0$ whose associated rank-$1$-valuation is nondiscrete, such that the Frobenius is surjective on $K^\circ/p$.
\end{definition}

We note that the requirement that the valuation is nondiscrete is needed to exclude unramified extensions of $\mathbb{Q}_p$. It has the following consequence.

\begin{lem} Let $|\cdot|: K\rightarrow \Gamma\cup \{0\}$ be the unique rank-$1$-valuation on $K$, where $\Gamma = |K^\times|$ is chosen minimal. Then $\Gamma$ is $p$-divisible.
\end{lem}

\begin{proof} As $\Gamma\neq |p|^{\mathbb{Z}}$, the group $\Gamma$ is generated by the set of all $|x|$ for $x\in K$ with $|p|<|x|\leq 1$. For such $x$, choose some $y$ such that $|x-y^p|\leq |p|$. Then $|y|^p = |y^p| = |x|$, as desired.
\end{proof}

The class of perfectoid fields naturally separates into the fields of characteristic $0$ and those of characteristic $p$. In characteristic $p$, a perfectoid field is the same as a complete perfect nonarchimedean field.

\begin{rem} The notion of a perfectoid field is closely related to the notion of a deeply ramified field. Taking the definition of deeply ramified fields given in \cite{GabberRamero}, we remark that Proposition 6.6.6 of \cite{GabberRamero} says that a perfectoid field $K$ is deeply ramified, and conversely, a complete deeply ramified field with valuation of rank $1$ is a perfectoid field.
\end{rem}

Now we describe the process of tilting for perfectoid fields, which is a functor from the category of all perfectoid fields to the category of perfectoid fields in characteristic $p$.

For its first description, choose some element $\varpi\in K^\times$ such that $|p|\leq |\varpi|<1$. Now consider
\[
\varprojlim_\Phi K^\circ/\varpi\ ,
\]
where $\Phi$ denotes the Frobenius morphism $x\mapsto x^p$. This gives a perfect ring of characteristic $p$. We equip it with the inverse limit topology; note that each $K^\circ/\varpi$ naturally has the discrete topology.

\begin{lem}\label{DefTiltingField}\begin{altenumerate}
\item[{\rm (i)}] There is a multiplicative homeomorphism
\[
\varprojlim_{x\mapsto x^p} K^\circ\buildrel\cong\over\rightarrow \varprojlim_\Phi K^\circ/\varpi\ ,
\]
given by projection. In particular, the right-hand side is independent of $\varpi$. Moreover, we get a map
\[
\varprojlim_\Phi K^\circ/\varpi\rightarrow K^\circ\ :\ x\mapsto x^\sharp\ .
\]
\item[{\rm (ii)}] There is an element $\varpi^\flat\in \varprojlim_\Phi K^\circ/\varpi$ with $|(\varpi^\flat)^\sharp| = |\varpi|$. Define
\[
K^\flat = (\varprojlim_\Phi K^\circ/\varpi)[(\varpi^{\flat})^{-1}]\ .
\]
\item[{\rm (iii)}] There is a multiplicative homeomorphism
\[
K^\flat = \varprojlim_{x\mapsto x^p} K\ .
\]
In particular, there is a map $K^\flat\rightarrow K$, $x\mapsto x^\sharp$. Then $K^\flat$ is a perfectoid field of characteristic $p$,
\[
K^{\flat \circ} = \varprojlim_{x\mapsto x^p} K^\circ\cong \varprojlim_\Phi K^\circ/\varpi\ ,
\]
and the rank-$1$-valuation on $K^\flat$ can be defined by $|x|_{K^\flat}=|x^\sharp|_K$. We have $|K^{\flat \times}| = |K^\times|$. Moreover,
\[
K^{\flat \circ}/ \varpi^\flat \cong K^\circ / \varpi\ ,\ K^{\flat \circ}/\mathfrak{m}^\flat = K^\circ/\mathfrak{m}\ ,
\]
where $\mathfrak{m}$, resp. $\mathfrak{m}^\flat$, is the maximal ideal of $K^\circ$, resp. $K^{\flat \circ}$.
\item[{\rm (iv)}] If $K$ is of characteristic $p$, then $K^\flat = K$.
\end{altenumerate}
\end{lem}

We call $K^\flat$ the tilt of $K$.

\begin{rem} Obviously, $\varpi^\flat$ in (ii) is not unique. As it is not harmful to replace $\varpi$ with an element of the same norm, we will usually redefine $\varpi = (\varpi^\flat)^\sharp$, which comes equipped with a compatible system
\[
\varpi^{1/p^n} = ((\varpi^{\flat})^{1/p^n})^\sharp
\]
of $p^n$-th roots. Conversely, $\varpi$ together with such a choice of $p^n$-th roots gives an element $\varpi^\flat$ of
\[
K^\flat = \varprojlim_{x\mapsto x^p} K\ .
\]
\end{rem}

\begin{proof}\begin{altenumerate}
\item[{\rm (i)}] We begin by constructing a multiplicative continuous map
\[
\varprojlim_\Phi K^\circ/\varpi\rightarrow K^\circ\ :\ x\mapsto x^\sharp\ .
\]
Let $(\overline{x}_0,\overline{x}_1,\ldots)\in \varprojlim_\Phi K^\circ/\varpi$. Choose any lift $x_n\in K^\circ$ of $\overline{x}_n$. Then we claim that the limit
\[
x^\sharp = \lim_{n\rightarrow \infty} x_n^{p^n}
\]
exists and is independent of all choices. For this, it is enough to see that $x_n^{p^n}$ gives a well-defined element of $K^\circ/\varpi^{n+1}$. But if $x_n^\prime$ is a second lift, then $x_n - x_n^\prime$ is divisible by $\varpi$. One checks by induction on $i=0,1,\ldots,n$ that
\[
(x_n^{\prime})^{p^i} - x_n^{p^i} = (x_n^{p^{i-1}} + ((x_n^\prime)^{p^{i-1}} - x_n^{p^{i-1}}))^p - x_n^{p^i}
\]
gives a well-defined element of $K^\circ/\varpi^{i+1}$, using that $\varpi|p$.

It is clear from the definition that $x\mapsto x^\sharp$ is multiplicative and continuous. Now the map $\varprojlim_\Phi K^\circ/\varpi\rightarrow \varprojlim_{x\mapsto x^p} K^\circ$ given by $x\mapsto (x^\sharp, (x^{1/p})^\sharp,\ldots)$ gives an inverse to the obvious projection map.

\item[{\rm (ii)}] Pick some element $\varpi_1$ with $|\varpi_1|^p = |\varpi|$. Then $\varpi_1$ defines a nonzero element of $K^\circ / \varpi$. Choose any sequence
\[
\varpi^\flat = (0,\varpi_1,\ldots)\in \varprojlim_\Phi K^\circ/\varpi\ ;
\]
this is possible by surjectivity of $\Phi$ on $K^\circ/\varpi$. By the proof of part (i), we have $|(\varpi^\flat)^\sharp - \varpi_1^p|\leq |\varpi|^2$. This gives $|(\varpi^\flat)^\sharp|=|\varpi|$, as desired.

\item[{\rm (iii)}] As $x\mapsto x^\sharp$ is multiplicative, it extends to a map
\[
K^\flat\rightarrow \varprojlim_{x\mapsto x^p} K\ .
\]
One easily checks that it is a homeomorphism; in particular $K^\flat$ is a field. One also checks that the topology on $\varprojlim_\Phi K^\circ/\varpi$ is induced by the norm $x\mapsto |x^\sharp|$. Hence the topology on $K^\flat$ is induced by the rank-$1$-valuation $x\mapsto |x^\sharp|$. Clearly, $K^\flat$ is perfect and complete, so $K^\flat$ is a perfectoid field of characteristic $p$. One easily deduces all other claims.

\item[{\rm (iv)}] This is clear since $K^\flat\cong \varprojlim_{x\mapsto x^p} K$.
\end{altenumerate}
\end{proof}

Recall that when working with adic spaces, it is important to understand the continuous valuations on $K$. Under the process of tilting, we have the following equivalence.

\begin{prop}\label{TiltingValuationFields1} Let $K$ be a perfectoid field with tilt $K^\flat$. Then the continuous valuations $|\cdot|$ of $K$ (up to equivalence) are mapped bijectively to the continuous valuations $|\cdot|^\flat$ of $K^\flat$ (up to equivalence) via $|x|^\flat = |x^\sharp|$.
\end{prop}

\begin{proof} First, we check that the map $|\cdot|\mapsto |\cdot|^\flat$ maps valuations to valuations. All properties except $|x+y|^\flat\leq \max(|x|^\flat,|y|^\flat)$ are immediate, using that $x\mapsto x^\sharp$ is multiplicative. But
\[\begin{aligned}
|x+y|^\flat&= \lim_{n\rightarrow \infty} |(x^{1/p^n})^\sharp + (y^{1/p^n})^\sharp|^{p^n}\\
&\leq \max(\lim_{n\rightarrow \infty} |(x^{1/p^n})^\sharp|^{p^n}, \lim_{n\rightarrow \infty} |(y^{1/p^n})^\sharp|^{p^n})=\mathrm{max}(|x^\sharp|,|y^\sharp|)\ .
\end{aligned}\]
It is clear that continuity is preserved.

On the other hand, continuous valuations are in bijection with open valuation subrings $K^+\subset K^\circ$. They necessarily contain the topologically nilpotent elements $\mathfrak{m}$. We see that open valuation subrings $K^+\subset K^\circ$ are in bijection with valuation subrings in $K^\circ/\mathfrak{m} = K^{\flat \circ}/\mathfrak{m}^\flat$. This implies that one gets a bijection with continuous valuations of $K$ and continuous valuations of $K^\flat$ which is easily seen to be the one described.
\end{proof}

The main theorem about tilting for perfectoid fields is the following theorem. For many fields, this was known by the classical work of Fontaine-Wintenberger.

\begin{thm}\label{TiltingEquivFields} Let $K$ be a perfectoid field.
\begin{altenumerate}
\item[{\rm (i)}] Let $L$ be a finite extension of $K$. Then $L$ (with its natural topology as a finite-dimensional $K$-vector space) is a perfectoid field.
\item[{\rm (ii)}] Let $K^\flat$ be the tilt of $K$. Then the tilting functor $L\mapsto L^\flat$ induces an equivalence of categories between the category of finite extensions of $K$ and the category of finite extensions of $K^\flat$. This equivalence preserves degrees.
\end{altenumerate}
\end{thm}

It turns out that many of the arguments will generalize directly to the context of perfectoid $K$-algebras introduced later. For this reason, we defer the proof of this theorem. Let us only prove the following special case here.

\begin{prop}\label{PerfectoidAlgebrClosed} Let $K$ be a perfectoid field with tilt $K^\flat$. If $K^\flat$ is algebraically closed, then $K$ is algebraically closed.
\end{prop}

\begin{proof} Let $P(X)=X^d + a_{d-1} X^{d-1} + \ldots + a_0\in K^\circ[X]$ be any monic irreducible polynomial of positive degree $d$. Then the Newton polygon of $P$ is a line. Moreover, we may assume that the constant term of $P$ has absolute value $|a_0|=1$, as $|K^\times |=|K^{\flat \times}|$ is a $\mathbb{Q}$-vector space.

Now let $Q(X)=X^d + b_{d-1}X^{d-1}+\ldots+b_0\in K^{\flat \circ}[X]$ be any polynomial such that $P$ and $Q$ have the same image in $K^\circ/\varpi[X] = K^{\flat \circ}/\varpi^\flat[X]$, and let $y\in K^{\flat \circ}$ be a root of $Q$.

Considering $P(X+y^\sharp)$, we see that the constant term $P(y^\sharp)$ is divisible by $\varpi$.  As it is still irreducible, its Newton polygon is a line and hence the polynomial $P_1(X) = c^{-d}P(cX+y^\sharp)$ has integral coefficients again, where $|c|^d = |P(y^\sharp)|\leq |\varpi|$. Repeating the arguments gives an algorithm converging to a root of $P$.
\end{proof}

Our proof of Theorem \ref{TiltingEquivFields} will make use of Faltings's almost mathematics. For this reason, we recall some necessary background in the next section.

\section{Almost mathematics}

We will use the book of Gabber-Ramero, \cite{GabberRamero}, as our basic reference.

Fix a perfectoid field $K$. Let $\mathfrak{m}=K^{\circ\circ}\subset K^\circ$ be the subset of topologically nilpotent elements; it is also the set $\{x\in K\mid |x|<1\}$, and the unique maximal ideal of $K^\circ$. The basic idea of almost mathematics is that one neglects $\mathfrak{m}$-torsion everywhere.

\begin{definition} Let $M$ be a $K^\circ$-module. An element $x\in M$ is almost zero if $\mathfrak{m} x=0$. The module $M$ is almost zero if all of its elements are almost zero; equivalently, $\mathfrak{m} M = 0$.
\end{definition}

\begin{lem} The full subcategory of almost zero objects in $K^\circ-\mathrm{mod}$ is thick.
\end{lem}

\begin{proof} The only nontrivial part is to show that it is stable under extensions, so let
\[
0\rightarrow M^\prime\rightarrow M\rightarrow M^{\prime\prime}\rightarrow 0
\]
be a short exact sequence of $K^\circ$-modules, with $\mathfrak{m} M^\prime = \mathfrak{m} M^{\prime\prime} = 0$. In general, one gets that $\mathfrak{m}^2 M = 0$. But in our situation, $\mathfrak{m}^2 = \mathfrak{m}$, so $M$ is almost zero.
\end{proof}

We note that there is a sequence of localization functors
\[
K^\circ-\mathrm{mod}\rightarrow K^\circ-\mathrm{mod}/(\mathfrak{m}-\mathrm{torsion})\rightarrow K-\mathrm{mod}\ .
\]
Their composite is the functor of passing from an integral structure to its generic fibre. In this sense, the category in the middle can be seen as a slightly generic fibre, or as an almost integral structure. It will turn out that in perfectoid situations, properties and objects over the generic fibre will extend automatically to the slightly generic fibre, in other words the generic fibre almost determines the integral level. It will be easy to justify this philosophy if $K$ has characteristic $p$, by using the following argument. Assume that some statement is true over $K$. By using some finiteness property, it follows that there is some big $N$ such that it is true up to $\varpi^N$-torsion. But Frobenius is bijective, hence the property stays true up to $\varpi^{N/p}$-torsion. Now iterate this argument to see that it is true up to $\varpi^{N/p^m}$-torsion for all $m$, i.e. almost true.

Following these ideas, our proof of Theorem \ref{TiltingEquivFields} will proceed as follows, using the subscript $\fet$ to denote categories of finite \'{e}tale (almost) algebras.
\[
K_\fet\cong K^{\circ a}_\fet\cong (K^{\circ a}/\varpi)_\fet = (K^{\flat \circ a}/\varpi^\flat)_\fet\cong K^{\flat \circ a}_\fet\cong K^\flat_\fet\ .
\]
Our principal aim in this section is to define all intermediate categories.

\begin{definition} Define the category of almost $K^\circ$-modules as
\[
K^{\circ a}-\mathrm{mod} = K^\circ-\mathrm{mod}/(\mathfrak{m}-\mathrm{torsion})\ .
\]
In particular, there is a localization functor $M\mapsto M^a$ from $K^\circ-\mathrm{mod}$ to $K^{\circ a}-\mathrm{mod}$, whose kernel is exactly the thick subcategory of almost zero modules.
\end{definition}

\begin{prop}[{\cite[\S 2.2.2]{GabberRamero}}] Let $M$, $N$ be two $K^\circ$-modules. Then
\[
\Hom_{K^{\circ a}}(M^a,N^a) = \Hom_{K^\circ}(\mathfrak{m}\otimes M,N)\ .
\]
In particular, $\Hom_{K^{\circ a}}(X,Y)$ has a natural structure of $K^\circ$-module for any two $K^{\circ a}$-modules $X$ and $Y$. The module $\Hom_{K^{\circ a}}(X,Y)$ has no almost zero elements.
\end{prop}

For two $K^{\circ a}$-modules $M$, $N$, we define $\alHom(X,Y) = \Hom(X,Y)^a$.

\begin{prop}[{\cite[\S 2.2.6, \S 2.2.12]{GabberRamero}}] The category $K^{\circ a}-\mathrm{mod}$ is an abelian tensor category, where we define kernels, cokernels and tensor products in the unique way compatible with their definition in $K^\circ-\mathrm{mod}$, e.g.
\[
M^a\otimes N^a = (M\otimes N)^a
\]
for any two $K^\circ$-modules $M$, $N$. For any three $K^{\circ a}$-modules $L, M, N$, there is a functorial isomorphism
\[
\Hom(L,\alHom(M,N)) = \Hom(L\otimes M,N)\ .
\]
\end{prop}

This means that $K^{\circ a}-\mathrm{mod}$ has all abstract properties of the category of modules over a ring. In particular, one can define in the usual abstract way the notion of a $K^{\circ a}$-algebra. For any $K^{\circ a}$-algebra $A$, one also has the notion of an $A$-module. Any $K^\circ$-algebra $R$ defines a $K^{\circ a}$-algebra $R^a$, as the tensor products are compatible. Moreover, localization also gives a functor from $R$-modules to $R^a$-modules. For example, $K^\circ$ gives the $K^{\circ a}$-algebra $A=K^{\circ a}$, and then $A$-modules are $K^{\circ a}$-modules, so that the terminology is consistent.

\begin{prop}[{\cite[Proposition 2.2.14]{GabberRamero}}] There is a right adjoint
\[
K^{\circ a}-\mathrm{mod}\rightarrow K^\circ-\mathrm{mod}\ :\ M\mapsto M_\ast
\]
to the localization functor $M\mapsto M^a$, given by the functor of almost elements
\[
M_\ast = \Hom_{K^{\circ a}}(K^{\circ a},M)\ .
\]
The adjunction morphism $(M_\ast)^a\rightarrow M$ is an isomorphism. If $M$ is a $K^\circ$-module, then $(M^a)_\ast = \Hom(\mathfrak{m},M)$.
\end{prop}

If $A$ is a $K^{\circ a}$-algebra, then $A_\ast$ has a natural structure as $K^\circ$-algebra and $A_\ast^a = A$. In particular, any $K^{\circ a}$-algebra comes via localization from a $K^\circ$-algebra. Moreover, the functor $M\mapsto M_\ast$ induces a functor from $A$-modules to $A_\ast$-modules, and one sees that also all $A$-modules come via localization from $A_\ast$-modules. We note that the category of $A$-modules is again an abelian tensor category, and all properties about the category of $K^{\circ a}$-modules stay true for the category of $A$-modules. We also note that one can equivalently define $A$-algebras as algebras over the category of $A$-modules, or as $K^{\circ a}$-algebras $B$ with an algebra morphism $A\rightarrow B$.

Finally, we need to extend some notions from commutative algebra to the almost context.

\begin{defprop} Let $A$ be any $K^{\circ a}$-algebra.
\begin{altenumerate}
\item[{\rm (i)}] An $A$-module $M$ is flat if the functor $X\mapsto M\otimes_A X$ on $A$-modules is exact. If $R$ is a $K^\circ$-algebra and $N$ is an $R$-module, then the $R^a$-module $N^a$ is flat if and only if for all $R$-modules $X$ and all $i>0$, the module $\Tor_i^R(N,X)$ is almost zero.
\item[{\rm (ii)}] An $A$-module $M$ is almost projective if the functor $X\mapsto \alHom_A(M,X)$ on $A$-modules is exact. If $R$ is a $K^\circ$-algebra and $N$ is an $R$-module, then $N^a$ is almost projective over $R^a$ if and only if for all $R$-modules $X$ and all $i>0$, the module $\Ext^i_R(N,X)$ is almost zero.
\item[{\rm (iii)}] If $R$ is a $K^\circ$-algebra and $N$ is an $R$-module, then $M=N^a$ is said to be an almost finitely generated (resp. almost finitely presented) $R^a$-module if and only if for all $\epsilon\in \mathfrak{m}$, there is some finitely generated (resp. finitely presented) $R$-module $N_\epsilon$ with a map $f_\epsilon: N_\epsilon\rightarrow N$ such that the kernel and cokernel of $f_\epsilon$ are annihilated by $\epsilon$. We say that $M$ is uniformly almost finitely generated if there is some integer $n$ such that $N_\epsilon$ can be chosen to be generated by $n$ elements, for all $\epsilon$.
\end{altenumerate}
\end{defprop}

\begin{proof} For parts (i) and (ii), cf. \cite{GabberRamero}, Definition 2.4.4, \S 2.4.10 and Remark 2.4.12 (i). For part (iii), cf. \cite{GabberRamero}, Definition 2.3.8, Remark 2.3.9 (i) and Corollary 2.3.13.
\end{proof}

\begin{rem}\label{StrangeSums} In (iii), we make the implicit statement that this property depends only on the $R^a$-module $N^a$. There is also the categorical notion of projectivity saying that the functor $X\mapsto \Hom(M,X)$ is exact, but not even $K^{\circ a}$ itself is projective in general: One can check that the map
\[
K^\circ = \Hom(K^{\circ a},K^{\circ a})\rightarrow \Hom(K^{\circ a},K^{\circ a}/\varpi) = \Hom(\mathfrak{m},K^\circ/\varpi)
\]
is in general not surjective, as the latter group contains sums of the form
\[
\sum_{i\geq 0} \varpi^{1-1/p^i} x_i
\]
for arbitrary $x_i\in K^\circ/\varpi$.
\end{rem}

\begin{example} As an example of an almost finitely presented module, consider the case that $K$ is the $p$-adic completion of $\mathbb{Q}_p(p^{1/p^\infty})$, $p\neq 2$. Consider the extension $L=K(p^{1/2})$. Then $L^{\circ a}$ is an almost finitely presented $K^{\circ a}$-module. Indeed, for any $n\geq 1$, we have injective maps
\[
K^\circ\oplus p^{1/2p^n} K^\circ\rightarrow L^\circ
\]
whose cokernel is killed by $p^{1/2p^n}$. In fact, in this example $L^{\circ a}$ is even uniformly almost finitely generated.
\end{example}

\begin{prop}[{\cite[Proposition 2.4.18]{GabberRamero}}] Let $A$ be a $K^{\circ a}$-algebra. Then an $A$-module $M$ is flat and almost finitely presented if and only if it is almost projective and almost finitely generated.
\end{prop}

By abuse of notation, we call such $A$-modules $M$ finite projective in the following, a terminology not used in \cite{GabberRamero}. If additionally, $M$ is uniformly almost finitely generated, we say that $M$ is uniformly finite projective.

For uniformly finite projective modules, there is a good notion of rank.

\begin{thm}[{\cite[Proposition 4.3.27, Remark 4.3.10 (i)]{GabberRamero}}] Let $A$ be a $K^{\circ a}$-algebra, and let $M$ be a uniformly finite projective $A$-module. Then there is a unique decomposition $A=A_0\times A_1\times \cdots \times A_k$ such that for each $i=0,\ldots,k$, the $A_i$-module $M_i=M\otimes_A A_i$ has the property that $\bigwedge^i M_i$ is invertible, and $\bigwedge^{i+1} M_i=0$. Here, an $A$-module $L$ is called invertible if $L\otimes_A \alHom_A(L,A) = A$.
\end{thm}

Finally, we need the notion of \'{e}tale morphisms.

\begin{definition} Let $A$ be a $K^{\circ a}$-algebra, and let $B$ be an $A$-algebra. Let $\mu: B\otimes_A B\rightarrow B$ denote the multiplication morphism.
\begin{altenumerate}
\item[{\rm (i)}] The morphism $A\rightarrow B$ is said to be unramified if there is some element $e\in (B\otimes_A B)_\ast$ such that $e^2=e$, $\mu(e) = 1$ and $xe=0$ for all $x\in \ker(\mu)_\ast$.
\item[{\rm (ii)}] The morphism $A\rightarrow B$ is said to be \'{e}tale if it is unramified and $B$ is a flat $A$-module.
\end{altenumerate}
\end{definition}

We note that the definition of unramified morphisms basically says that the diagonal morphism $\mu: B\otimes_A B\rightarrow B$ is a closed immersion in the geometric picture.

In the following, we will be particularly interested in almost finitely presented \'{e}tale maps.

\begin{definition} A morphism $A\rightarrow B$ of $K^{\circ a}$-algebras is said to be finite \'{e}tale if it is \'{e}tale and $B$ is an almost finitely presented $A$-module. Write $A_\fet$ for the category of finite \'{e}tale $A$-algebras.
\end{definition}

We note that in this case $B$ is a finite projective $A$-module. Also, this terminology is not used in \cite{GabberRamero}, but we feel that it is the appropriate almost analogue of finite \'{e}tale covers.

There is an equivalent characterization of finite \'{e}tale morphisms in terms of trace morphisms. If $A$ is any $K^{\circ a}$-algebra, and $P$ is some finite projective $A$-module, we define $P^\ast = \alHom(P,A)$, which is a finite projective $A$-module again. Moreover, $P^{\ast\ast} \cong P$ canonically, and there is an isomorphism
\[
\End(P)^a = P\otimes_A P^\ast\ .
\]
In particular, one gets a trace morphism $\tr_{P/A}: \End(P)^a\rightarrow A$.

\begin{definition} Let $A$ be a $K^{\circ a}$-algebra, and let $B$ be an $A$-algebra such that $B$ is a finite projective $A$-module. Then we define the trace form as the bilinear form
\[
t_{B/A}: B\otimes_A B\rightarrow A
\]
given by the composition of $\mu: B\otimes_A B\rightarrow B$ and the map $B\rightarrow A$ sending any $b\in B$ to the trace of the endomorphism $b^\prime\mapsto bb^\prime$ of $B$.
\end{definition}

\begin{rem} We should remark that the latter definition does not literally make sense, as one can not talk about an element $b$ of some almost object $B$: There is no underlying set. However, one can define a map $B_\ast \rightarrow \End_{A_\ast}(B_\ast)$ in the way described, and we are considering the corresponding map of almost objects $B\rightarrow \End_{A_\ast}(B_\ast)^a = \End_A(B)^a$.
\end{rem}

\begin{thm}[{\cite[Theorem 4.1.14]{GabberRamero}}] In the situation of the definition, the morphism $A\rightarrow B$ is finite \'{e}tale if and only if the trace map is a perfect pairing, i.e. induces an isomorphism $B\cong B^\ast$.
\end{thm}

An important property is that finite \'{e}tale covers lift uniquely over nilpotents.

\begin{thm}\label{LiftingAlmostEtaleCovers} Let $A$ be a $K^{\circ a}$-algebra. Assume that $A$ is flat over $K^{\circ a}$ and $\varpi$-adically complete, i.e.
\[
A\cong \varprojlim A/\varpi^n\ .
\]
Then the functor $B\mapsto B\otimes_A A/\varpi$ induces an equivalence of categories $A_\fet\cong (A/\varpi)_\fet$. Any $B\in A_\fet$ is again flat over $K^{\circ a}$ and $\varpi$-adically complete. Moreover, $B$ is a uniformly finite projective $A$-module if and only if $B\otimes_A A/\varpi$ is a uniformly finite projective $A/\varpi$-module.
\end{thm}

\begin{proof} The first part follows from \cite{GabberRamero}, Theorem 5.3.27. The rest is easy.
\end{proof}

Recall that we wanted to prove the string of equivalences
\[
K_\fet\cong K^{\circ a}_\fet\cong (K^{\circ a}/\varpi)_\fet = (K^{\flat \circ a}/\varpi^\flat)_\fet\cong K^{\flat \circ a}_\fet\cong K^\flat_\fet\ .
\]
The identification in the middle is tautological as $K^{\circ}/\varpi = K^{\flat \circ}/\varpi^\flat$, and the corresponding almost settings agree. The previous theorem shows that the inner two functors are equivalences. For the other two equivalences, we feel that it is more convenient to study them in the more general setup of perfectoid $K$-algebras.

\section{Perfectoid algebras}

Fix a perfectoid field $K$.

\begin{definition}
\begin{altenumerate}
\item[{\rm (i)}] A perfectoid $K$-algebra is a Banach $K$-algebra $R$ such that the subset $R^\circ\subset R$ of powerbounded elements is open and bounded, and the Frobenius morphism $\Phi: R^\circ/\varpi\rightarrow R^\circ/\varpi$ is surjective. Morphisms between perfectoid $K$-algebras are the continuous morphisms of $K$-algebras.
\item[{\rm (ii)}] A perfectoid $K^{\circ a}$-algebra is a $\varpi$-adically complete flat $K^{\circ a}$-algebra $A$ on which Frobenius induces an isomorphism
\[
\Phi: A/\varpi^{\frac 1p}\cong A/\varpi\ .
\]
Morphisms between perfectoid $K^{\circ a}$-algebras are the morphisms of $K^{\circ a}$-algebras.
\item[{\rm (iii)}] A perfectoid $K^{\circ a}/\varpi$-algebra is a flat $K^{\circ a}/\varpi$-algebra $\overline{A}$ on which Frobenius induces an isomorphism
\[
\Phi: \overline{A}/\varpi^{\frac 1p}\cong \overline{A}\ .
\]
Morphisms are the morphisms of $K^{\circ a}/\varpi$-algebras.
\end{altenumerate}
\end{definition}

Let $K-\Perf$ denote the category of perfectoid $K$-algebras, and similarly for $K^{\circ a}-\Perf$, ... . Let $K^\flat$ be the tilt of $K$. Then the main theorem of this section is the following.

\begin{thm}\label{TiltingEquivalence} The categories of perfectoid $K$-algebras and perfectoid $K^\flat$-algebras are equivalent. In fact, we have the following series of equivalences of categories.
\[
K-\Perf\cong K^{\circ a}-\Perf\cong (K^{\circ a}/\varpi)-\Perf = (K^{\flat \circ a}/\varpi^\flat)-\Perf\cong K^{\flat \circ a}-\Perf\cong K^\flat-\Perf
\]
\end{thm}

In other words, a perfectoid $K$-algebra, which is an object over the generic fibre, has a canonical extension to the almost integral level as a perfectoid $K^{\circ a}$-algebra, and perfectoid $K^{\circ a}$-algebras are determined by their reduction modulo $\varpi$.

The following lemma expresses the conditions imposed on a perfectoid $K^{\circ a}$-algebra in terms of classical commutative algebra.

\begin{lem}\label{AlmostVsUsual}\begin{altenumerate} Let $M$ be a $K^{\circ a}$-module.
\item[{\rm (i)}] The module $M$ is flat over $K^{\circ a}$ if and only if $M_\ast$ is flat over $K^\circ$ if and only if $M_\ast$ has no $\varpi$-torsion.
\item[{\rm (ii)}] If $N$ is a flat $K^\circ$-module and $M=N^a$, then $M$ is flat over $K^{\circ a}$ and we have $M_\ast = \{x\in N[\frac 1{\varpi}]\mid \forall \epsilon\in \mathfrak{m}: \epsilon x\in N\}$.
\item[{\rm (iii)}] If $M$ is flat over $K^{\circ a}$, then for all $x\in K^\circ$, we have $(xM)_\ast = xM_\ast$. Moreover, $M_\ast/xM_\ast\subset (M/xM)_\ast$, and for all $\epsilon\in \mathfrak{m}$ the image of $(M/x\epsilon M)_\ast$ in $(M/xM)_\ast$ is equal to $M_\ast/xM_\ast$.
\item[{\rm (iv)}] If $M$ is flat over $K^{\circ a}$, then $M$ is $\varpi$-adically complete if and only if $M_\ast$ is $\varpi$-adically complete.
\end{altenumerate}
\end{lem}

\begin{rem} The non-surjectivity in (iii) is due to elements as in Remark \ref{StrangeSums}.
\end{rem}

\begin{proof}\begin{altenumerate}
\item[{\rm (i)}] By definition, $M$ is a flat $K^{\circ a}$-module if and only if all $\mathrm{Tor}_i^{K^\circ}(M_\ast,N)$ are almost zero for all $i>0$ and all $K^\circ$-modules $N$. Hence if $M_\ast$ is a flat $K^\circ$-module, then $M$ is a flat $K^{\circ a}$-module. Conversely, choosing $N=K^\circ/\varpi$ and $i=1$, we find that the kernel of multiplication by $\varpi$ on $M_\ast$ is almost zero. But
\[
M_\ast=\mathrm{Hom}_{K^{\circ a}}(K^{\circ a},M) = \mathrm{Hom}_{K^\circ}(\mathfrak{m},M_\ast)
\]
does not have nontrivial almost zero elements, hence has no $\varpi$-torsion. But a $K^\circ$-module $N$ is flat if and only if it has no $\varpi$-torsion.

\item[{\rm (ii)}] We have
\[
M_\ast=\mathrm{Hom}_{K^{\circ a}}(K^{\circ a},M) = \mathrm{Hom}_{K^\circ}(\mathfrak{m},N)\ .
\]
As $N$ is flat over $K^\circ$, we can write the last term as the subset of those $x\in \mathrm{Hom}_K(K,N[\frac 1{\varpi}])=N[\frac 1{\varpi}]$ satisfying the condition that for all $\epsilon\in \mathfrak{m}$, we have $\epsilon x\in N$.

\item[{\rm (iii)}] Note that $(xM_\ast)^a = xM$, and $xM_\ast$ is a flat $K^\circ$-module. Hence
\[
(xM)_\ast = \mathrm{Hom}(\mathfrak{m},xM_\ast) = \{y\in M_\ast[\frac 1{\varpi}]\mid \forall \epsilon\in \mathfrak{m} : \epsilon y\in xM_\ast\} = xM_\ast\ .
\]
Now using that $_\ast$ is left-exact (since right-adjoint to $M\mapsto M^a$), we get the inclusion $M_\ast/xM_\ast\subset (M/xM)_\ast$. If $m\in (M/xM)_\ast$ lifts to $\tilde{m}\in (M/x\epsilon M)_\ast$, then evaluate $\tilde{m}\in \mathrm{Hom}(\mathfrak{m},M_\ast/x\epsilon M_\ast)$ on $\epsilon$. This gives an element $n=\tilde{m}(\epsilon)\in M_\ast/x\epsilon M_\ast$, which we lift to $\tilde{n}\in M_\ast$. One checks that $\tilde{n}$ is divisible by $\epsilon$: It suffices to check that $\delta \tilde{n}$ is divisible by $\epsilon$ for any $\delta\in \mathfrak{m}$. But $\delta n=\delta \tilde{m}(\epsilon) = \epsilon \tilde{m}(\delta)$ lies in $\epsilon M_\ast/x\epsilon M_\ast$, hence $\delta \tilde{n}\in \epsilon M_\ast$.

Then $m_1=\frac{\tilde{n}}{\epsilon}\in M_\ast$ is the desired lift of $m\in (M/xM)_\ast$: Multiplication by $\epsilon$ induces an injection $(M/xM)_\ast\rightarrow (M/x\epsilon M)_\ast$, because $_\ast$ is left-exact, and the images agree: $\tilde{n}$ maps to $\epsilon m = \tilde{m}(\epsilon) = n$ in $(M/x\epsilon M)_\ast$.

\item[{\rm (iv)}] The functors $M\mapsto M_\ast$ and $N\mapsto N^a$ between the category of $K^{\circ a}$-modules and the category of $K^\circ$-modules admit left adjoints, given by $N\mapsto N^a$ and $M\mapsto M_!=\mathfrak{m}\otimes M_\ast$, respectively, and hence commute with inverse limits. Now if $M$ is $\varpi$-adically complete, then
\[
M_\ast = (\varprojlim M/\varpi^n M)_\ast = \varprojlim (M/\varpi^n M)_\ast = \varprojlim M_\ast/\varpi^n M_\ast\ ,
\]
using part (iii) in the last equality, hence $M_\ast$ is $\varpi$-adically complete. Conversely, if $M_\ast$ is $\varpi$-adically complete, then
\[
M = (M_\ast)^a = (\varprojlim M_\ast/\varpi^n M_\ast)^a = \varprojlim (M_\ast/\varpi^n M_\ast)^a = \varprojlim M/\varpi^n M\ .
\]
\end{altenumerate}
\end{proof}

\begin{prop}\label{PropKvsKAlm} Let $R$ be a perfectoid $K$-algebra. Then $\Phi$ induces an isomorphism $R^\circ/\varpi^{1/p}\cong R^\circ/\varpi$, and $A=R^{\circ a}$ is a perfectoid $K^{\circ a}$-algebra.
\end{prop}

\begin{proof} By assumption, $\Phi$ is surjective. Injectivity is clear: If $x\in R^\circ$ is such that $x^p/\varpi$ is powerbounded, then $x/\varpi^{1/p}$ is powerbounded. Obviously, $R^\circ$ is $\varpi$-adically complete and flat over $K^\circ$; now the previous lemma shows that $R^{\circ a}$ is $\varpi$-adically complete and flat over $K^{\circ a}$.
\end{proof}

\begin{lem}\label{FrobIsom} Let $A$ be a perfectoid $K^{\circ a}$-algebra, and let $R=A_{\ast}[\varpi^{-1}]$. Equip $R$ with the Banach $K$-algebra structure making $A_\ast$ open and bounded. Then $A_{\ast}=R^\circ$ is the set of power-bounded elements, $R$ is perfectoid, and
\[
\Phi: A_\ast/\varpi^{1/p}\cong A_\ast/\varpi\ .
\]
\end{lem}

\begin{proof} By definition, $\Phi$ is an isomorphism $A/\varpi^{1/p}\cong A/\varpi$, hence $\Phi$ is an almost isomorphism $A_\ast/\varpi^{1/p}\rightarrow A_\ast/\varpi$. It is injective: If $x\in A_\ast$ and $x^p\in \varpi A_\ast$, then for all $\epsilon\in \mathfrak{m}$, $\epsilon x\in \varpi^{1/p} A_\ast$ by almost injectivity, hence $x\in (\varpi^{1/p} A)_\ast = \varpi^{1/p} A_\ast$.

\begin{lem}\label{PthRootNormal} Assume that $x\in R$ satisfies $x^p\in A_\ast$. Then $x\in A_\ast$.
\end{lem}

\begin{proof} Injectivity of $\Phi$ says that if $y\in A_\ast$ satisfies $y^p\in \varpi A_\ast$, then $y\in \varpi^{\frac 1p} A_\ast$. There is some positive integer $k$ such that $y=\varpi^{\frac kp}x\in A_\ast$, and as long as $k\geq 1$, $y^p\in \varpi A_\ast$, so that $y\in \varpi^{\frac 1p}A_\ast$. Because $A_\ast$ has no $\varpi$-torsion, we get $\varpi^{\frac{k-1}p}x\in A_\ast$. By induction, we get the result.
\end{proof}

Obviously, $A_{\ast}$ consists of power-bounded elements. Now assume that $x\in R$ is power-bounded. Then $\epsilon x$ is topologically nilpotent for all $\epsilon\in \mathfrak{m}$. In particular, $(\epsilon x)^{p^N}\in A_\ast$ for $N$ sufficiently large. By the last lemma, this implies $\epsilon x\in A_\ast$. This is true for all $\epsilon\in \mathfrak{m}$, so that by Lemma \ref{AlmostVsUsual} (ii), we have $x\in A_{\ast}$.

Next, $\Phi$ is surjective: It is almost surjective, hence it suffices to show that the composition $A_\ast/\varpi^{1/p}\rightarrow A_\ast/\varpi\rightarrow A_\ast/\mathfrak{m}$ is surjective. Let $x\in A_\ast$. By almost surjectivity, $\varpi^{1/p} x \equiv y^p$ modulo $\varpi A_\ast$, for some $y\in A_\ast$. Let $z=\frac y{\varpi^{1/p^2}}$. This implies $z^p\equiv x$ modulo $\varpi^{(p-1)/p} A_\ast$, in particular $z^p\in A_\ast$. By the lemma, also $z\in A_\ast$. As $x\equiv z^p$ modulo $\varpi^{(p-1)/p} A_\ast$, in particular modulo $\mathfrak{m} A_\ast$, this gives the desired surjectivity.

Finally, we see that $R$ is Banach $K$-algebra such that $R^\circ = A_\ast$ is open and bounded, and such that $\Phi$ is surjective on $R^\circ/\varpi = A_\ast/\varpi$. This means that $R$ is perfectoid, as desired.
\end{proof}

In particular, we get the desired equivalence $K^{\circ a}-\Perf\cong K-\Perf$. Let us note some further propositions.

\begin{prop}\label{PerfectoidReduced} Let $R$ be a perfectoid $K$-algebra. Then $R$ is reduced.
\end{prop}

\begin{proof} Assume $0\neq x\in R$ is nilpotent. Then $Kx\subset R^\circ$, contradicting the condition that $R^\circ$ is bounded.
\end{proof}

If $K$ has characteristic $p$, being perfectoid is basically the same as being perfect.

\begin{prop} Let $K$ be of characteristic $p$, and let $R$ be a Banach $K$-algebra such that the set of powerbounded elements $R^\circ\subset R$ is open and bounded. Then $R$ is perfectoid if and only if $R$ is perfect.
\end{prop}

\begin{proof} Assume $R$ is perfect. Then also $R^\circ$ is perfect, as an element $x$ is powerbounded if and only if $x^p$ is powerbounded. In particular, $\Phi: R^\circ/\varpi\rightarrow R^\circ/\varpi$ is surjective.

Now assume that $R$ is perfectoid, hence by Proposition \ref{PropKvsKAlm}, $\Phi$ induces an isomorphism $R^\circ/\varpi^{\frac 1p}\cong R^\circ/\varpi$. By successive approximation, we see that $R^\circ$ is perfect, and then that $R$ is perfect.
\end{proof}

In order to finish the proof of Theorem \ref{TiltingEquivalence}, it suffices to prove the following result.

\begin{thm}\label{DeformPerfectoid} The functor $A\mapsto \overline{A} = A/\varpi$ induces an equivalence of categories $K^{\circ a}-\Perf\cong (K^{\circ a}/\varpi)-\Perf$.
\end{thm}

In other words, we have to prove that a perfectoid $K^{\circ a}/\varpi$-algebra admits a unique deformation to $K^{\circ a}$. For this, we will use the theory of the cotangent complex. Let us briefly recall it here.

In classical commutative algebra, the definition of the cotangent complex is due to Quillen, \cite{Quillen}, and its theory was globalized on toposes and applied to deformation problems by Illusie, \cite{IllusieCotangent}, \cite{IllusieCotangent2}. To any morphism $R\rightarrow S$ of rings, one associates a complex $\mathbb{L}_{S/R}\in D^{\leq 0}(S)$, where $D(S)$ is the derived category of the category of $S$-modules, and $D^{\leq 0}(S)\subset D(S)$ denotes the full subcategory of objects which have trivial cohomology in positive degrees. The cohomology in degree $0$ of $\mathbb{L}_{S/R}$ is given by $\Omega^1_{S/R}$, and for any morphisms $R\rightarrow S\rightarrow T$ of rings, there is a triangle in $D(T)$:
\[
T\otimes^{\mathbb{L}}_S \mathbb{L}_{S/R}\rightarrow \mathbb{L}_{T/R}\rightarrow \mathbb{L}_{T/S}\rightarrow\ ,
\]
extending the short exact sequence
\[
T\otimes_S \Omega^1_{S/R}\rightarrow \Omega^1_{T/R}\rightarrow \Omega^1_{T/S}\rightarrow 0\ .
\]

Let us briefly recall the construction. First, one uses the Dold-Kan equivalence to reinterpret $D^{\leq 0}(S)$ as the category of simplicial $S$-modules modulo weak equivalence. Now one takes a simplicial resolution $S_\bullet$ of the $R$-algebra $S$ by free $R$-algebras. Then one defines $\mathbb{L}_{S/R}$ as the object of $D^{\leq 0}(S)$ associated to the simplicial $S$-module $\Omega^1_{S_\bullet/R}\otimes_{S_\bullet} S$.

Just as under certain favorable assumptions, one can describe many deformation problems in terms of tangent or normal bundles, it turns out that in complete generality, one can describe them via the cotangent complex. In special cases, this gives back the classical results, as e.g. if $R\rightarrow S$ is a smooth morphism, then $\mathbb{L}_{S/R}$ is concentrated in degree $0$, and is given by the cotangent bundle.

Specifically, we will need the following results. Fix some ring $R$ with an ideal $I\subset R$ such that $I^2=0$. Moreover, fix a flat $R_0=R/I$-algebra $S_0$. We are interested in the obstruction towards deforming $S_0$ to a flat $R$-algebra $S$.

\begin{thm}[{\cite[III.2.1.2.3]{IllusieCotangent}, \cite[Proposition 3.2.9]{GabberRamero}}] There is an obstruction class in $\Ext^2(\mathbb{L}_{S_0/R_0},S_0\otimes_{R_0} I)$ which vanishes precisely when there exists a flat $R$-algebra $S$ such that $S\otimes_R R_0 = S_0$. If there exists such a deformation, then the set of all isomorphism classes of such deformations forms a torsor under $\Ext^1(\mathbb{L}_{S_0/R_0},S_0\otimes_{R_0} I)$, and every deformation has automorphism group $\Hom(\mathbb{L}_{S_0/R_0},S_0\otimes_{R_0} I)$.
\end{thm}

Here, a deformation comes with the isomorphism $S\otimes_R R_0\cong S_0$, and isomorphisms of deformations are required to act trivially on $S\otimes_R R_0=S_0$.

Now assume that one has two flat $R$-algebras $S$, $S^\prime$ with reduction $S_0$, $S_0^\prime$ to $R_0$, and a morphism $f_0: S_0\rightarrow S_0^\prime$. We are interested in the obstruction to lifting $f_0$ to a morphism $f: S\rightarrow S^\prime$.

\begin{thm}[{\cite[III.2.2.2]{IllusieCotangent}, \cite[Proposition 3.2.16]{GabberRamero}}] There is an obstruction class in $\Ext^1(\mathbb{L}_{S_0/R_0},S_0^\prime\otimes_{R_0} I)$ which vanishes precisely when there exists an extension of $f_0$ to $f: S\rightarrow S^\prime$. If there exists such a lift, then the set of all lifts forms a torsor under $\Hom(\mathbb{L}_{S_0/R_0},S_0^\prime\otimes_{R_0} I)$.
\end{thm}

We will need the following criterion for the vanishing of the cotangent complex. This appears as Lemma 6.5.13 i) in \cite{GabberRamero}.

\begin{prop}\label{PerfectTrivialCotangent}
\begin{altenumerate}
\item[{\rm (i)}] Let $R$ be a perfect $\mathbb{F}_p$-algebra. Then $\mathbb{L}_{R/\mathbb{F}_p}\cong 0$.
\item[{\rm (ii)}] Let $R\rightarrow S$ be a morphism of $\mathbb{F}_p$-algebras. Let $R_{(\Phi)}$ be the ring $R$ with the $R$-algebra structure via $\Phi: R\rightarrow R$, and define $S_{(\Phi)}$ similarly. Assume that the relative Frobenius $\Phi_{S/R}$ induces an isomorphism
\[
R_{(\Phi)}\otimes^{\mathbb{L}}_R S\rightarrow S_{(\Phi)}
\]
in $D(R)$. Then $\mathbb{L}_{S/R}\cong 0$.
\end{altenumerate}
\end{prop}

\begin{rem} Of course, (i) is a special case of (ii), and we will only need part (ii). However, we feel that (i) is an interesting statement that does not seem to be very well-known. It allows one to define the ring of Witt vectors $W(R)$ of $R$ simply by saying that it is the unique deformation of $R$ to a flat $p$-adically complete $\mathbb{Z}_p$-algebra. Also note that it is clear that $\Omega^1_{R/\mathbb{F}_p}=0$ in part (i): Any $x\in R$ can be written as $y^p$, and then $dx=dy^p=py^{p-1} dy=0$. This identity is at the heart of this proposition.
\end{rem}

\begin{proof} We sketch the proof of part (ii), cf. proof of Lemma 6.5.13 i) in \cite{GabberRamero}. Let $S^\bullet$ be a simplicial resolution of $S$ by free $R$-algebras. We have the relative Frobenius map
\[
\Phi_{S^\bullet/R}: R_{(\Phi)}\otimes_R S^\bullet\rightarrow S^\bullet_{(\Phi)}\ .
\]
Note that identifying $S^k$ with a polynomial algebra $R[X_1,X_2,\ldots]$, the relative Frobenius map $\Phi_{S^k/R}$ is given by the $R_{(\Phi)}$-algebra map sending $X_i\mapsto X_i^p$.

The assumption says that $\Phi_{S^\bullet/R}$ induces a quasiisomorphism of simplicial $R_{(\Phi)}$-algebras. This implies that $\Phi_{S^\bullet/R}$ gives an isomorphism
\[
R_{(\Phi)}\otimes^{\mathbb{L}}_R \mathbb{L}_{S/R}\cong \mathbb{L}_{S_{(\Phi)}/R_{(\Phi)}}\ .
\]
On the other hand, the explicit description shows that the map induced by $\Phi_{S^k/R}$ on differentials will map $dX_i$ to $dX_i^p = 0$, and hence is the zero map. This shows that $\mathbb{L}_{S_{(\Phi)}/R_{(\Phi)}}\cong 0$, and we may identify this with $\mathbb{L}_{S/R}$. 
\end{proof}

In their book \cite{GabberRamero}, Gabber and Ramero generalize the theory of the cotangent complex to the almost context. Specifically, they show that if $R\rightarrow S$ is a morphism of $K^\circ$-algebras, then $\mathbb{L}_{S/R}^a$ as an element of $D(S^a)$, the derived category of $S^a$-modules, depends only the morphism $R^a\rightarrow S^a$ of almost $K^\circ$-algebras. This allows one to define $\mathbb{L}^a_{B/A}\in D^{\leq 0}(B)$ for any morphism $A\rightarrow B$ of $K^{\circ a}$-algebras. With this modification, the previous theorems stay true in the almost world without change.

\begin{rem} In fact, the cotangent complex $\mathbb{L}_{B/A}$ is defined as an object of a derived category of modules over an actual ring in \cite{GabberRamero}, but for our purposes it is enough to consider its almost version $\mathbb{L}^a_{B/A}$.
\end{rem}

\begin{cor} Let $\overline{A}$ be a perfectoid $K^{\circ a}/\varpi$-algebra. Then $\mathbb{L}^a_{\overline{A}/(K^{\circ a}/\varpi)}\cong 0$.
\end{cor}

\begin{proof} This follows from the almost version of Proposition \ref{PerfectTrivialCotangent}, which can be proved in the same way. Alternatively, argue with $B=(\overline{A}\times K^{\circ a}/\varpi)_{!!}$, which is a flat $K^\circ/\varpi$-algebra such that $B/\varpi^{1/p}\cong B$ via $\Phi$. Here, we use the functor $C\mapsto C_{!!}$ from \cite{GabberRamero}, \S 2.2.25.
\end{proof}

Now we can prove Theorem \ref{DeformPerfectoid}.

\begin{proof} ({\it of Theorem \ref{DeformPerfectoid}}) Let $\overline{A}$ be a perfectoid $K^{\circ a}/\varpi$-algebra. We see inductively that all obstructions and ambiguities in lifting inductively to a flat $(K^\circ/\varpi^n)^a$-algebra $\overline{A}_n$ vanish: All groups occuring can be expressed in terms of the cotangent complex by the theorems above, so that it suffices to show that $\mathbb{L}^a_{\overline{A}_n/(K^\circ/\varpi^n)^a}=0$. But by Theorem 2.5.36 of \cite{GabberRamero}, the short exact sequence
\[
0\rightarrow \overline{A}\buildrel \varpi^{n-1}\over \rightarrow \overline{A}_n\rightarrow \overline{A}_{n-1}\rightarrow 0
\]
induces after tensoring with $\mathbb{L}_{\overline{A}_n/(K^\circ/\varpi^n)^a}$ a triangle
\[
\mathbb{L}^a_{\overline{A}/(K^\circ/\varpi)^a}\rightarrow \mathbb{L}^a_{\overline{A}_n/(K^\circ/\varpi^n)^a}\rightarrow \mathbb{L}^a_{\overline{A}_{n-1}/(K^\circ/\varpi^{n-1})^a}\rightarrow \ ,
\]
and the claim follows by induction.

This gives a unique system of flat $(K^\circ/\varpi^n)^a$-algebras $\overline{A}_n$ with isomorphisms
\[
\overline{A}_n/\varpi^{n-1}\cong \overline{A}_{n-1}\ .
\]
Let $A$ be their inverse limit. Then $A$ is $\varpi$-adically complete with $A/\varpi^n A = \overline{A}_n$. This shows that $A$ is perfectoid, and we get an equivalence between perfectoid $K^{\circ a}$-algebras and perfectoid $K^{\circ a}/\varpi$-algebras, as desired.
\end{proof}

In particular, we also arrive at the tilting equivalence, $K-\Perf\cong K^\flat-\Perf$. We want to compare this with Fontaine's construction. Let $R$ be a perfectoid $K$-algebra, with $A=R^{\circ a}$. Define
\[
A^\flat = \varprojlim_\Phi A/\varpi\ ,
\]
which we regard as a $K^{\flat \circ a}$-algebra via
\[
K^{\flat \circ a} = (\varprojlim_\Phi K^\circ/\varpi)^a = \varprojlim_{\Phi} (K^\circ/\varpi)^a = \varprojlim_{\Phi} K^{\circ a}/\varpi\ ,
\]
and set $R^\flat = A^\flat_\ast[(\varpi^\flat)^{-1}]$.

\begin{prop} This defines a perfectoid $K^\flat$-algebra $R^\flat$ with corresponding perfectoid $K^{\flat \circ a}$-algebra $A^\flat$, and $R^\flat$ is the tilt of $R$. Moreover,
\[
R^\flat = \varprojlim_{x\mapsto x^p} R\ ,\ A^\flat_\ast = \varprojlim_{x\mapsto x^p} A_\ast\ ,\ A^\flat_\ast/\varpi^\flat\cong A_\ast/\varpi\ .
\]
In particular, we have a continuous multiplicative map $R^\flat\rightarrow R$, $x\mapsto x^\sharp$.
\end{prop}

\begin{rem} It follows that the tilting functor is independent of the choice of $\varpi$ and $\varpi^\flat$. We note that this explicit description comes from the fact that the lifting from perfectoid $K^{\flat \circ a}/\varpi^\flat$-algebras to perfectoid $K^{\flat \circ a}$-algebras can be made explicit by means of the inverse limit over the Frobenius.
\end{rem}

\begin{proof} First, we have
\[
A^\flat_\ast = (\varprojlim_\Phi A/\varpi)_\ast = \varprojlim_\Phi (A/\varpi)_\ast = \varprojlim_\Phi A_{\ast}/\varpi\ ,
\]
because $_\ast$ commutes with inverse limits and using Lemma \ref{AlmostVsUsual} (iii). Note that the image of $\Phi: (A/\varpi)_\ast\rightarrow (A/\varpi)_\ast$ is $A_\ast/\varpi$, because it factors over $(A/\varpi^{1/p})_\ast$, and the image of the projection $(A/\varpi)_\ast \rightarrow (A/\varpi^{1/p})_\ast$ is $A_\ast/\varpi^{1/p}$. But
\[
\varprojlim_\Phi A_{\ast}/\varpi = \varprojlim_{x\mapsto x^p} A_{\ast}\ ,
\]
as in the proof of Lemma \ref{DefTiltingField} (i).

This shows that $A^\flat_\ast$ is a $\varpi^\flat$-adically complete flat $K^{\flat \circ}$-algebra. Moreover, the projection $x\mapsto x^\sharp$ of $A^\flat_\ast$ onto the first component $x^\sharp\in A_\ast$ induces an isomorphism
\[
A^\flat_\ast/\varpi^\flat\cong A_\ast/\varpi\ ,
\]
because of Lemma \ref{FrobIsom}. Therefore $A^\flat$ is a perfectoid $K^{\flat \circ a}$-algebra.

To see that $R^\flat$ is the tilt of $R$, we go through all equivalences. Indeed, $R$ has corresponding perfectoid $K^{\circ a}$-algebra $A$, which reduces to the perfectoid $K^{\circ a}/\varpi$-algebra $A/\varpi$, which is the same as the perfectoid $K^{\flat \circ a}/\varpi^\flat$-algebra $A^\flat/\varpi^\flat$, which lifts to the perfectoid $K^{\flat \circ a}$-algebra $A^\flat$, which in turn gives rise to $R^\flat$.
\end{proof}

\begin{rem} In fact, one can write down the functors in both directions. From characteristic $0$ to characteristic $p$, we have already given the explicit functor. The converse functor is given by $R=W(R^{\flat \circ})\otimes_{W(K^{\flat \circ})} K$, using the usual map $\theta: W(K^{\flat \circ})\rightarrow K$. We leave it as an exercise to the reader to give a direct proof of the theorem via this description, cf. \cite{KedlayaLiu}. This avoids the use of almost mathematics in the proof of the tilting equivalence. We stress however that in our proof we never need to talk about big rings like $W(R^{\flat \circ})$, and that the point of view of the given proof will be useful in later arguments.
\end{rem}

Let us give a prototypical example for the tilting process.

\begin{prop}\label{ExampleTilting} Let
\[
R=K\langle T_1^{1/p^\infty},\ldots,T_n^{1/p^\infty}\rangle = \widehat{K^\circ[T_1^{1/p^\infty},\ldots,T_n^{1/p^\infty}]}[\varpi^{-1}]\ .
\]
Then $R$ is a perfectoid $K$-algebra, and its tilt $R^\flat$ is given by $K^\flat \langle T_1^{1/p^\infty},\ldots,T_n^{1/p^\infty}\rangle$.
\end{prop}

\begin{proof} One checks that $R^\circ = \widehat{K^\circ[T_1^{1/p^\infty},\ldots,T_n^{1/p^\infty}]}$, which is $\varpi$-adically complete and flat over $K^\circ$. Moreover, it reduces to $R^\circ/\varpi = K^\circ/\varpi[T_1^{1/p^\infty},\ldots,T_n^{1/p^\infty}]$, on which Frobenius is surjective. This shows that $R$ is a perfectoid $K$-algebra.

To see that its tilt has the desired form, we only have to check that $R^\circ/\varpi = R^{\flat \circ}/\varpi^\flat$, by the proof of the tilting equivalence. But this is obvious.
\end{proof}

We note that under the process of tilting, perfectoid fields are identified.

\begin{lem}\label{TiltingPerfectoidFields} Let $R$ be a perfectoid $K$-algebra with tilt $R^\flat$. Then $R$ is a perfectoid field if and only if $R^\flat$ is a perfectoid field.
\end{lem}

\begin{proof} Note that $R$ is a perfectoid field if and only if it is a nonarchimedean field, i.e. its topology is induced by a rank-$1$-valuation. This valuation is necessarily given by the spectral norm
\[
||x||_R = \inf \{ |t|^{-1}\mid t\in K^\times, tx\in R^\circ \}
\]
on $R$. It is easy to check that for $x\in R^\flat$, we have $||x||_{R^\flat} = ||x^\sharp||_R$. In particular, if $||\cdot||_R$ is multiplicative, then so is $||\cdot||_{R^\flat}$, i.e. if $R$ is a perfectoid field, then so is $R^\flat$.

Conversely, assume that $R^\flat$ is a perfectoid field. We have to check that the spectral norm $||\cdot||_R$ on $R$ is multiplicative. Let $x,y\in R$; after multiplication by elements of $K$, we may assume $x,y\in R^\circ$, but not in $\varpi^{1/p} R^\circ$. We want to see that $||x||_R ||y||_R = ||xy||_R$. But we can find $x^\flat,y^\flat \in R^{\flat \circ}$ with $x-(x^\flat)^\sharp,y-(y^\flat)^\sharp \in \varpi R^\circ$. Then $||x||_R=||x^\flat||_{R^\flat}$, $||y||_R=||y^\flat||_{R^\flat}$ and $||xy||_R = ||x^\flat y^\flat||_{R^\flat}$, and we get the claim.

To see that $R$ is a field, choose $x$ such that $x\in R^\circ$, but not in $\varpi R^\circ$, and take $x^\flat$ as before. Then by multiplicativity of $||\cdot||_R$, $||1-\frac x{(x^\flat)^\sharp}||_R<1$, and hence $\frac x{(x^\flat)^\sharp}$ is invertible, and then also $x$.
\end{proof}

Finally, let us discuss finite \'{e}tale covers of perfectoid algebras, and finish the proof of Theorem \ref{TiltingEquivFields}.

\begin{prop} Let $\overline{A}$ be a perfectoid $K^{\circ a}/\varpi$-algebra, and let $\overline{B}$ be a finite \'{e}tale $\overline{A}$-algebra. Then $\overline{B}$ is a perfectoid $K^{\circ a}/\varpi$-algebra.
\end{prop}

\begin{proof} Obviously, $\overline{B}$ is flat. The statement about Frobenius follows from Theorem 3.5.13 ii) of \cite{GabberRamero}.
\end{proof}

In particular, Theorem \ref{LiftingAlmostEtaleCovers} provides us with the following commutative diagram, where $R$, $A$, $\overline{A}$, $A^\flat$ and $R^\flat$ form a sequence of rings under the tilting procedure.
\[\xymatrix{
R_\fet & A_\fet \ar[l]\ar[d]\ar[r]^\cong & \overline{A}_\fet \ar[d] & A^\flat_\fet\ar[l]_\cong \ar[d]\ar[r] & R^\flat_\fet \\
K-\Perf & K^{\circ a}-\Perf \ar[l]_\cong \ar[r]^\cong & (K^{\circ a}/\varpi)-\Perf & K^{\flat \circ a}-\Perf \ar[l]_\cong \ar[r]^\cong & K^\flat-\Perf
}\]

It follows from this diagram that the functors $A_\fet\rightarrow R_\fet$ and $A^\flat_\fet\rightarrow R^\flat_\fet$ are fully faithful. A main theorem is that both of them are equivalences: This amounts to Faltings's almost purity theorem. At this point, we will prove this only in characteristic $p$.

\begin{prop}\label{AlmostPurityCharP} Let $K$ be of characteristic $p$, let $R$ be a perfectoid $K$-algebra, and let $S/R$ be finite \'{e}tale. Then $S$ is perfectoid and $S^{\circ a}$ is finite \'{e}tale over $R^{\circ a}$. Moreover, $S^{\circ a}$ is a uniformly finite projective $R^{\circ a}$-module.
\end{prop}

\begin{rem} We need to define the topology on $S$ here. Recall that if $A$ is any ring with $t\in A$ not a zero-divisor, then any finitely generated $A[t^{-1}]$-module $M$ carries a canonical topology, which gives any finitely generated $A$-submodule of $M$ the $t$-adic topology. Any morphism of finitely generated $A[t^{-1}]$-modules is continuous for this topology, cf. \cite{GabberRamero}, Definition 5.4.10 and 5.4.11. If $A$ is complete and $M$ is projective, then $M$ is complete, as one checks by writing $M$ as a direct summand of a finitely generated free $A$-module. In particular, if $R$ is a perfectoid $K$-algebra and $S/R$ a finite \'{e}tale cover, then $S$ has a canonical topology for which it is complete.
\end{rem}

\begin{proof} This follows from Theorem 3.5.28 of \cite{GabberRamero}. Let us recall the argument. Note that $S$ is a perfect Banach $K$-algebra. We claim that it is perfectoid. Let $S_0\subset S$ be some finitely generated $R^\circ$-subalgebra with $S_0\otimes K = S$. Let $S_0^\bot\subset S$ be defined as the set of all $x\in S$ such that $t_{S/R}(x,S_0)\subset R^\circ$, using the perfect trace form pairing
\[
t_{S/R}: S\otimes_R S\rightarrow R\ ;
\]
then $S_0$ and $S_0^\bot$ are open and bounded. Let $Y$ be the integral closure of $R^\circ$ in $S$. Then $S_0\subset Y\subset S_0^\bot$: Indeed, the elements of $S_0$ are clearly integral over $R^\circ$, and we have $t_{S/R}(Y,Y)\subset R^\circ$. It follows that $Y$ is open and bounded. As $S^{\circ a} = Y^a$, it follows that $S^\circ$ is open and bounded, as desired.

Next, we want to check that $S^{\circ a}$ is a uniformly finite projective $R^{\circ a}$-module. For this, it is enough to prove that there is some $n$ such that for any $\epsilon\in \mathfrak{m}$, there are maps $S^{\circ}\rightarrow R^{\circ n}$ and $R^{\circ n}\rightarrow S^{\circ}$ whose composite is multiplication by $\epsilon$.

Let $e\in S\otimes_R S$ be the idempotent showing that $S$ is unramified over $R$. Then $\varpi^N e$ is in the image of $S^{\circ}\otimes_{R^{\circ}} S^{\circ}$ in $S\otimes_R S$ for some $N$. Write $\varpi^N e = \sum_{i=1}^n x_i\otimes y_i$. As Frobenius is bijective, we have $\varpi^{N/p^m} e = \sum_{i=1}^n x_i^{1/p^m}\otimes y_i^{1/p^m}$ for all $m$. In particular, for any $\epsilon\in \mathfrak{m}$, we can write $\epsilon e = \sum_{i=1}^n a_i\otimes b_i$ for certain $a_i,b_i\in S^{\circ}$, depending on $\epsilon$.

We get the map $S^\circ\rightarrow R^{\circ n}$,
\[
s\mapsto (t_{S/R}(s,b_1),\ldots,t_{S/R}(s,b_n))\ ,
\]
and the map $R^{\circ n}\rightarrow S^\circ$,
\[
(r_1,\ldots,r_n)\mapsto \sum_{i=1}^n a_i r_i\ .
\]
One easily checks that their composite is multiplication by $\epsilon$, giving the claim.

It remains to see that $S^{\circ a}$ is an unramified $R^{\circ a}$-algebra. But this follows from the previous arguments, which show that $e$ defines an almost element of $S^{\circ a}\otimes_{R^{\circ a}} S^{\circ a}$ with the desired properties.
\end{proof}

It follows that the diagram above extends as follows.
\[\xymatrix{
R_\fet & A_\fet \ar[l]\ar[d]\ar[r]^\cong & \overline{A}_\fet \ar[d] & A^\flat_\fet\ar[l]_\cong \ar[d]\ar[r]^\cong & R^\flat_\fet \ar[d] \\
K-\Perf & K^{\circ a}-\Perf \ar[l]_\cong \ar[r]^\cong & (K^{\circ a}/\varpi)-\Perf & K^{\flat \circ a}-\Perf \ar[l]_\cong \ar[r]^\cong & K^\flat-\Perf
}\]
Moreover, using Theorem \ref{LiftingAlmostEtaleCovers}, it follows that all finite \'{e}tale algebras over $A$, $\overline{A}$ or $A^\flat$ are uniformly almost finitely presented. Let us summarize the discussion.

\begin{thm}\label{AlmostFiniteEtaleCovers} Let $R$ be a perfectoid $K$-algebra with tilt $R^\flat$. There is a fully faithful functor from $R^\flat_\fet$ to $R_\fet$ inverse to the tilting functor. The essential image of this functor consists of the finite \'{e}tale covers $S$ of $R$, for which $S$ (with its natural topology) is perfectoid and $S^{\circ a}$ is finite \'{e}tale over $R^{\circ a}$. In this case, $S^{\circ a}$ is a uniformly finite projective $R^{\circ a}$-module.
\end{thm}

In particular, we see that the fully faithful functor $R^\flat_\fet\hookrightarrow R_\fet$ preserves degrees. We will later prove that this is an equivalence in general. For now, we prove that it is an equivalence for perfectoid fields, i.e. we finish the proof of Theorem \ref{TiltingEquivFields}.

\begin{proof} ({\it of Theorem \ref{TiltingEquivFields}}) Let $K$ be a perfectoid field with tilt $K^\flat$. Using the previous theorem, it is enough to show that the fully faithful functor $K^\flat_\fet\rightarrow K_\fet$ is an equivalence.

{\it Proof using ramification theory.} Proposition 6.6.2 (cf. its proof) and Proposition 6.6.6 of \cite{GabberRamero} show that for any finite extension $L$ of $K$, the extension $L^{\circ a}/K^{\circ a}$ is \'{e}tale. Moreover, it is finite projective by Proposition 6.3.6 of \cite{GabberRamero}, giving the desired result.

{\it Proof reducing to the case where $K^\flat$ is algebraically closed.} Let $M=\widehat{\overline{K^\flat}}$ be the completion of an algebraic closure of $K^\flat$. Clearly, $M$ is complete and perfect, i.e. $M$ is perfectoid. Let $M^\sharp$ be the untilt of $M$. Then by Lemma \ref{TiltingPerfectoidFields} and Proposition \ref{PerfectoidAlgebrClosed}, $M^\sharp$ is an algebraically closed perfectoid field containing $K$. Any finite extension $L\subset M$ of $K^\flat$ gives the untilt $L^\sharp\subset M^\sharp$, a finite extension of $K$. It is easy to see that the union $N=\bigcup_L L^\sharp\subset M^\sharp$ is a dense subfield. Now Krasner's lemma implies that $N$ is algebraically closed. Hence any finite extension $F$ of $K$ is contained in $N$; this means that there is some Galois extension $L$ of $K^\flat$ such that $F$ is contained in $L^\sharp$. Note that $L^\sharp$ is still Galois, as the functor $L\mapsto L^\sharp$ preserves degrees and automorphisms. In particular, $F$ is given by some subgroup $H$ of $\Gal(L^\sharp/K)=\Gal(L/K^\flat)$, which gives the desired finite extension $F^\flat=L^H$ of $K^\flat$ that untilts to $F$: The equivalence of categories shows that $(F^\flat)^\sharp\subset (L^\sharp)^H=F$, and as they have the same degree, they are equal.
\end{proof}

\section{Perfectoid spaces: Analytic topology}

In the following, we are interested in the adic spaces associated to perfectoid algebras. Specifically, note that perfectoid $K$-algebras are Tate, and we will look at the following type of affinoid $K$-algebras.

\begin{definition} A perfectoid affinoid $K$-algebra is an affinoid $K$-algebra $(R,R^+)$ such that $R$ is a perfectoid $K$-algebra.
\end{definition}

We note that in this case $\mathfrak{m} R^\circ\subset R^+\subset R^\circ$, because all topologically nilpotent elements lie in $R^+$, as $R^+$ is integrally closed. In particular, $R^+$ is almost equal to $R^\circ$.

\begin{lem}\label{TiltPerfectoidAffinoid} The categories of perfectoid affinoid $K$-algebras and perfectoid affinoid $K^\flat$-algebras are equivalent. If $(R,R^+)$ maps to $(R^\flat, R^{\flat +})$ under this equivalence, then $x\mapsto x^\sharp$ induces an isomorphism $R^{\flat +}/\varpi^\flat \cong R^+/\varpi$. Also $R^{\flat +} = \varprojlim_{x\mapsto x^p} R^+$.
\end{lem}

\begin{proof} Giving an open integrally closed subring of $R^\circ$ is equivalent to giving an integrally closed subring of $R^\circ/\mathfrak{m}$. This description is compatible with tilting. One easily checks the last identities.
\end{proof}

It turns out that also in this case, the presheaf $\mathcal{O}_X$ is a sheaf. In fact, the main theorem of this section is the following.

\begin{thm}\label{MainThmAnalytic} Let $(R,R^+)$  be a perfectoid affinoid $K$-algebra, and let $X=\Spa(R,R^+)$ with associated presheaves $\mathcal{O}_X$, $\mathcal{O}_X^+$. Also, let $(R^\flat, R^{\flat +})$ be the tilt given by Lemma \ref{TiltPerfectoidAffinoid}, and let $X^ \flat =\Spa(R^\flat,R^{\flat +})$ etc. .
\begin{altenumerate}
\item[{\rm (i)}] We have a homeomorphism $X\cong X^\flat$, given by mapping $x\in X$ to the valuation $x^ \flat \in X^ \flat $ defined by $|f(x^ \flat)| = |f^\sharp(x)|$. This homeomorphism identifies rational subsets.
\item[{\rm (ii)}] For any rational subset $U\subset X$ with tilt $U^ \flat\subset X^ \flat$, the complete affinoid $K$-algebra $(\mathcal{O}_X(U),\mathcal{O}_X^+(U))$ is perfectoid, with tilt $(\mathcal{O}_{X^\flat}(U^ \flat),\mathcal{O}_{X^ \flat}^+(U^ \flat))$.
\item[{\rm (iii)}] The presheaves $\mathcal{O}_X$, $\mathcal{O}_{X^\flat}$ are sheaves.
\item[{\rm (iv)}] The cohomology group $H^i(X,\mathcal{O}_X^+)$ is $\mathfrak{m}$-torsion for $i>0$.
\end{altenumerate}
\end{thm}

We remark that we did not assume that $R^+$ is a $K^\circ$-algebra, although this is satisfied in all examples
of interest to us. For this reason, it does not literally make sense to use the language of almost mathematics
in the context of $\mathcal{O}_X^+$. In the following, the reader may safely assume that $R^+$ is a
$K^\circ$-algebra, which avoids some small extra twists.

Let us give an outline of the proof. First, we show that the map $X\rightarrow X^\flat$ is continuous. Next, we prove a slightly weaker version of (ii), and give an explicit description of the perfectoid $K^{\circ a}$-algebra associated to $\mathcal{O}_X(U)$. This will be used to prove a crucial approximation lemma, dealing with the problem that the map $g\mapsto g^\sharp$ is far from being surjective. Nonetheless, it turns out that one can approximate any function $f\in R$ by a function of the form $g^\sharp$ such that the maps $x\mapsto |f(x)|$ and $x\mapsto |g^\sharp(x)|$ are identical except maybe at points $x$ where both of them are very small. It is then easy to deduce part (i), and also part (ii). We note that the same approximation lemma will be used later in the proof of the weight-monodromy conjecture for complete intersections.

It remains to prove that $\mathcal{O}_X$ is a sheaf with vanishing higher cohomology, and that the vanishing even extends to the almost integral level. The proof proceeds in several steps. First, we prove it in the case that $K$ is of characteristic $p$ and $(R,R^+)$ is the completed perfection of an affinoid $K$-algebra of tft. In that case, it is easy to deduce the result from Tate's acyclicity theorem. Again, the direct limit over the Frobenius extends the vanishing of cohomology to the almost integral level. Next, we do the general characteristic $p$ case by writing an arbitrary perfectoid affinoid $K$-algebra $(R,R^+)$ as the completed direct limit of algebras of the previous form. Finally, we deduce the case where $K$ has characteristic $0$ by using the result in characteristic $p$, making use of parts (i) and (ii) already proved.

\begin{proof} First, we check that the map $X\rightarrow X^\flat$ is well-defined and continuous: To check welldefinedness, we have to see that it maps valuations to valuations. This was already verified in the proof of Proposition \ref{TiltingValuationFields1}.

Moreover, the map $X\rightarrow X^\flat$ is continuous, because the preimage of the rational subset $U(\frac{f_1,\ldots,f_n}g)$ is given by $U(\frac{f_1^\sharp,\ldots,f_n^\sharp}{g^\sharp})$, assuming as in Remark \ref{RemRationalSubset} that $f_n$ is a power of $\varpi^\flat$ to ensure that $f_1^\sharp,\ldots,f_n^\sharp$ still generate $R$.

We have the following description of $\mathcal{O}_X$.

\begin{lem}\label{AlmostDescriptionSheaf} Let $U = U(\frac{f_1,\ldots,f_n}g)\subset \Spa(R^\flat,R^{\flat +})$ be rational, with preimage $U^\sharp\subset \Spa(R,R^+)$. Assume that all $f_i,g\in R^{\flat \circ}$ and that $f_n = \varpi^{\flat N}$ for some $N$; this is always possible without changing the rational subspace.
\begin{altenumerate}
\item[{\rm (i)}] Consider the $\varpi$-adic completion
\[
R^\circ\langle \left(\frac{f_1^\sharp}{g^\sharp}\right)^{1/p^\infty},\ldots,\left(\frac{f_n^\sharp}{g^\sharp}\right)^{1/p^\infty}\rangle 
\]
of the subring
\[
R^\circ[ \left(\frac{f_1^\sharp}{g^\sharp}\right)^{1/p^\infty},\ldots,\left(\frac{f_n^\sharp}{g^\sharp}\right)^{1/p^\infty}]\subset R[\frac 1{g^\sharp}]\ .
\]
Then $R^\circ\langle \left(\frac{f_1^\sharp}{g^\sharp}\right)^{1/p^\infty},\ldots,\left(\frac{f_n^\sharp}{g^\sharp}\right)^{1/p^\infty}\rangle^a$ is a perfectoid $K^{\circ a}$-algebra.
\item[{\rm (ii)}] The algebra $\mathcal{O}_X(U^\sharp)$ is a perfectoid $K$-algebra, with associated perfectoid $K^{\circ a}$-algebra
\[
\mathcal{O}_X(U^\sharp)^{\circ a} = R^\circ\langle \left(\frac{f_1^\sharp}{g^\sharp}\right)^{1/p^\infty},\ldots,\left(\frac{f_n^\sharp}{g^\sharp}\right)^{1/p^\infty}\rangle^a\ .
\]
\item[{\rm (iii)}] The tilt of $\mathcal{O}_X(U^\sharp)$ is given by $\mathcal{O}_{X^\flat}(U)$.
\end{altenumerate}
\end{lem}

\begin{proof} \begin{altenumerate}
\item[{\rm (i), ($K$ of characteristic $p$)}] Assume that $K$ has characteristic $p$, and identify $K^\flat = K$; the general case is dealt with below. We see from the definition that
\[
R^\circ\langle \left(\frac{f_1}{g}\right)^{1/p^\infty},\ldots,\left(\frac{f_n}{g}\right)^{1/p^\infty}\rangle
\]
is flat over $K^\circ$ and $\varpi$-adically complete.

We want to show that modulo $\varpi$, Frobenius is almost surjective with kernel almost generated by $\varpi^{1/p}$. We have a surjection
\[
R^\circ[T_1^{1/p^\infty},\ldots,T_n^{1/p^\infty}]\rightarrow R^\circ[\left(\frac{f_1}{g}\right)^{1/p^\infty},\ldots,\left(\frac{f_n}{g}\right)^{1/p^\infty}]\  .
\]
Its kernel contains the ideal $I$ generated by all $T_i^{1/p^m}g^{1/p^m} - f_i^{1/p^m}$. We claim that the induced morphism
\[
R^\circ[T_1^{1/p^\infty},\ldots,T_n^{1/p^\infty}] / I\rightarrow R^\circ[\left(\frac{f_1}{g}\right)^{1/p^\infty},\ldots,\left(\frac{f_n}{g}\right)^{1/p^\infty}]
\]
is an almost isomorphism. Indeed, it is an isomorphism after inverting $\varpi$, because this also inverts $g$. If $f$ lies in the kernel of this map, there is some $k$ with $\varpi^k f\in I$. But then $(\varpi^{k/p^m} f)^{p^m}\in I$, and because $I$ is perfect, also $\varpi^{k/p^m}f\in I$. This gives the desired statement.

Reducing modulo $\varpi$, we have an almost isomorphism
\[
R^\circ[T_1^{1/p^\infty},\ldots,T_n^{1/p^\infty}] / (I,\varpi) \rightarrow R^\circ\langle \left(\frac{f_1}{g}\right)^{1/p^\infty},\ldots,\left(\frac{f_n}{g}\right)^{1/p^\infty}\rangle / \varpi\ .
\]
From the definition of $I$, it is immediate that Frobenius gives an isomorphism
\[
R^\circ[T_1^{1/p^\infty},\ldots,T_n^{1/p^\infty}] / (I,\varpi^{1/p})\cong R^\circ[T_1^{1/p^\infty},\ldots,T_n^{1/p^\infty}] / (I,\varpi)\ .
\]
This finally shows that
\[
R^\circ\langle \left(\frac{f_1}{g}\right)^{1/p^\infty},\ldots,\left(\frac{f_n}{g}\right)^{1/p^\infty}\rangle^a
\]
is a perfectoid $K^{\circ a}$-algebra, giving part (i) in characteristic $p$.

\item[{\rm (i)$\Rightarrow$(ii), (General $K$)}] We show that in general, (i) implies (ii). Note that $R^\circ\subset R$ is open and bounded, hence we may choose $R_0=R^\circ$ in Definition \ref{DefinitionPresheaf}. We have the inclusions
\[
R^\circ[ \frac{f_1^\sharp}{g^\sharp},\ldots,\frac{f_n^\sharp}{g^\sharp}]\subset R^\circ[\left(\frac{f_1^\sharp}{g^\sharp}\right)^{1/p^\infty},\ldots,\left(\frac{f_n^\sharp}{g^\sharp}\right)^{1/p^\infty}]\subset R[\frac 1{g^\sharp}]\ .
\]
Moreover, we claim that
\[
\varpi^{nN} R^\circ[\left(\frac{f_1^\sharp}{g^\sharp}\right)^{1/p^\infty},\ldots,\left(\frac{f_n^\sharp}{g^\sharp}\right)^{1/p^\infty}]\subset R^\circ[ \frac{f_1^\sharp}{g^\sharp},\ldots,\frac{f_n^\sharp}{g^\sharp}]\ :
\]
Indeed, $\frac 1{g^\sharp} = \varpi^{-N} \frac{f_n^\sharp}{g^\sharp}$, and any element on the left-hand side can be written as a sum of terms on the right-hand side with coefficients in $\frac 1{(g^\sharp)^n} R^\circ$.

Now we may pass to the $\varpi$-adic completion and get inclusions
\[
R^\circ\langle \frac{f_1^\sharp}{g^\sharp},\ldots,\frac{f_n^\sharp}{g^\sharp}\rangle \subset R^\circ\langle \left(\frac{f_1^\sharp}{g^\sharp}\right)^{1/p^\infty},\ldots,\left(\frac{f_n^\sharp}{g^\sharp}\right)^{1/p^\infty}\rangle\subset R\langle \frac{f_1^\sharp}{g^\sharp},\ldots,\frac{f_n^\sharp}{g^\sharp}\rangle=\mathcal{O}_X(U) .
\]
Thus it follows from part (i) that
\[
\mathcal{O}_X(U) = R^\circ\langle \left(\frac{f_1^\sharp}{g^\sharp}\right)^{1/p^\infty},\ldots,\left(\frac{f_n^\sharp}{g^\sharp}\right)^{1/p^\infty}\rangle[\varpi^{-1}]
\]
is perfectoid, with corresponding perfectoid $K^{\circ a}$-algebra.

\item[{\rm (i), (iii), (General $K$)}] Again, we see from the definition that
\[
R^\circ\langle \left(\frac{f_1^\sharp}{g^\sharp}\right)^{1/p^\infty},\ldots,\left(\frac{f_n^\sharp}{g^\sharp}\right)^{1/p^\infty}\rangle
\]
is flat over $K^\circ$ and $\varpi$-adically complete. We have to show that modulo $\varpi$, Frobenius is almost surjective with kernel almost generated by $\varpi^{1/p}$. We still have the map
\[
R^\circ[T_1^{1/p^\infty},\ldots,T_n^{1/p^\infty}] / I \rightarrow R^\circ[\left(\frac{f_1^\sharp}{g^\sharp}\right)^{1/p^\infty},\ldots,\left(\frac{f_n^\sharp}{g^\sharp}\right)^{1/p^\infty}]\ ,
\]
where $I$ is the ideal generated by all $T_i^{1/p^m} (g^{1/p^m})^\sharp - (f_i^{1/p^m})^\sharp$. Also, we may apply our results for the tilted situation. In particular, we know that $(\mathcal{O}_{X^\flat}(U),\mathcal{O}_{X^\flat}^+(U))$ is a perfectoid affinoid $K^\flat$-algebra. Let $(S,S^+)$ be its tilt. Then $\Spa(S,S^+)\rightarrow X$ factors over $U^\sharp$, and hence we get a map $(\mathcal{O}_X(U^\sharp),\mathcal{O}_X^+(U^\sharp))\rightarrow (S,S^+)$. The composite map
\[
R^\circ\langle T_1^{1/p^\infty},\ldots,T_n^{1/p^\infty}\rangle^a \rightarrow R^\circ\langle \left(\frac{f_1^\sharp}{g^\sharp}\right)^{1/p^\infty},\ldots,\left(\frac{f_n^\sharp}{g^\sharp}\right)^{1/p^\infty}\rangle^a \rightarrow \mathcal{O}_X(U^\sharp)^{\circ a}\rightarrow S^{\circ a}
\]
is a map of perfectoid $K^{\circ a}$-algebras, which is the tilt of the composite map
\[
R^{\flat \circ}\langle T_1^{1/p^\infty},\ldots,T_n^{1/p^\infty}\rangle^a \rightarrow R^{\flat \circ}\langle T_1^{1/p^\infty},\ldots,T_n^{1/p^\infty}\rangle^a / I^\flat \rightarrow \mathcal{O}_{X^\flat}(U)^{\circ a}\ ,
\]
where $I^\flat$ is the corresponding ideal which occurs in the tilted situation. Note that
\[
R^{\flat \circ}\langle T_1^{1/p^\infty},\ldots,T_n^{1/p^\infty}\rangle / (I^\flat,\varpi^\flat) = R^\circ\langle T_1^{1/p^\infty},\ldots,T_n^{1/p^\infty}\rangle / (I,\varpi)
\]
from the explicit description. Since
\[
R^{\flat \circ}\langle T_1^{1/p^\infty},\ldots,T_n^{1/p^\infty}\rangle^a / (I^\flat,\varpi^\flat) \rightarrow \mathcal{O}_{X^\flat}(U)^{\circ a}/\varpi^\flat
\]
is an isomorphism, so is the composite map
\[
R^\circ\langle T_1^{1/p^\infty},\ldots,T_n^{1/p^\infty}\rangle^a / (I,\varpi) \rightarrow R^\circ\langle \left(\frac{f_1^\sharp}{g^\sharp}\right)^{1/p^\infty},\ldots,\left(\frac{f_n^\sharp}{g^\sharp}\right)^{1/p^\infty}\rangle^a/\varpi \rightarrow S^{\circ a}/\varpi\ ,
\]
as it identifies with the previous map under tilting. The first map being surjective, it follows that both maps are isomorphisms. In particular,
\[
R^\circ\langle \left(\frac{f_1^\sharp}{g^\sharp}\right)^{1/p^\infty},\ldots,\left(\frac{f_n^\sharp}{g^\sharp}\right)^{1/p^\infty}\rangle^a/\varpi \cong S^{\circ a}/\varpi\ .
\]
This gives part (i), and hence part (ii), and then the latter isomorphism gives part (iii).
\end{altenumerate}
\end{proof}

We need an approximation lemma.

\begin{lem}\label{ApproximationLemma} Let $R=K\langle T_0^{1/p^\infty},\ldots,T_n^{1/p^\infty} \rangle$. Let $f\in R^\circ$ be a homogeneous element of degree $d\in \mathbb{Z}[\frac 1p]$. Then for any rational number $c\geq 0$ and any $\epsilon>0$, there exists an element
\[
g_{c,\epsilon}\in R^{\flat \circ}= K^{\flat \circ} \langle T_0^{1/p^\infty},\ldots,T_n^{1/p^\infty} \rangle
\]
homogeneous of degree $d$ such that for all $x\in X=\Spa(R,R^\circ)$, we have
\[
|f(x)-g_{c,\epsilon}^\sharp(x)|\leq |\varpi|^{1-\epsilon} \max(|f(x)|,|\varpi|^c)\ .
\]
\end{lem}

\begin{rem} Note that for $\epsilon<1$, the given estimate says in particular that for all $x\in X=\Spa(R,R^\circ)$, we have
\[
\max(|f(x)|,|\varpi|^c) = \max(|g_{c,\epsilon}^\sharp(x)|,|\varpi|^c)\ .
\]
\end{rem}

\begin{proof} We fix $\epsilon>0$, and assume $\epsilon<1$ and $\epsilon\in \mathbb{Z}[\frac 1p]$. We also fix $f$. Then we prove inductively that for any $c$ one can find some $\epsilon(c)>0$ and some
\[
g_c\in R^{\flat \circ}= K^{\flat \circ} \langle T_0^{1/p^\infty},\ldots,T_n^{1/p^\infty}\rangle
\]
homogeneous of degree $d$ such that for all $x\in X=\Spa(R,R^\circ)$, we have
\[
|f(x)-g_c^\sharp(x)|\leq |\varpi|^{1-\epsilon+\epsilon(c)} \max(|f(x)|,|\varpi|^c)\ .
\]
We will need $\epsilon(c)$, as each induction step will lose some small constant because of some almost mathematics involved. Now we argue by induction, increasing from $c$ to $c^\prime = c + a$, where $0<a< \epsilon$ is some fixed rational number in $\mathbb{Z}[\frac 1p]$. The case $c=0$ is obvious: One may take $\epsilon(0)=\epsilon$. We are free to replace $\epsilon(c)$ by something smaller, so without loss of generality, we assume $\epsilon(c)\leq\epsilon - a$ and $\epsilon(c)\in \mathbb{Z}[\frac 1p]$.

Let $X=\Spa(R,R^+)$, where $R^+=R^\circ = K^\circ\langle T_0^{1/p^\infty},\ldots,T_n^{1/p^\infty} \rangle$. Let $U_c\subset X^\flat=\Spa(R^\flat,R^{\flat +})$ be the rational subset given by $|g_c(x)|\leq |\varpi^\flat|^c$. Its preimage $U_c^\sharp\subset X$ is given by $|f(x)|\leq |\varpi|^c$. The condition implies that
\[
h=f-g_c^\sharp \in \varpi^{c+1-\epsilon +\epsilon(c)}\mathcal{O}_X^+(U_c^\sharp)\ .
\]
The previous lemma shows that
\[
\mathcal{O}_X(U_c^\sharp)^{\circ a} = R^\circ\left\langle\left(\frac{g^\sharp_c}{\varpi^c}\right)^{\frac 1{p^{\infty}}}\right\rangle^a\ .
\]

But $h$ is a homogeneous element, so that $h$ lies almost in the $\varpi$-adic completion of
\[
\bigoplus_{i\in \mathbb{Z}[\frac 1p],0\leq i\leq 1} \varpi^{c+1-\epsilon+\epsilon(c)} \left(\frac{g^\sharp_c}{\varpi^c}\right)^i R_{\mathrm{deg}=d-di}^\circ\ .
\]
This shows that we can find elements $r_i\in R^+$ homogeneous of degree $d-di$, such that $r_i\rightarrow 0$, with
\[
h = \sum_{i\in \mathbb{Z}[\frac 1p],0\leq i\leq 1} \varpi^{c+1-\epsilon + \epsilon(c^\prime)} \left(\frac{g^\sharp_c}{\varpi^c}\right)^i r_i\ ,
\]
where we choose some $0<\epsilon(c^\prime)<\epsilon(c)$, $\epsilon(c^\prime)\in \mathbb{Z}[\frac 1p]$. Choose $s_i\in R^{\flat +}$ homogeneous of degree $d-di$, $s_i\rightarrow 0$, such that $\varpi$ divides $r_i-s_i^\sharp$. Now set
\[
g_{c^\prime} = g_c + \sum_{i\in \mathbb{Z}[\frac 1p],0\leq i\leq 1} (\varpi^\flat)^{c+1-\epsilon + \epsilon(c^\prime)} \left(\frac{g_c}{(\varpi^\flat)^c}\right)^i s_i\ .
\]
We claim that for all $x\in X$, we have
\[
|f(x)-g^\sharp_{c^\prime}(x)|\leq |\varpi|^{1-\epsilon + \epsilon(c^\prime)}\max(|f(x)|,|\varpi|^{c^\prime})\ .
\]
Assume first that $|f(x)|> |\varpi|^c$. Then we have $|g^\sharp_c(x)|=|f(x)|> |\varpi|^c$. It is enough to show that
\[
\left|\left((\varpi^\flat)^{c+1-\epsilon + \epsilon(c^\prime)} \left(\frac{g_c}{(\varpi^\flat)^c}\right)^i s_i\right)^\sharp(x)\right|\leq |\varpi|^{1-\epsilon + \epsilon(c^\prime)} |f(x)|\ .
\]
Neglecting $|s_i^\sharp(x)|\leq 1$, the left-hand side is maximal when $i=1$, in which case it evaluates to the right-hand side, so that we get the desired estimate.

Now we are left with the case $|f(x)|\leq |\varpi|^c$. We claim that in fact
\[
|f(x)-g^\sharp_{c^\prime}(x)|\leq |\varpi|^{c^\prime + 1-\epsilon + \epsilon(c^\prime)}
\]
in this case, which is clearly enough. For this, it is enough to see that $f-g^\sharp_{c^\prime}$ is an element of $\varpi^{c+1}\mathcal{O}_X(U_c^\sharp)^\circ$, because $c+1> c^\prime + 1-\epsilon + \epsilon(c^\prime)$. But we have
\[
\frac{g_{c^\prime}}{(\varpi^\flat)^c} = \frac{g_c}{(\varpi^\flat)^c} + \sum_i (\varpi^\flat)^{1-\epsilon + \epsilon(c^\prime)} \left(\frac{g_c}{(\varpi^\flat)^c}\right)^i s_i\ ,
\]
with all terms being in $\mathcal{O}_{X^\flat}(U_c)^\circ$. Hence we get that
\[
\frac{g^\sharp_{c^\prime}}{\varpi^c} = \frac{g^\sharp_c}{\varpi^c} + \sum_i \varpi^{1-\epsilon + \epsilon(c^\prime)} \left(\frac{g^\sharp_c}{\varpi^c}\right)^i r_i
\]
in $\mathcal{O}_X(U_c^\sharp)^\circ$, modulo $\varpi$. Multiplying by $\varpi^c$, this rewrites as
\[
f-g^\sharp_{c^\prime} = f - g^\sharp_c - h = 0
\]
modulo $\varpi^{c+1}$. This gives the desired estimate.
\end{proof}

\begin{cor}\label{TiltingHomeomorphism} Let $(R,R^+)$ be a perfectoid affinoid $K$-algebra, with tilt $(R^\flat,R^{\flat +})$, and let $X=\Spa(R,R^+)$, $X^\flat=\Spa(R^\flat,R^{\flat +})$.
\begin{altenumerate}
\item[{\rm (i)}] For any $f\in R$ and any $c\geq 0$, $\epsilon>0$, there exists $g_{c,\epsilon}\in R^\flat$ such that for all $x\in X$, we have
\[
|f(x)-g_{c,\epsilon}^\sharp(x)|\leq |\varpi|^{1-\epsilon} \max(|f(x)|,|\varpi|^c)\ .
\]
\item[{\rm (ii)}] For any $x\in X$, the completed residue field $\widehat{k(x)}$ is a perfectoid field.
\item[{\rm (iii)}] The morphism $X\rightarrow X^\flat$ induces a homeomorphism, identifying rational subsets.
\end{altenumerate}
\end{cor}

\begin{proof}\begin{altenumerate}
\item[{\rm (i)}] As any maximal point of $\Spa(R,R^+)$ is contained in $\Spa(R,R^\circ)$, and it is enough to
check the inequality at maximal points after increasing $\epsilon$ slightly, it is enough to prove
this if $R^+=R^\circ$. At the expense of enlarging $c$, we may assume that $f\in R^\circ$, and also assume that
$c$ is an integer. Further, we can write $f=g_0^\sharp+\varpi g_1^\sharp + \ldots + \varpi^c g_c^\sharp +
\varpi^{c+1} f_{c+1}$ for certain $g_0,\ldots,g_c\in R^{\flat \circ}$ and $f_{c+1}\in R^\circ$. We can assume
$f_{c+1}=0$. Now we have the map
\[
K\langle T_0^{1/p^\infty},\ldots,T_c^{1/p^\infty}\rangle \rightarrow R
\]
sending $T_i^{1/p^m}$ to $(g_i^{1/p^m})^\sharp$, and $f$ is the image of $T_0 + \varpi T_1 + \ldots +
\varpi^c T_c$, to which we may apply Lemma \ref{ApproximationLemma}.

\item[{\rm (ii), ($K$ of characteristic $p$)}] In this case, we know that $\mathcal{O}_X(U)^{\circ a}$ is perfectoid for any rational subset $U$. It follows that the $\varpi$-adic completion of $\mathcal{O}_{X,x}^{\circ a}$ is a perfectoid $K^{\circ a}$-algebra, hence $\widehat{k(x)}$ is a perfectoid $K$-algebra. As it is also a nonarchimedean field, the result follows.

\item[{\rm (iii)}] First, part (i) immediately implies that any rational subset of $X$ is the preimage of a rational subset of $X^\flat$. Because $X$ is $T_0$, this implies that the map is injective. Now any $x\in X^\flat$ factors as a composite $R^\flat \rightarrow \widehat{k(x)}\rightarrow \Gamma \cup \{0\}$. As $\widehat{k(x)}$ is perfectoid, we may untilt to a perfectoid field over $K$, and we may also untilt the valuation by Proposition \ref{TiltingValuationFields1}. This shows that the map is surjective, giving part (iii). Now part (ii) follows in general with the same proof.
\end{altenumerate}
\end{proof}

For any subset $M\subset X$, we write $M^\flat\subset X^\flat$ for the corresponding subset of $X^\flat$.

\begin{cor}\label{StructureSheavesTilted} Let $(R,R^+)$ be a perfectoid affinoid $K$-algebra, with tilt $(R^\flat,R^{\flat +})$, and let $X=\Spa(R,R^+)$, $X^\flat=\Spa(R^\flat,R^{\flat +})$. Then for all rational $U\subset X$, the pair $(\mathcal{O}_X(U),\mathcal{O}_X^+(U))$ is a perfectoid affinoid $K$-algebra with tilt $(\mathcal{O}_{X^\flat}(U^\flat),\mathcal{O}_{X^\flat}^+(U^\flat))$.
\end{cor}

\begin{proof} Corollary \ref{TiltingHomeomorphism} (iii) and Lemma \ref{AlmostDescriptionSheaf} (ii) show that $(\mathcal{O}_X(U),\mathcal{O}_X^+(U))$ is a perfectoid affinoid $K$-algebra. It can be characterized by the universal property of Proposition \ref{UnivPropertyPresheaf} among all perfectoid affinoid $K$-algebras, and tilting this universal property shows that its tilt has the analogous universal property characterizing $(\mathcal{O}_{X^\flat}(U^\flat),\mathcal{O}_{X^\flat}^+(U^\flat))$ among all perfectoid affinoid $K^\flat$-algebras.
\end{proof}

At this point, we have proved parts (i) and (ii) of Theorem \ref{MainThmAnalytic}.

To prove the sheaf properties, we start in characteristic $p$, with a certain class of perfectoid rings which are particularly easy to access.

\begin{definition} Assume $K$ is of characteristic $p$. Then a perfectoid affinoid $K$-algebra $(R,R^+)$ is said to be p-finite if there exists a reduced affinoid $K$-algebra $(S,S^+)$ of topologically finite type such that $(R,R^+)$ is the completed perfection of $(S,S^+)$, i.e. $R^+$ is the $\varpi$-adic completion of $\varinjlim_\Phi S^+$ and $R=R^+[\varpi^{-1}]$.
\end{definition}

At this point, let us recall some facts about reduced affinoid $K$-algebras of topologically finite type.

\begin{prop} Let $(S,S^+)$ be a reduced affinoid $K$-algebra of topologically finite type, and let $X=\Spa(S,S^+)$.
\begin{altenumerate}
\item[{\rm (i)}] The subset $S^+=S^\circ\subset S$ is open and bounded.
\item[{\rm (ii)}] For any rational subset $U\subset X$, the affinoid $K$-algebra $(\mathcal{O}_X(U),\mathcal{O}_X^+(U))$ is reduced and of topologically finite type.
\item[{\rm (iii)}] For any covering $X=\bigcup U_i$ by finitely many rational subsets $U_i\subset X$, each cohomology group of the complex
\[
0\rightarrow \mathcal{O}_X(X)^\circ\rightarrow \prod_i \mathcal{O}_X(U_i)^\circ\rightarrow \prod_{i,j} \mathcal{O}_X(U_i\cap U_j)^\circ\rightarrow \ldots
\]
is annihilated by some power of $\varpi$.
\end{altenumerate}
\end{prop}

\begin{proof} Using \cite{HuberDefAdic}, Proposition 4.3 and its proof, one sees that all statements are readily translated into the classical language of rigid geometry, and we use results from the book of Bosch-G{\"u}ntzer-Remmert, \cite{BoschGuentzerRemmert}. Part (i) is precisely their 6.2.4 Theorem 1, and part (ii) is 7.3.2 Corollary 10.

Moreover, Tate's acyclicity theorem, 8.2.1 Theorem 1 in \cite{BoschGuentzerRemmert}, says that
\[
0\rightarrow \mathcal{O}_X(X)\buildrel {d_0}\over \rightarrow \prod_i \mathcal{O}_X(U_i)\buildrel {d_1}\over \rightarrow \prod_{i,j} \mathcal{O}_X(U_i\cap U_j)\buildrel {d_2}\over \rightarrow \ldots
\]
is exact. Then $\ker d_i$ is a closed subspace of a $K$-Banach space, hence itself a $K$-Banach space, and $d_{i-1}$ is a surjection onto $\ker d_i$. By Banach's open mapping theorem, the map $d_{i-1}$ is an open map to $\ker d_i$. This says that the subspace and quotient topologies on $\ker d_i = \im d_{i-1}$ coincide. Now consider the sequence
\[
0\rightarrow \mathcal{O}_X(X)^\circ\buildrel {d_0^\circ}\over \rightarrow \prod_i \mathcal{O}_X(U_i)^\circ\buildrel {d_1^\circ}\over \rightarrow \prod_{i,j} \mathcal{O}_X(U_i\cap U_j)^\circ\buildrel {d_2^\circ}\over \rightarrow \ldots \ .
\]
By parts (i) and (ii), the quotient topology on $\im d_{i-1}$ has $\varpi^n\im d_{i-1}^\circ$, $n\in \mathbb{Z}$, as a basis of open neighborhoods of $0$, and the subspace topology of $\ker d_i$ has $\varpi^n \mathrm{ker} d_i^\circ$, $n\in \mathbb{Z}$, as a basis of open neighborhoods of $0$. That they agree precisely amounts to saying that the cohomology group is annihilated by some power of $\varpi$.
\end{proof}

\begin{prop}\label{SheafPFiniteCase} Assume that $K$ is of characteristic $p$, and that $(R,R^+)$ is p-finite, given as the completed perfection of a reduced affinoid $K$-algebra $(S,S^+)$ of topologically finite type.
\begin{altenumerate}
\item[{\rm (i)}] The map $X=\Spa(R,R^+)\cong Y=\Spa(S,S^+)$ is a homeomorphism identifying rational subspaces.
\item[{\rm (ii)}] For any $U\subset X$ rational, corresponding to $V\subset Y$, the perfectoid affinoid $K$-algebra $(\mathcal{O}_X(U),\mathcal{O}_X^+(U))$ is equal to the completed perfection of $(\mathcal{O}_Y(V),\mathcal{O}_Y^+(V))$.
\item[{\rm (iii)}] For any covering $X=\bigcup_i U_i$ by rational subsets, the sequence
\[
0\rightarrow \mathcal{O}_X(X)^{\circ a}\rightarrow \prod_i \mathcal{O}_X(U_i)^{\circ a}\rightarrow \prod_{i,j} \mathcal{O}_X(U_i\cap U_j)^{\circ a}\rightarrow \ldots
\]
is exact. In particular, $\mathcal{O}_X$ is a sheaf, and $H^i(X,\mathcal{O}_X^{\circ a})=0$ for $i>0$.
\end{altenumerate}
\end{prop}

\begin{rem} The last assertion is equivalent to the assertion that $H^i(X,\mathcal{O}_X^+)$ is annihilated by $\mathfrak{m}$.
\end{rem}

\begin{proof}\begin{altenumerate}
\item[{\rm (i)}] Going to the perfection does not change the associated adic space and rational subspaces, and going to the completion does not by Proposition \ref{SpaCompletion}.

\item[{\rm (ii)}] The completed perfection of $(\mathcal{O}_Y(V),\mathcal{O}_Y^+(V))$ is a perfectoid affinoid $K$-algebra. It has the universal property defining $(\mathcal{O}_X(U),\mathcal{O}_X^+(U))$ among all perfectoid affinoid $K$-algebras.

\item[{\rm (iii)}] Note that the corresponding sequence for $Y$ is exact up to some $\varpi$-power. Hence after taking the perfection, it is almost exact, and stays so after completion.\end{altenumerate}
\end{proof}

\begin{lem}\label{DirectLimitofPFinite} Assume $K$ is of characteristic $p$.
\begin{altenumerate}
\item[{\rm (i)}] Any perfectoid affinoid $K$-algebra $(R,R^+)$ for which $R^+$ is a $K^\circ$-algebra
is the completion of a filtered direct limit of p-finite perfectoid affinoid $K$-algebras $(R_i,R_i^+)$.
\item[{\rm (ii)}] This induces a homeomorphism $\Spa(R,R^+)\cong \varprojlim \Spa(R_i,R_i^+)$, and each rational $U\subset X=\Spa(R,R^+)$ comes as the preimage of some rational $U_i\subset X_i=\Spa(R_i,R_i^+)$.
\item[{\rm (iii)}] In this case $(\mathcal{O}_X(U),\mathcal{O}_X^+(U))$ is equal to the completion of the filtered direct limit of the $(\mathcal{O}_{X_j}(U_j),\mathcal{O}_{X_j}^+(U_j))$, where $U_j$ is the preimage of $U_i$ in $X_j$ for $j\geq i$.
\item[{\rm (iv)}] If $U_i\subset X_i$ is some quasicompact open subset containing the image of $X$, then there is some $j$ such that the image of $X_j$ is contained in $U_i$.
\end{altenumerate}
\end{lem}

\begin{proof}\begin{altenumerate}
\item[{\rm (i)}] For any finite subset $I\subset R^+$, we have the $K$-subalgebra $S_I\subset R$ given as the image of $K\langle T_i|i\in I\rangle\rightarrow R$. Then $S_I$ is a reduced quotient of $K\langle T_i|i\in I\rangle$, and we give $S_I$ the quotient topology. Let $S_I^+\subset S_I$ be the set of power-bounded elements; it is also the set of elements integral over $K^\circ\langle T_i|i\in I\rangle$ by \cite{TateRigidAnalytic}, Theorem 5.2. In particular, $S_I^+\subset R^+$. We caution the reader that $S_I^+$ is in general not the preimage of $R^+$ in $S_I$.

Now let $(R_I,R_I^+)$ be the completed perfection of $(S_I,S_I^+)$, i.e. $R_I^+$ is the $\varpi$-adic completion of $\varinjlim_\Phi S_I^+$, and $R_I=R_I^+[\varpi^{-1}]$. We get an induced map $(R_I,R_I^+)\rightarrow (R,R^+)$, and $(R_I,R_I^+)$ is a p-finite perfectoid affinoid $K$-algebra. We claim that $R^+/\varpi^n = \varinjlim_I R_I^+/\varpi^n$. Indeed, the map is clearly surjective. It is also injective, since if $f_1,f_2\in R_I^+$ satisfy $f_1-f_2 = \varpi^n g$ for some $g\in R^+$, then for some larger $J\supset I$ containing $g$, also $f_1 - f_2\in \varpi^n R_J^+$. This shows that $R^+$ is the completed direct limit of the $R_I^+$, i.e. $(R,R^+)$ is the completed direct limit of the $(R_I,R_I^+)$.

\item[{\rm (ii)}] Let $(L,L^+)$ be the direct limit of the $(R_I,R_I^+)$, equipped with the $\varpi$-adic topology. Then one checks by hand that $\Spa(L,L^+)\cong \varprojlim \Spa(R_i,R_i^+)$, compatible with rational subspaces. But then the same thing holds true for the completed direct limit by Proposition \ref{SpaCompletion}.

\item[{\rm (iii)}] The completion of the direct limit of the $(\mathcal{O}_{X_j}(U_j),\mathcal{O}_{X_j}^+(U_j))$ is a perfectoid affinoid $K$-algebra, and it satisfies the universal property describing $(\mathcal{O}_X(U),\mathcal{O}_X^+(U))$.

\item[{\rm (iv)}] This is an abstract property of spectral spaces and spectral maps. Let $A_i$ be the closed complement of $U_i$, and for any $j\geq i$, let $A_j$ be the preimage of $A_i$ in $X_j$. Then the $A_j$ are constructible subsets of $X_j$, hence spectral, and the transition maps between the $A_j$ are spectral. If one gives the $A_j$ the constructible topology, they are compact topological spaces, and the transition maps are continuous. If their inverse limit is zero, then one of them has to be zero.
\end{altenumerate}
\end{proof}

\begin{prop}\label{SheafGeneralCase} Let $K$ be of any characteristic, and let $(R,R^+)$ be a perfectoid affinoid $K$-algebra, $X=\Spa(R,R^+)$. For any covering $X=\bigcup_i U_i$ by finitely many rational subsets, the sequence
\[
0\rightarrow \mathcal{O}_X(X)^{\circ a}\rightarrow \prod_i \mathcal{O}_X(U_i)^{\circ a}\rightarrow \prod_{i,j} \mathcal{O}_X(U_i\cap U_j)^{\circ a}\rightarrow \ldots
\]
is exact. In particular, $\mathcal{O}_X$ is a sheaf, and $H^i(X,\mathcal{O}_X^{\circ a})=0$ for $i>0$.
\end{prop}

\begin{proof} Assume first that $K$ has characteristic $p$. We may replace $K$ by a perfectoid subfield, such as the
$\varpi$-adic completion of $\mathbb{F}_p((\varpi))(\varpi^{1/p^\infty})$; this ensures that for any perfectoid
affinoid $K$-algebra $(R,R^+)$, the ring $R^+$ is a $K^\circ$-algebra. Then use Lemma \ref{DirectLimitofPFinite}
to write $X=\Spa(R,R^+)\cong\varprojlim X_i=\Spa(R_i,R_i^+)$ as an inverse limit, with $(R_i,R_i^+)$ p-finite.
Any rational subspace comes from a finite level, and a cover by finitely many rational subspaces is the pullback of a cover by finitely many rational subspaces on a finite level. Hence the almost exactness of the sequence follows by taking the completion of the direct limit of the corresponding statement for $X_i$, which is given by Proposition \ref{SheafPFiniteCase}. The rest follows as before.

In characteristic $0$, first use the exactness of the tilted sequence, then reduce modulo $\varpi^\flat$ (which is still exact by flatness), and then remark that this is just the original sequence reduced modulo $\varpi$. As this is exact, the original sequence is exact, by flatness and completeness. Again, we also get the other statements.
\end{proof}

This finishes the proof of Theorem \ref{MainThmAnalytic}.
\end{proof}

We see that to any perfectoid affinoid $K$-algebra $(R,R^+)$, we have associated an affinoid adic space $X=\Spa(R,R^+)$. We call these spaces affinoid perfectoid spaces.

\begin{definition} A perfectoid space is an adic space over $K$ that is locally isomorphic to an affinoid perfectoid space. Morphisms between perfectoid spaces are the morphisms of adic spaces.
\end{definition}

The process of tilting glues.

\begin{definition} We say that a perfectoid space $X^\flat$ over $K^\flat$ is the tilt of a perfectoid space $X$ over $K$ if there is a functorial isomorphism $\Hom(\Spa(R^\flat,R^{\flat +}),X^\flat) = \Hom(\Spa(R,R^+),X)$ for all perfectoid affinoid $K$-algebras $(R,R^+)$ with tilt $(R^\flat,R^{\flat +})$.
\end{definition}

\begin{prop}\label{TiltingEquivalenceSpaces} Any perfectoid space $X$ over $K$ admits a tilt $X^\flat$, unique up to unique isomorphism. This induces an equivalence between the category of perfectoid spaces over $K$ and the category of perfectoid spaces over $K^\flat$. The underlying topological spaces of $X$ and $X^\flat$ are naturally identified. A perfectoid space $X$ is affinoid perfectoid if and only if its tilt $X^\flat$ is affinoid perfectoid. Finally, for any affinoid perfectoid subspace $U\subset X$, the pair $(\mathcal{O}_X(U),\mathcal{O}_X^+(U))$ is a perfectoid affinoid $K$-algebra with tilt $(\mathcal{O}_{X^\flat}(U^\flat),\mathcal{O}_{X^\flat}^+(U^\flat))$.
\end{prop}

\begin{proof} This is a formal consequence of Theorem \ref{TiltingEquivalence}, Theorem \ref{MainThmAnalytic} and Proposition \ref{MapsToAffinoid}. Note that to any open $U\subset X$, one gets an associated perfectoid space with underlying topological space $U$ by restricting the structure sheaf and valuations to $U$, and hence its global sections are $(\mathcal{O}_X(U),\mathcal{O}_X^+(U))$, so that if $U$ is affinoid, then $U=\Spa(\mathcal{O}_X(U),\mathcal{O}_X^+(U))$. This gives the last part of the proposition.
\end{proof}

Let us finish this section by noting one way in which perfectoid spaces behave better than adic spaces (cf. Proposition 1.2.2 of \cite{Huber}).

\begin{prop}\label{FibreProducts} If $X\rightarrow Y\leftarrow Z$ are perfectoid spaces over $K$, then the fibre product $X\times_Y Z$ exists in the category of adic spaces over $K$, and is a perfectoid space.
\end{prop}

\begin{proof} As usual, one reduces to the affine case, $X=\Spa(A,A^+)$, $Y=\Spa(B,B^+)$ and $Z=\Spa(C,C^+)$, and we want to construct $W=X\times_Y Z$. This is given by $W=\Spa(D,D^+)$, where $D$ is the completion of $A\otimes_B C$, and $D^+$ is the completion of the integral closure of the image of $A^+\otimes_{B^+} C^+$ in $D$. Note that $\widehat{A^{\circ a}\otimes_{B^{\circ a}} C^{\circ a}}$ is a perfectoid $K^{\circ a}$-algebra: It is enough to check that $A^{\circ a}\otimes_{B^{\circ a}} C^{\circ a}/\varpi$ is flat over $K^{\circ a}/\varpi$, hence one reduces to characteristic $p$. Here, it is enough to check that $A^{\circ a}\otimes_{B^{\circ a}} C^{\circ a}$ is $\varpi$-torsion free; but if $\varpi f=0$, then $\varpi^{1/p} f^{1/p} = 0$ by perfectness, hence $\varpi^{1/p} f=0$. Continuing gives the result. In particular, $(D,D^+)$ is a perfectoid affinoid $K$-algebra. One immediately checks that it satisfies the desired universal property.
\end{proof}

\section{Perfectoid spaces: Etale topology}

In this section, we use the term locally noetherian adic space over $k$ for the adic spaces over $k$ considered in \cite{Huber}, i.e. they are locally of the form $\Spa(A,A^+)$, were $A$ is a strongly noetherian Tate $k$-algebra. If additionally, they are quasicompact and quasiseparated, we call them noetherian adic spaces.

Although perfectoid rings are always reduced, and hence a definition involving lifting of nilpotents is not possible, there is a good notion of \'{e}tale morphisms. In the following definition, $k$ can be an arbitrary nonarchimedean field.

\begin{definition}
\begin{altenumerate}
\item[{\rm (i)}] A morphism $(R,R^+)\rightarrow (S,S^+)$ of affinoid $k$-algebras is called finite \'{e}tale if $S$ is a finite \'{e}tale $R$-algebra with the induced topology, and $S^+$ is the integral closure of $R^+$ in $S$.
\item[{\rm (ii)}] A morphism $f:X\rightarrow Y$ of adic spaces over $k$ is called finite \'{e}tale if there is a cover of $Y$ by open affinoids $V\subset Y$ such that the preimage $U=f^{-1}(V)$ is affinoid, and the associated morphism of affinoid $k$-algebras
\[
(\mathcal{O}_Y(V),\mathcal{O}_Y^+(V))\rightarrow (\mathcal{O}_X(U),\mathcal{O}_X^+(U))
\]
is finite \'{e}tale.
\item[{\rm (iii)}] A morphism $f:X\rightarrow Y$ of adic spaces over $k$ is called \'{e}tale if for any point $x\in X$ there are open neighborhoods $U$ and $V$ of $x$ and $f(x)$ and a commutative diagram
\[\xymatrix{
U\ar[r]^j \ar[dr]_{f|_U}& W\ar[d]^{p}\\
&V
}\]
where $j$ is an open embedding and $p$ is finite \'{e}tale.
\end{altenumerate}
\end{definition}

For locally noetherian adic spaces over $k$, this recovers the usual notions, by Example 1.6.6 ii) and Lemma 2.2.8 of \cite{Huber}, respectively. We will see that these notions are useful in the case of perfectoid spaces, and will not use them otherwise. However, we will temporarily need a stronger notion of \'{e}tale morphisms for perfectoid spaces. After proving the almost purity theorem, we will see that there is no difference. In the following let $K$ be a perfectoid field again.

\begin{definition}
\begin{altenumerate}
\item[{\rm (i)}] A morphism $(R,R^+)\rightarrow (S,S^+)$ of perfectoid affinoid $K$-algebras is called strongly finite \'{e}tale if it is finite \'{e}tale and additionally $S^{\circ a}$ is a finite \'{e}tale $R^{\circ a}$-algebra.
\item[{\rm (ii)}] A morphism $f:X\rightarrow Y$ of perfectoid spaces over $K$ is called strongly finite \'{e}tale if there is a cover of $Y$ by open affinoid perfectoids $V\subset Y$ such that the preimage $U=f^{-1}(V)$ is affinoid perfectoid, and the associated morphism of perfectoid affinoid $K$-algebras
\[
(\mathcal{O}_Y(V),\mathcal{O}_Y^+(V))\rightarrow (\mathcal{O}_X(U),\mathcal{O}_X^+(U))
\]
is strongly finite \'{e}tale.
\item[{\rm (iii)}] A morphism $f:X\rightarrow Y$ of perfectoid spaces over $K$ is called strongly \'{e}tale if for any point $x\in X$ there are open neighborhoods $U$ and $V$ of $x$ and $f(x)$ and a commutative diagram
\[\xymatrix{
U\ar[r]^j \ar[dr]_{f|_U}& W\ar[d]^{p}\\
&V
}\]
where $j$ is an open embedding and $p$ is strongly finite \'{e}tale.
\end{altenumerate}
\end{definition}

From the definitions, Proposition \ref{TiltingEquivalenceSpaces}, and Theorem \ref{AlmostFiniteEtaleCovers}, we see that $f: X\rightarrow Y$ is strongly finite \'{e}tale, resp. strongly \'{e}tale, if and only if the tilt $f^\flat: X^\flat\rightarrow Y^\flat$ is strongly finite \'{e}tale, resp. strongly \'{e}tale. Moreover, in characteristic $p$, anything (finite) \'{e}tale is also strongly (finite) \'{e}tale.

\begin{lem}\label{PullbackFiniteEtale}
\begin{altenumerate}
\item[{\rm (i)}] Let $f: X\rightarrow Y$ be a strongly finite \'{e}tale, resp. strongly \'{e}tale, morphism of perfectoid spaces and let $g: Z\rightarrow Y$ be an arbitrary morphism of perfectoid spaces. Then $X\times_Y Z\rightarrow Z$ is a strongly finite \'{e}tale, resp. strongly \'{e}tale, morphism of perfectoid spaces. Moreover, the map of underlying topological spaces $|X\times_Z Y|\rightarrow |X| \times_{|Z|} |Y|$ is surjective.
\item[{\rm (ii)}] If in (i), all spaces $X=\Spa(A,A^+)$, $Y=\Spa(B,B^+)$ and $Z=\Spa(C,C^+)$ are affinoid, with $(A,A^+)$ strongly finite \'{e}tale over $(B,B^+)$, then $X\times_Y Z=\Spa(D,D^+)$, where $D=A\otimes_B C$ and $D^+$ is the integral closure of $C^+$ in $D$, and $(D,D^+)$ is strongly finite \'{e}tale over $(C,C^+)$.
\item[{\rm (iii)}] Assume that $K$ is of characteristic $p$. If $f: X\rightarrow Y$ is a finite \'{e}tale, resp. \'{e}tale, morphism of adic spaces over $k$ and $g: Z\rightarrow Y$ is a map from a perfectoid space $Z$ to $Y$, then the fibre product $X\times_Y Z$ exists in
the category of adic spaces over $K$, is a perfectoid space, and the projection $X\times_Y Z\rightarrow Z$ is finite \'{e}tale, resp. \'{e}tale. Moreover, the map of underlying topological spaces $|X\times_Z Y|\rightarrow |X| \times_{|Z|} |Y|$ is surjective.
\item[{\rm (iv)}] Assume that in the situation of (iii), all spaces $X=\Spa(A,A^+)$, $Y=\Spa(B,B^+)$ and $Z=\Spa(C,C^+)$ are affinoid, with $(A,A^+)$ finite \'{e}tale over $(B,B^+)$, then $X\times_Y Z=\Spa(D,D^+)$, where $D=A\otimes_B C$, $D^+$ is the integral closure of $C^+$ in $D$ and $(D,D^+)$ is finite \'{e}tale over $(C,C^+)$.
\end{altenumerate}
\end{lem}

\begin{proof}\begin{altenumerate}
\item[{\rm (ii)}] As $A\otimes_B C$ is finite projective over $C$, it is already complete. One easily deduces the universal property. Also, $D^{\circ a}$ is finite \'{e}tale over $C^{\circ a}$, as base-change preserves finite \'{e}tale morphisms.

\item[{\rm (i)}] Applying the definition of a strongly finite \'{e}tale map, one reduces the statement about strongly finite \'{e}tale maps to the situation handled in part (ii). Now the statement for strongly \'{e}tale maps follows, because fibre products obviously preserve open embeddings. The surjectivity statement follows from the argument of \cite{HuberDefAdic}, proof of Lemma 3.9 (i).

\item[{\rm (iv)}] Proposition \ref{AlmostPurityCharP} shows that $D=A\otimes_B C$ is perfectoid. Therefore $\Spa(D,D^+)$ is a perfectoid space. One easily checks the universal property.

\item[{\rm (iii)}] The finite \'{e}tale case reduces to the situation considered in part (iv). Again, it is trivial to handle open embeddings, giving also the \'{e}tale case. Surjectivity is proved as before.
\end{altenumerate}
\end{proof}

Let us recall the following statement about henselian rings.

\begin{prop}\label{FiniteEtaleHenselian} Let $A$ be a flat $K^\circ$-algebra such that $A$ is henselian along $(\varpi)$. Then the categories of finite \'{e}tale $A[\varpi^{-1}]$ and finite \'{e}tale $\hat{A}[\varpi^{-1}]$-algebras are equivalent, where $\hat{A}$ is the $\varpi$-adic completion of $A$.
\end{prop}

\begin{proof} See e.g. \cite{GabberRamero}, Proposition 5.4.53.
\end{proof}

We recall that $\varpi$-adically complete algebras $A$ are henselian along $(\varpi)$, and that if $A_i$ is a direct system of $K^\circ$-algebras henselian along $(\varpi)$, then so is the direct limit $\varinjlim A_i$. In particular, we get the following lemma.

\begin{lem}\label{FiniteEtaleDirectLimit}
\begin{altenumerate}
\item[{\rm (i)}] Let $A_i$ be a filtered direct system of complete flat $K^\circ$-algebras, and let $A$ be the completion of the direct limit, which is again a complete flat $K^\circ$-algebra. Then we have an equivalence of categories
\[
A[\varpi^{-1}]_\fet\cong 2-\varinjlim A_i[\varpi^{-1}]_\fet\ .
\]
In particular, if $R_i$ is a filtered direct system of perfectoid $K$-algebras and $R$ is the completion of their direct limit, then $R_\fet\cong 2-\varinjlim (R_i)_\fet$.
\item[{\rm (ii)}] Assume that $K$ has characteristic $p$, and let $(R,R^+)$ be a p-finite perfectoid affinoid $K$-algebra, given as the completed perfection of the reduced affinoid $K$-algebra $(S,S^+)$ of topologically finite type. Then $R_\fet\cong S_\fet$.
\end{altenumerate}
\end{lem}

\begin{proof}
\begin{altenumerate}
\item[{\rm (i)}] Because finite \'{e}tale covers, and morphisms between these, are finitely presented objects, we have
\[
(\varinjlim A_i[\varpi^{-1}])_\fet\cong 2-\varinjlim A_i[\varpi^{-1}]_\fet\ .
\]
On the other hand, $\varinjlim A_i$ is henselian along $(\varpi)$, hence the left-hand side agrees with $A[\varpi^{-1}]_\fet$ by Proposition \ref{FiniteEtaleHenselian}.
\item[{\rm (ii)}] From part (i), we know that
\[
R_\fet\cong 2-\varinjlim S^{1/p^n}_\fet\ .
\]
But the categories $S^{1/p^n}_\fet$ are all equivalent to $S_\fet$.
\end{altenumerate}
\end{proof}

\begin{prop}\label{DecompletionFiniteEtaleGlobal} If $f:X\rightarrow Y$ is a strongly finite \'{e}tale morphism of perfectoid spaces, then for any open affinoid perfectoid $V\subset Y$, its preimage $U$ is affinoid perfectoid, and
\[
(\mathcal{O}_Y(V),\mathcal{O}_Y^+(V))\rightarrow (\mathcal{O}_X(U),\mathcal{O}_X^+(U))
\]
is strongly finite \'{e}tale.
\end{prop}

\begin{proof} Tilting the situation and using Theorem \ref{AlmostFiniteEtaleCovers}, we immediately reduce to the case that $K$ is of characteristic $p$. Again, we replace $K$ by a perfectoid subfield to ensure that $R^+$ is a $K^\circ$-algebra in all cases.

We may assume $Y=V=\Spa(R,R^+)$ is affinoid. Writing $(R,R^+)$ as the completion of the direct limit of p-finite perfectoid affinoid $K$-algebras $(R_i,R_i^+)$ as in Lemma \ref{DirectLimitofPFinite}, we see that $Y$ is already defined as a finite \'{e}tale cover of some $Y_i=\Spa(R_i,R_i^+)$: Indeed, there are finitely many rational subsets of $Y$ over which we have a finite \'{e}tale cover; by Lemma \ref{FiniteEtaleDirectLimit} (i), these are defined over a finite level, and because $Y$ is quasi-separated, also the gluing data over intersections (as well as the cocycle condition) are defined over a finite level.

Hence by Lemma \ref{PullbackFiniteEtale} (ii), we are reduced to the case that $(R,R^+)$ is p-finite, given as the completed perfection of $(S,S^+)$. But then $X$ is already defined as a finite \'{e}tale cover of $Z=\Spa(S,S^+)$ by similar reasoning using Lemma \ref{FiniteEtaleDirectLimit} (ii), and we conclude by using the result for locally noetherian adic spaces, cf. \cite{Huber}, Example 1.6.6 (ii), and Lemma \ref{PullbackFiniteEtale} (iv).
\end{proof}

Using Proposition \ref{AlmostPurityCharP}, this shows that if $K$ is of characteristic $p$, then the finite \'{e}tale covers of an affinoid perfectoid space $X=\Spa(R,R^+)$ are the same as the finite \'{e}tale covers of $R$.

The same method also proves the following proposition.

\begin{prop}\label{EtaleComesFromAdicMap} Assume that $K$ is of characteristic $p$. Let $f:X\rightarrow Y$ be an \'{e}tale map of perfectoid spaces. Then for any $x\in X$, there exist affinoid perfectoid neighborhoods $x\in U\subset X$, $f(U)\subset V\subset Y$, and an \'{e}tale morphism of affinoid noetherian adic spaces $U^0\rightarrow V^0$ over $K$, such that $U=U^0\times_{V^0} V$.
\end{prop}

\begin{proof} We may assume that $X$ and $Y$ affinoid perfectoid, and that $X$ is a rational subdomain of a finite \'{e}tale cover of $Y$. Then one reduces to the p-finite case by the same argument as above, and hence to noetherian adic spaces.
\end{proof}

\begin{cor}\label{CompositeEtale} Strongly \'{e}tale maps of perfectoid spaces are open. If $f:X\rightarrow Y$ and $g:Y\rightarrow Z$ are strongly (finite) \'{e}tale morphisms of perfectoid spaces, then the composite $g\circ f$ is strongly (finite) \'{e}tale.
\end{cor}

\begin{proof} We may assume that $K$ has characteristic $p$. The first part follows directly from the previous proposition and the result for locally noetherian adic spaces, cf. \cite{Huber}, Proposition 1.7.8. For the second part, argue as in the previous proposition for both $f$ and $g$ to reduce to the analogous result for locally noetherian adic spaces, \cite{Huber}, Proposition 1.6.7 (ii).
\end{proof}

The following theorem gives a strong form of Faltings's almost purity theorem.

\begin{thm}\label{AlmostEtaleChar0} Let $(R,R^+)$ be a perfectoid affinoid $K$-algebra, and let $X=\Spa(R,R^+)$ with tilt $X^\flat$.
\begin{altenumerate}
\item[{\rm (i)}] For any open affinoid perfectoid subspace $U\subset X$, we have a fully faithful functor from the category of strongly finite \'{e}tale covers of $U$ to the category of finite \'{e}tale covers of $\mathcal{O}_X(U)$, given by taking global sections.
\item[{\rm (ii)}] For any $U$, this functor is an equivalence of categories.
\item[{\rm (iii)}] For any finite \'{e}tale cover $S/R$, $S$ is perfectoid and $S^{\circ a}$ is finite \'{e}tale over $R^{\circ a}$. Moreover, $S^{\circ a}$ is a uniformly almost finitely generated $R^{\circ a}$-module.
\end{altenumerate}
\end{thm}

\begin{proof}
\begin{altenumerate}
\item[{\rm (i)}] By Proposition \ref{DecompletionFiniteEtaleGlobal} and Theorem \ref{AlmostFiniteEtaleCovers}, the perfectoid spaces strongly finite \'{e}tale over $U$ are the same as the finite \'{e}tale $\mathcal{O}_X(U)^{\circ a}$-algebras, which are a full subcategory of the finite \'{e}tale $\mathcal{O}_X(U)$-algebras.

\item[{\rm (ii)}] We may assume that $U=X$. Fix a finite \'{e}tale $R$-algebra $S$. First we check that for any $x\in X$, we can find an affinoid perfectoid neighborhood $x\in U\subset X$ and a strongly finite \'{e}tale cover $V\rightarrow U$ which gives via (i) the finite \'{e}tale algebra $\mathcal{O}_X(U)\otimes_R S$ over $\mathcal{O}_X(U)$.

As a first step, note that we have an equivalence of categories between the direct limit of the category of finite \'{e}tale $\mathcal{O}_X(U)$-algebras over all affinoid perfectoid neighborhoods $U$ of $x$ and the category of finite \'{e}tale covers of the completion $\widehat{k(x)}$ of the residue field at $x$, by Lemma \ref{FiniteEtaleDirectLimit} (i). The latter is a perfectoid field.

By Theorem \ref{TiltingEquivFields}, the categories $\widehat{k(x)}_\fet$ and $\widehat{k(x^\flat)}_\fet$ are equivalent, where $k(x^\flat)$ is the residue field of $X^\flat$ at the point $x^\flat$ corresponding to $x$. Combining, we see that
\[
2-\varinjlim_{x\in U} (\mathcal{O}_X(U))_\fet\cong 2-\varinjlim_{x\in U} (\mathcal{O}_{X^\flat}(U^\flat))_\fet\ .
\]
In particular, we can find $V^\flat$ finite \'{e}tale over $U^\flat$ for some $U$ such that the pullbacks of $S$ to $\widehat{k(x)}$ resp. of the global sections of $V^\flat$ to $\widehat{k(x^\flat)}$ are tilts; but then, they are already identified over some smaller neighborhood. Shrinking $U$, we untilt to get the desired strongly finite \'{e}tale $V\rightarrow U$.

This shows that there is a cover $X=\bigcup U_i$ by finitely many rational subsets and strongly finite \'{e}tale maps $V_i\rightarrow U_i$ such that the global sections of $V_i$ are $S_i = \mathcal{O}_X(U_i)\otimes_R S$. Let $S_i^+$ be the integral closure of $\mathcal{O}_X^+(U_i)$ in $S_i$; then $V_i=\Spa(S_i,S_i^+)$.

By Lemma \ref{PullbackFiniteEtale} (ii), the pullback of $V_i$ to some affinoid perfectoid $U^\prime\subset U_i$ has the same description, involving $\mathcal{O}_X(U^\prime)\otimes_R S$, and hence the $V_i$ glue to some perfectoid space $Y$ over $X$, and $Y\rightarrow X$ is strongly finite \'{e}tale. By Proposition \ref{DecompletionFiniteEtaleGlobal}, $Y$ is affinoid perfectoid, i.e. $Y=\Spa(A,A^+)$, with $(A,A^+)$ an affinoid perfectoid $K$-algebra. It suffices to show that $A=S$. But the sheaf property of $\mathcal{O}_Y$ gives us an exact sequence
\[
0\rightarrow A\rightarrow \prod_i \mathcal{O}_X(U_i)\otimes_R S \rightarrow \prod_{i,j} \mathcal{O}_X(U_i\cap U_j)\otimes_R S\ .
\]
On the other hand, the sheaf property for $\mathcal{O}_X$ gives an exact sequence
\[
0\rightarrow R\rightarrow \prod_i \mathcal{O}_X(U_i) \rightarrow \prod_{i,j} \mathcal{O}_X(U_i\cap U_j)\ .
\]
Because $S$ is flat over $R$, tensoring is exact, and the first sequence is identified with the second sequence after $\otimes_R S$. Therefore $A=S$, as desired.

\item[{\rm (iii)}] This is a formal consequence of part (ii), Proposition \ref{DecompletionFiniteEtaleGlobal} and Theorem \ref{AlmostFiniteEtaleCovers}.
\end{altenumerate}
\end{proof}

We see in particular that any (finite) \'{e}tale morphism of perfectoid spaces is strongly (finite) \'{e}tale. Now one can also pullback \'{e}tale maps between adic spaces in characteristic $0$.

\begin{prop}\label{PullbackEtaleFromAdic} Parts (iii) and (iv) of Lemma \ref{PullbackFiniteEtale} stay true in characteristic $0$.
\end{prop}

\begin{proof} The same proof as for Lemma \ref{PullbackFiniteEtale} works, using Theorem \ref{AlmostEtaleChar0} (iii).
\end{proof}

Finally, we can define the \'{e}tale site of a perfectoid space.

\begin{definition} Let $X$ be a perfectoid space. Then the \'{e}tale site of $X$ is the category $X_\et$ of perfectoid spaces which are \'{e}tale over $X$, and coverings are given by topological coverings. The associated topos is denoted $X_\et^\sim$.
\end{definition}

The previous results show that all conditions on a site are satisfied, and that a morphism $f:X\rightarrow Y$ of perfectoid spaces induces a morphism of sites $X_\et\rightarrow Y_\et$. Also, a morphism $f:X\rightarrow Y$ from a perfectoid space $X$ to a locally noetherian adic space $Y$ induces a morphism of sites $X_\et\rightarrow Y_\et$.

After these preparations, we get the technical main result.

\begin{thm}\label{TiltingEquivalenceSites} Let $X$ be a perfectoid space over $K$ with tilt $X^\flat$ over $K^\flat$. Then the tilting operation induces an isomorphism of sites $X_\et\cong X^\flat_\et$. This isomorphism is functorial in $X$.
\end{thm}

\begin{proof} This is immediate.
\end{proof}

The almost vanishing of cohomology proved in Proposition \ref{SheafGeneralCase} extends to the \'{e}tale topology.

\begin{prop}\label{SheafEtaleCase} For any perfectoid space $X$ over $K$, the sheaf $U\mapsto \mathcal{O}_U(U)$ is a sheaf $\mathcal{O}_X$ on $X_\et$, and $H^i(X_\et,\mathcal{O}_X^{\circ a})=0$ for $i>0$ if $X$ is affinoid perfectoid.
\end{prop}

\begin{proof} It suffices to check exactness of
\[
0\rightarrow \mathcal{O}_X(X)^{\circ a}\rightarrow \prod_i \mathcal{O}_{U_i}(U_i)^{\circ a}\rightarrow \prod_{i,j} \mathcal{O}_{U_i\times_X U_j}(U_i\times_X U_j)^{\circ a}\rightarrow \ldots
\]
for any covering of an affinoid perfectoid $X$ by finitely many \'{e}tale $U_i\rightarrow X$ given as rational subsets of finite \'{e}tale maps to rational subsets of $X$. Under tilting, this reduces to the assertion in characteristic $p$, and then to the assertion for p-finite $(R,R^+)$. In that case, one uses that the analogous sequence for noetherian adic spaces is exact up to a bounded $\varpi$-power, and hence after taking the perfection almost exact.
\end{proof}

To make use of the \'{e}tale site of a perfectoid space, we have to compare the \'{e}tale sites of perfectoid spaces with those of locally noetherian adic spaces. This is possible under a certain assumption, cf. Section 2.4 of \cite{Huber} for an analogous result.

\begin{definition}\label{DecompletionAdicSpace} Let $X$ be a perfectoid space. Further, let $X_i$, $i\in I$, be a filtered inverse system of noetherian adic spaces over $K$, and let $\varphi_i: X\rightarrow X_i$, $i\in I$, be a map to the inverse system.

Then we write $X\sim \varprojlim X_i$ if the mapping of underlying topological spaces $|X|\rightarrow \varprojlim |X_i|$ is a homeomorphism, and for any $x\in X$ with images $x_i\in X_i$, the map of residue fields
\[
\varinjlim k(x_i)\rightarrow k(x)
\]
has dense image.
\end{definition}

\begin{rem} We recall that by assumption all $X_i$ are qcqs. If $X\sim \varprojlim X_i$, then $|X|$ is an inverse limit of spectral spaces with spectral transition maps, hence spectral, and in particular qcqs again.
\end{rem}

\begin{prop}\label{DecompletionPullbackEtale} Let the situation be as in Definition \ref{DecompletionAdicSpace}, and let $Y\rightarrow X_i$ be an \'{e}tale morphism of noetherian adic spaces. Then $Y\times_{X_i} X \sim \varprojlim_{j\geq i} Y\times_{X_i} X_j$.
\end{prop}

\begin{proof} The same proof as for Remark 2.4.3 of \cite{Huber} works. In particular, we note that in Definition \ref{DecompletionAdicSpace}, if $|X|\rightarrow \varprojlim |X_i|$ is bijective and the condition on residue fields is satisfied, then already $X\sim \varprojlim X_i$, i.e. $|X|\rightarrow \varprojlim |X_i|$ is a homeomorphism.
\end{proof}

With this definition, we have the following analogue of Proposition 2.4.4 of \cite{Huber}.

\begin{thm}\label{DecompletionTopoi} Let the situation be as in Definition \ref{DecompletionAdicSpace}. Then $X_\et^\sim$ is a projective limit of the fibred topos $(X_{i,\et}^\sim)_i$.
\end{thm}

\begin{proof} The same proof as for Proposition 2.4.4 of \cite{Huber} works, except that one uses that any \'{e}tale morphism factors locally as the composite of an open immersion and a finite \'{e}tale map instead of appealing to Corollary 1.7.3 of \cite{Huber} on the top of page 128. The latter kind of morphisms can be descended to a finite level because of Lemma \ref{FiniteEtaleDirectLimit}.
\end{proof}

As in \cite{Huber}, Corollary 2.4.6, this gives the following corollary.

\begin{cor}\label{DecompletionCohom} Let the situation be as in Definition \ref{DecompletionAdicSpace}, and let $F_i$ be a sheaf of abelian groups on $X_{i,\et}$, with preimages $F_j$ on $X_{j,\et}$ for $j\geq i$ and $F$ on $X_\et$. Then the natural mapping
\[
\varinjlim H^n(X_{j,\et},F_j)\rightarrow H^n(X_\et,F)
\]
is bijective for all $n\geq 0$.$\hfill \Box$
\end{cor}

In some cases, one can even say more.

\begin{cor}\label{DecompletionFrob} Assume that in the situation of Definition \ref{DecompletionAdicSpace} all transition maps $X_j\rightarrow X_i$ induce purely inseparable extensions on completed residue fields and homeomorphisms $|X_j|\rightarrow |X_i|$. Then $X_\et^\sim$ is equivalent to $X_{i,\et}^\sim$ for any $i$.
\end{cor}

\begin{proof} Use the remark after Proposition 2.3.7 of \cite{Huber}.
\end{proof}

\section{An example: Toric varieties}

Let us recall the definition of a toric variety, valid over any field $k$.

\begin{definition} A toric variety over $k$ is a normal separated scheme $X$ of finite type over $k$ with an action of a split torus $T\cong \mathbb{G}_m^k$ on $X$ and a point $x\in X(k)$ with trivial stabilizer in $T$, such that the $T$-orbit $T\cong Tx= U\subset X$ of $x$ is open and dense.
\end{definition}

We recall that toric varieties may be described in terms of fans.

\begin{definition} Let $N$ be a free abelian group of finite rank.
\begin{altenumerate}
\item[{\rm (i)}] A strongly convex polyhedral cone $\sigma$ in $N\otimes \mathbb{R}$ is a subset of the form $\sigma=\mathbb{R}_{\geq 0} x_1+ \ldots + \mathbb{R}_{\geq 0} x_n$ for certain $x_1,\ldots,x_n\in N$, subject to the condition that $\sigma$ contains no line through the origin.
\item[{\rm (ii)}] A fan $\Sigma$ in $N\otimes \mathbb{R}$ is a nonempty finite collection of strongly convex polyhedral cones stable under taking faces, and such that the intersection of any two cones in $\Sigma$ is a face of both of them.
\end{altenumerate}
\end{definition}

Let $M=\Hom(N,\mathbb{Z})$ be the dual lattice. To any strongly convex polyhedral cone $\sigma\subset N\otimes \mathbb{R}$, one gets the dual $\sigma^\vee\subset M\otimes \mathbb{R}$, and we associate to $\sigma$ the variety
\[
U_\sigma = \Spec k[\sigma^\vee\cap M]\ .
\]
We denote the function on $U_\sigma$ corresponding to $u\in \sigma^\vee\cap M$ by $\chi^u$. If $\tau$ is a face of $\sigma$, then $\sigma^\vee\subset \tau^\vee$, inducing an open immersion $U_\tau\rightarrow U_\sigma$. These maps allow us to glue a variety $X_{\Sigma}$ associated to any fan $\Sigma$. Note that $T = U_{\{0\}} = \Spec k[M]$ is a torus, which acts on $X_{\Sigma}$ with open dense orbit $U_{\{0\}}\subset X_{\Sigma}$. Also $T$ has the base point $1\in T$, giving a point $x\in X(k)$, making $X_{\Sigma}$ a toric variety. Let us recall the classification of toric varieties.

\begin{thm} Any toric variety over $k$ is canonically isomorphic to $X_{\Sigma}$ for a unique fan $\Sigma$ in $X_\ast(T)\otimes \mathbb{R}$.
\end{thm}

We also need to recall some statements about divisors on toric varieties.

\begin{defprop} Let $\{\tau_i\}\subset \Sigma$ be the $1$-dimensional cones, and fix a generator $v_i\in \tau_i\cap N$ of $\tau_i\cap N$. Each $\tau_i$ gives rise to $U_{\tau_i}\cong \mathbb{A}^1\times \mathbb{G}_m^{k-1}$, giving rise to a $T$-invariant Weil divisor $D_i = D(\tau_i)$ on $X_{\Sigma}$, defined as the closure of $\{0\}\times \mathbb{G}_m^{k-1}$.

A $T$-Weil divisor is by definition an element of $\bigoplus_i \mathbb{Z} D_i$. Every Weil divisor is equivalent to a $T$-Weil divisor. If $D=\sum a_i D_i$ is a $T$-Weil divisor, then
\[
H^0(X_{\Sigma},\mathcal{O}(D)) = \bigoplus_{\substack{u\in M\\ \langle u,v_i\rangle\geq -a_i}} k \chi^u\ .
\]
\end{defprop}

Now we adapt these definitions to the world of usual adic spaces, and to the world of perfectoid spaces. Assume first that $k$ is a complete nonarchimedean field, and let $\Sigma$ be a fan as above. Then we can associate to $\Sigma$ the adic space $\mathcal{X}^\ad_{\Sigma}$ of finite type over $k$ which is glued out of
\[
\mathcal{U}_\sigma^{\ad} = \Spa(k\langle \sigma^\vee\cap M\rangle,k^\circ \langle \sigma^\vee\cap M\rangle)\ .
\]
We note that this is not in general the adic space $X_{\Sigma}^\ad$ over $k$ associated to the variety $X_{\Sigma}$: For example, if $X_{\Sigma}$ is just affine space, then $\mathcal{X}_{\Sigma}^\ad$ will be a closed unit ball. In general, let $X_{\Sigma,k^\circ}$ be the toric scheme over $k^\circ$ associated to $\Sigma$. Let $\hat{X}_{\Sigma,k^\circ}$ be the formal completion of $X_{\Sigma,k^\circ}$ along its special fibre, which is an admissible formal scheme over $k^\circ$. Then $\mathcal{X}^\ad_\Sigma$ is the generic fibre $\hat{X}^\ad_{\Sigma,k^\circ}$ associated to $\hat{X}_{\Sigma,k^\circ}$. In particular, if $X_{\Sigma}$ is proper, then $X^\ad_{\Sigma} = \mathcal{X}^\ad_{\Sigma}$.

Similarly, if $K$ is a perfectoid field, we can associate a perfectoid space $\mathcal{X}^\perf_{\Sigma}$ over $K$ to $\Sigma$, which is glued out of
\[
\mathcal{U}_\sigma^\perf = \Spa(K\langle \sigma^\vee \cap M[p^{-1}]\rangle,K^\circ \langle \sigma^\vee\cap M[p^{-1}]\rangle)\ .
\]

Note that on $\mathcal{X}^\perf_{\Sigma}$, we have a sheaf $\mathcal{O}(D)$ for any $D\in \bigoplus \mathbb{Z}[p^{-1}] D_i$. Moreover, $H^0(\mathcal{X}^\perf_{\Sigma},\mathcal{O}(D))$ is the free Banach-$K$-vector space with basis given by $\{\chi^u\}$, where $u$ ranges over $u\in M[p^{-1}]$ with $\langle u,v_i\rangle\geq -a_i$.

We have the following comparison statements. Note that any toric variety $X_{\Sigma}$ comes with a map $\varphi: X_{\Sigma}\rightarrow X_{\Sigma}$ induced from multiplication by $p$ on $M$; the same applies to $\mathcal{X}_{\Sigma}^\ad$, etc. . For clarity, we use subscripts to denote the field over which we consider the toric variety.

\begin{thm} Let $K$ be a perfectoid field with tilt $K^\flat$.
\begin{altenumerate}
\item[{\rm (i)}] The perfectoid space $\mathcal{X}_{\Sigma,K}^\perf$ tilts to $\mathcal{X}_{\Sigma,K^\flat}^\perf$.
\item[{\rm (ii)}] The perfectoid space $\mathcal{X}_{\Sigma,K}^\perf$ can be written as
\[
\mathcal{X}_{\Sigma,K}^\perf\sim \varprojlim_{\varphi} \mathcal{X}_{\Sigma,K}^\ad\ .
\]
\item[{\rm (iii)}] There is a homeomorphism of topological spaces
\[
|\mathcal{X}_{\Sigma,K^\flat}^\ad|\cong \varprojlim_\varphi |\mathcal{X}_{\Sigma,K}^\ad|\ .
\]
\item[{\rm (iv)}] There is an isomorphism of \'{e}tale topoi
\[
(\mathcal{X}_{\Sigma,K^\flat}^\ad)_\et^\sim\cong \varprojlim_\varphi (\mathcal{X}_{\Sigma,K}^\ad)_\et^\sim\ .
\]
\item[{\rm (v)}] For any open subset $U\subset \mathcal{X}_{\Sigma,K}^\ad$ with preimage $V\subset \mathcal{X}_{\Sigma,K^\flat}^\ad$, we have a morphism of \'{e}tale topoi $V_\et^\sim\rightarrow U_\et^\sim$, giving a commutative diagram
\[\xymatrix{
V_\et^\sim\ar@{^(->}[r]\ar[d] & (\mathcal{X}_{\Sigma,K^\flat}^\ad)_\et^\sim \ar[d]\\
U_\et^\sim\ar@{^(->}[r] & (\mathcal{X}_{\Sigma,K}^\ad)_\et^\sim
}\]
\end{altenumerate}
\end{thm}

\begin{proof} This is an immediate consequence of our previous results: Part (i) can be checked on affinoid pieces, where it is an immediate generalization of Proposition \ref{ExampleTilting}. Part (ii) can be checked one affinoid pieces again, where it is easy. Then parts (iii) and (iv) follow from Theorem \ref{DecompletionTopoi}, Corollary \ref{DecompletionFrob} and the preservation of topological spaces and \'{e}tale topoi under tilting. Finally, part (v) follows from Proposition \ref{DecompletionPullbackEtale}, using the previous arguments.
\end{proof}

Let us denote by $\pi: \mathcal{X}_{\Sigma,K^\flat}^\ad\rightarrow \mathcal{X}_{\Sigma,K}^\ad$ the projection, which exists on topological spaces and \'{e}tale topoi. In the following, we restrict to proper smooth toric varieties for simplicity.

\begin{prop}\label{CohomToric} Let $X_\Sigma$ be a proper smooth toric variety. Let $\ell\neq p$ be prime. Assume that $K$, and hence $K^\flat$, is algebraically closed. Then for all $i\in \mathbb{Z}$, the projection map $\pi$ induces an isomorphism
\[
H^i(X_{\Sigma,K,\et}^\ad,\mathbb{Z}/\ell^m\mathbb{Z})\cong H^i(X_{\Sigma,K^\flat,\et}^\ad,\mathbb{Z}/\ell^m\mathbb{Z})\ .
\]
\end{prop}

\begin{proof} This follows from part (iv) of the previous theorem combined with the observation that $\varphi: X_{\Sigma,K}^\ad\rightarrow X_{\Sigma,K}^\ad$ induces an isomorphism on cohomology with $\mathbb{Z}/\ell^m \mathbb{Z}$-coefficients. Using proper base change, this can be checked on $X_{\Sigma,\kappa}$, where $\kappa$ is the residue field of $K$. But here, $\varphi$ is purely inseparable, and hence induces an equivalence of \'{e}tale topoi.
\end{proof}

We need the following approximation property.

\begin{prop}\label{ApproximationHypersurface} Assume that $X_{\Sigma,K}$ is proper smooth. Let $Y\subset X_{\Sigma,K}$ be a hypersurface. Let $\tilde{Y}\subset X_{\Sigma,K}^\ad$ be a small open neighborhood of $Y$. Then there exists a hypersurface $Z\subset X_{\Sigma,K^\flat}$ such that $Z^\ad\subset \pi^{-1}(\tilde{Y})$. One can assume that $Z$ is defined over a given dense subfield of $K^\flat$.
\end{prop}

\begin{proof} Let $D=\sum a_i D_i$ be a $T$-Weil divisor representing $Y$. Let $f\in H^0(X_{\Sigma,K},\mathcal{O}(D))$ be the equation with zero locus $Y$. Consider the graded ring
\[
\bigoplus_{j\in \mathbb{Z}[p^{-1}]} H^0(\mathcal{X}_{\Sigma,K}^\perf,\mathcal{O}(jD)) = \bigoplus_{j\in \mathbb{Z}[p^{-1}]} \widehat{\bigoplus}_{\substack{u\in M[p^{-1}]\\ \langle u,v_i\rangle\geq -ja_i}} K \chi^u\ ,
\]
and let $R$ be its completion (with respect to the obvious $K^\circ$-submodule). Here $\widehat{\bigoplus}$ denotes the Banach space direct sum. Then as in Proposition \ref{ExampleTilting}, $R$ is a perfectoid $K$-algebra whose tilt is given by the similar construction over $K^\flat$. Note that $D$ is given combinatorially and hence transfers to $K^\flat$.

We may assume that $\tilde{Y}$ is given by
\[
\tilde{Y} = \{x\in X_{\Sigma,K}^\ad\mid |f(x)|\leq \epsilon\}\ ,
\]
for some $\epsilon$. In order to make sense of the inequality $|f(x)|\leq \epsilon$, note that $X_{\Sigma,K}$ and $\mathcal{O}(D)$ have a tautological integral model over $K^\circ$ (by applying the toric constructions over $K^\circ$), which is enough to talk about absolute values: Trivialize the line bundle $\mathcal{O}(D)$ locally on the integral model to interpret $f$ as a function; any two different choices differ by a unit of $K^\circ$, and hence give the same absolute value.

Now the analogue of Lemma \ref{ApproximationLemma} holds true for $R$, with the same proof. This implies that we can find $g\in H^0(\mathcal{X}_{\Sigma,K^\flat}^\perf,\mathcal{O}(D))$ such that
\[
\pi^{-1}(\tilde{Y}) = \{x\in \mathcal{X}_{\Sigma,K^\flat}^\perf\mid |g(x)|\leq \epsilon\}\ .
\]
Let $k\subset K^\flat$ be a dense subfield. Changing $g$ slightly, we can assume that
\[
g\in \bigoplus_{\substack{u\in M[p^{-1}]\\ \langle u,v_i\rangle\geq -ja_i}} k \chi^u\ .
\]
Replacing $g$ by a large $p$-power gives a regular function $h$ on $X_{\Sigma,K^\flat}$, such that its zero locus $Z\subset X_{\Sigma,K^\flat}$ is contained in $\pi^{-1}(\tilde{Y})$, as desired.
\end{proof}

By intersecting several hypersurfaces, one arrives at the following corollary.

\begin{cor}\label{ApproximationCompleteIntersection} Assume that $X_{\Sigma,K}$ is projective and smooth. Let $Y\subset X_{\Sigma,K}$ be a set-theoretic complete intersection, i.e. $Y$ is set-theoretically equal to an intersection $Y_1\cap \ldots \cap Y_c$ of hypersurfaces $Y_i\subset X_{\Sigma,K}$, where $c$ is the codimension of $Y$. Let $\tilde{Y}\subset X_{\Sigma,K}^\ad$ be a small open neighborhood of $Y$. Then there exists a closed subvariety $Z\subset X_{\Sigma,K^\flat}$ such that $Z^\ad\subset \pi^{-1}(\tilde{Y})$ with $\dim Z = \dim Y$. One can assume that $Z$ is defined over a given dense subfield of $K^\flat$.
\end{cor}

\begin{proof} The only nontrivial point is to check that the intersection over $K^\flat$ will be nonempty; if the dimension was too large, one can also just cut by further hypersurfaces. For nonemptiness, choose an ample line bundle to define a notion of degree of subvarieties; then the degree of a complete intersection is determined combinatorially. As $Y$ has positive degree, so has the corresponding complete intersection over $K^\flat$, and in particular is nonempty.
\end{proof}

\section{The weight-monodromy conjecture}

We first recall some facts about $\ell$-adic representations of the absolute Galois group $G_k=\Gal(\bar{k}/k)$ of a local field $k$ of residue characteristic $p$, cf. \cite{IllusieAutourDuTML}. Let $q$ be the cardinality of the residue field of $k$. Recall that the maximal pro-$\ell$-quotient of the inertia subgroup $I_k\subset G_k$ is given by the quotient $t_\ell: I_k\rightarrow \mathbb{Z}_\ell(1)$, which is the inverse limit of the homomorphisms $t_{\ell,n}: I_k\rightarrow \mu_{\ell^n}$ defined by choosing a system of $\ell^n$-th roots $\varpi^{1/\ell^n}$ of a uniformizer $\varpi$ of $k$, and requiring
\[
\sigma(\varpi^{1/\ell^n}) = t_{\ell,n}(\sigma) \varpi^{1/\ell^n}
\]
for all $\sigma\in I_k$. Now recall Grothendieck's quasi-unipotence theorem.

\begin{prop} Let $V$ be a finite-dimensional $\bar{\mathbb{Q}}_\ell$-representation of $G_k$, given by a map $\rho: G_k\rightarrow \GL(V)$. Then there is an open subgroup $I_1\subset I_k$ such that for all $\sigma\in I_1$, the element $\rho(\sigma)\in \GL(V)$ is unipotent, and in this case there is a unique nilpotent morphism $N: V\rightarrow V(-1)$ such that for all $\sigma\in I_1$,
\[
\rho(\sigma) = \exp(Nt_\ell(g))\ .
\]
\end{prop}

We fix an isomorphism $\mathbb{Q}_\ell(1)\cong \mathbb{Q}_\ell$ in order to consider $N$ as a nilpotent endomorphism of $V$. Changing the choice of isomorphism $\mathbb{Q}_\ell(1)\cong \mathbb{Q}_\ell$ replaces $N$ by a scalar multiple, which will have no effect on the following discussion. From uniqueness of $N$, it follows that for any geometric Frobenius element $\Phi\in G_k$, we have $N \Phi = q \Phi N$.

Also recall the general monodromy filtration.

\begin{defprop} Let $V$ be a finite-dimensional vector space over any field, and let $N: V\rightarrow V$ be a nilpotent morphism. Then there is a unique separated and exhaustive increasing filtration $\Fil_i^N V\subset V$, $i\in \mathbb{Z}$, called the monodromy filtration, such that $N(\Fil_i^N V)\subset \Fil_{i-2}^N V$ for all $i\in \mathbb{Z}$ and $\gr_i^N V\cong \gr_{-i}^N V$ via $N^i$ for all $i\geq 0$.
\end{defprop}

In fact, we have the formula
\[
\Fil_i^N V = \sum_{i_1-i_2=i} \ker N^{i_1+1}\cap \im N^{i_2}
\]
for the monodromy filtration. Now we can formulate the weight-monodromy conjecture.

\begin{conj}[Deligne, {\cite{DeligneICM}}]\label{WeightMonodromy} Let $X$ be a proper smooth variety over $k$, and let $V=H^i(X_{\bar{k}},\bar{\mathbb{Q}}_\ell)$. Then for all $j\in \mathbb{Z}$ and for any geometric Frobenius $\Phi\in G_k$, all eigenvalues of $\Phi$ on $\gr_j^N V$ are Weil numbers of weight $i+j$, i.e. algebraic numbers $\alpha$ such that $|\alpha|=q^{(i+j)/2}$ for all complex absolute values.
\end{conj}

We note that in order to prove this conjecture, one is allowed to replace $k$ by a finite extension. Using the formalism of $\zeta$- and $L$-functions, the conjecture has the following interpretation. Let $\mathbb{K}$ be a global field, and let $X/\mathbb{K}$ be a proper smooth variety. Choose some integer $i$. Recall that the $L$-function associated to the $i$-th $\ell$-adic cohomology group $H^i(X) = H^i(X_{\bar{\mathbb{K}}},\bar{\mathbb{Q}}_\ell)$ of $X$ is defined as a product
\[
L(H^i(X),s) = \prod_v L_v(H^i(X),s)\ ,
\]
where the product runs over all places $v$ of $K$. Let us recall the definition of the local factor at a finite prime $v$ not dividing $\ell$, whose local field is $k$:
\[
L_v(H^i(X),s) = \det(1-q^{-s}\Phi | H^i(X)^{I_k})^{-1}\ .
\]
At primes of good reduction, the Weil conjectures imply that all poles of this expression have real part $\frac i2$. Moreover, one checks in the usual way that hence the product defining $L(H^i(X),s)$ is absolutely convergent when the real part of $s$ is greater than $\frac i2 + 1$, except possibly for finitely many factors. The weight-monodromy conjecture implies that all other local factors will not have any poles of real part greater than $\frac i2$.

Over equal characteristic local fields, Deligne, \cite{DeligneWeil2}, turned this argument into a proof:

\begin{thm} Let $C$ be a curve over $\mathbb{F}_q$, $x\in C(\mathbb{F}_q)$, such that $k$ is the local field of $C$ at $x$. Let $X$ be a proper smooth scheme over $C\setminus \{x\}$. Then the weight-monodromy conjecture holds true for $X_k = X\times_{C\setminus \{x\}} \Spec k$.
\end{thm}

Let us give a brief summary of the proof. Let $f: X\rightarrow C\setminus \{x\}$ be the proper smooth morphism. Possibly replacing $C$ by a finite cover, we may assume that the action of $I_k$ on  $V=H^i(X_{\bar{k}},\bar{\mathbb{Q}}_\ell)$ is unipotent. One considers the local system $R^i f_\ast \bar{\mathbb{Q}}_\ell$ on $C\setminus \{x\}$. By the Weil conjectures, this sheaf is pure of weight $i$. From the formalism of $L$-functions for sheaves over curves, one deduces semicontinuity of weights, which in this situation means that on the invariants $V^{I_k}$ of $V=H^i(X_{\bar{k}},\bar{\mathbb{Q}}_\ell)$, all occuring weights are $\leq i$. A similar property holds true for all tensor powers of $V$, and for all tensor powers of the dual $V^\vee$. Then one applies the following lemma from linear algebra, which applies for all local fields $k$.

\begin{lem} Let $V$ be an $\ell$-adic representation of $G_k$, on which $I_k$ acts unipotently. Then $\gr_j^N V$ is pure of weight $i+j$ for all $j\in \mathbb{Z}$ if and only if for all $j\geq 0$, all weights on $(V^{\otimes j})^{I_k}$ are at most $ij$, and all weights on $((V^{\vee})^{\otimes j})^{I_k}$ are at most $-ij$.
\end{lem}

Our main theorem is the following.

\begin{thm} Let $k$ be a local field of characteristic $0$. Let $Y$ be a geometrically connected proper smooth variety over $k$ such that $Y$ is a set-theoretic complete intersection in a projective smooth toric variety $X_{\Sigma}$. Then the weight-monodromy conjecture is true for $Y$.
\end{thm}

\begin{proof} Let $\varpi\in k$ be a uniformizer, and let $K$ be the completion of $k(\varpi^{1/p^\infty})$; then $K$ is perfectoid. Let $K^\flat$ be its tilt. Then $K^\flat$ is the completed perfection of $k^\prime = \mathbb{F}_q((t))$, where $t=\varpi^\flat$. This gives an isomorphism between the absolute Galois groups of $K$ and $k^\prime$. The notion of weights and the monodromy operator $N$ is compatible with this isomorphism of Galois groups.

By Theorem 3.6 a) of \cite{HuberFiniteness}, there is some open neighborhood $\tilde{Y}$ of $Y^\ad_K$ in $X^\ad_{\Sigma,K}$ such that $\tilde{Y}_{\mathbb{C}_p}$ and $Y_{\mathbb{C}_p}^\ad$ have the same $\mathbb{Z}/\ell\mathbb{Z}$-cohomology. By induction, they have the same $\mathbb{Z}/\ell^m\mathbb{Z}$-cohomology for all $m\geq 1$. Also recall the following comparison theorem.

\begin{thm}[{\cite[Theorem 3.8.1]{Huber}}] Let $X$ be an algebraic variety over an algebraically closed nonarchimedean field $k$, with associated adic space $X^\ad$. Then
\[
H^i(X_\et,\mathbb{Z}/\ell^m\mathbb{Z})\cong H^i(X^\ad_\et,\mathbb{Z}/\ell^m\mathbb{Z})\ .
\]
\end{thm}

By Corollary \ref{ApproximationCompleteIntersection}, there is some closed subvariety $Z\subset X_{\Sigma,K^\flat}$ such that $Z^\ad\subset \pi^{-1}(\tilde{Y})$ and $\dim Z = \dim Y$. Moreover, we can assume that $Z$ is defined over a global field and geometrically irreducible. Let $Z^\prime$ be a projective smooth alteration of $Z$. We get a commutative diagram of \'{e}tale topoi of adic spaces
\[\xymatrix{
(X^\ad_{\Sigma,\mathbb{C}_p^\flat})_\et^\sim\ar[r]^{\pi}& (X^\ad_{\Sigma,\mathbb{C}_p})_\et^\sim \\
(\pi^{-1}(\tilde{Y})_{\mathbb{C}_p^\flat})_\et^\sim\ar@{^(->}[u]\ar[r]&(\tilde{Y}_{\mathbb{C}_p})_\et^\sim\ar@{^(->}[u] \\
(Z^{\prime \ad}_{\mathbb{C}_p^\flat})_\et^\sim\ar[u]& (Y^\ad_{\mathbb{C}_p})_\et^\sim\ar@{^(->}[u]\ .
}\]
There is a canonical action of the absolute Galois group $G=G_K=G_{K^\flat}$ on this diagram such that all morphisms are $G$-equivariant. Using the comparison theorem, this induces a $G$-equivariant map
\[
f^{\ast}: H^i(Y_{\mathbb{C}_p,\et},\mathbb{Z}/\ell^m\mathbb{Z})=H^i(\tilde{Y}_{\mathbb{C}_p,\et},\mathbb{Z}/\ell^m\mathbb{Z})\rightarrow H^i(Z_{\mathbb{C}_p^\flat,\et}^\prime,\mathbb{Z}/\ell^m\mathbb{Z})\ ,
\]
compatible with the cup product. Formally taking the inverse limit and tensoring with $\bar{\mathbb{Q}}_\ell$, it follows that we get a $G$-equivariant map
\[
H^i(Y_{\mathbb{C}_p,\et},\bar{\mathbb{Q}}_{\ell})\rightarrow H^i(Z_{\mathbb{C}_p^\flat,\et}^\prime,\bar{\mathbb{Q}}_{\ell})\ .
\]

\begin{lem} For $i=2\dim Y$, this is an isomorphism.
\end{lem}

\begin{proof} For all $m$, we have a commutative diagram
\[\xymatrix{
H^{2\dim Y}(X_{\Sigma,\mathbb{C}_p^\flat,\et},\mathbb{Z}/\ell^m\mathbb{Z})\ar[d] & H^{2\dim Y}(X_{\Sigma,\mathbb{C}_p,\et},\mathbb{Z}/\ell^m\mathbb{Z})\ar[l]_{\cong}\ar[d] \\
H^{2\dim Y}(\pi^{-1}(\tilde{Y})_{\mathbb{C}_p^\flat,\et},\mathbb{Z}/\ell^m\mathbb{Z})\ar[d] & H^{2\dim Y}(\tilde{Y}_{\mathbb{C}_p,\et},\mathbb{Z}/\ell^m\mathbb{Z})\ar[l]\ar[d]^{\cong} \\
H^{2\dim Y}(Z_{\mathbb{C}_p^\flat,\et}^\prime,\mathbb{Z}/\ell^m\mathbb{Z}) & H^{2\dim Y}(Y_{\mathbb{C}_p,\et},\mathbb{Z}/\ell^m\mathbb{Z}) 
}\]
The isomorphism in the top row is from Proposition \ref{CohomToric}. We can pass to the inverse limit over $m$ and tensor with $\bar{\mathbb{Q}}_{\ell}$. If
\[
H^{2\dim Y}(Y_{\mathbb{C}_p,\et},\bar{\mathbb{Q}}_{\ell})\rightarrow H^{2\dim Y}(Z_{\mathbb{C}_p^\flat,\et}^\prime,\bar{\mathbb{Q}}_{\ell})\ .
\]
is not an isomorphism, it is the zero map, and hence the diagram implies that the restriction map
\[
H^{2\dim Y}(X_{\Sigma,\mathbb{C}_p^\flat,\et},\bar{\mathbb{Q}}_\ell)\rightarrow H^{2\dim Y}(Z_{\mathbb{C}_p^\flat,\et}^\prime,\bar{\mathbb{Q}}_\ell)
\]
is the zero map as well. But the $\dim Y$-th power of the first Chern class of an ample line bundle on $X_{\Sigma,\mathbb{C}_p^\flat}$ will have nonzero image in $H^{2\dim Y}(Z_{\mathbb{C}_p^\flat,\et}^\prime,\bar{\mathbb{Q}}_\ell)$.
\end{proof}

Now the Poincar\'{e} duality pairing implies that $H^i(Y_{\mathbb{C}_p,\et},\bar{\mathbb{Q}}_{\ell})$ is a direct summand of $H^i(Z_{\mathbb{C}_p^\flat,\et}^\prime,\bar{\mathbb{Q}}_{\ell})$. By Deligne's theorem, $H^i(Z_{\mathbb{C}_p^\flat,\et}^\prime,\bar{\mathbb{Q}}_{\ell})$ satisfies the weight-monodromy conjecture, and hence so does its direct summand $H^i(Y_{\mathbb{C}_p,\et},\bar{\mathbb{Q}}_{\ell})$.
\end{proof}

\bibliographystyle{abbrv}
\bibliography{PhDThesis}

\def\cprime{$'$}
\begin{thebibliography}{10}

\bibitem{BerkovichSpectralTheory}
V.~G. Berkovich.
\newblock {\em Spectral theory and analytic geometry over non-{A}rchimedean
  fields}, volume~33 of {\em Mathematical Surveys and Monographs}.
\newblock American Mathematical Society, Providence, RI, 1990.

\bibitem{BoschGuentzerRemmert}
S.~Bosch, U.~G{\"u}ntzer, and R.~Remmert.
\newblock {\em Non-{A}rchimedean analysis}, volume 261 of {\em Grundlehren der
  Mathematischen Wissenschaften [Fundamental Principles of Mathematical
  Sciences]}.
\newblock Springer-Verlag, Berlin, 1984.
\newblock A systematic approach to rigid analytic geometry.

\bibitem{BoschLuetkebohmert}
S.~Bosch and W.~L{\"u}tkebohmert.
\newblock Formal and rigid geometry. {I}. {R}igid spaces.
\newblock {\em Math. Ann.}, 295(2):291--317, 1993.

\bibitem{Boyer}
P.~Boyer.
\newblock Monodromie du faisceau pervers des cycles \'evanescents de quelques
  vari\'et\'es de {S}himura simples.
\newblock {\em Invent. Math.}, 177(2):239--280, 2009.

\bibitem{Boyer2}
P.~Boyer.
\newblock Conjecture de monodromie-poids pour quelques vari\'et\'es de
  {S}himura unitaires.
\newblock {\em Compos. Math.}, 146(2):367--403, 2010.

\bibitem{Caraiani}
A.~Caraiani.
\newblock Local-global compatibility and the action of monodromy on nearby
  cycles.
\newblock arXiv:1010.2188.

\bibitem{Dat}
J.-F. Dat.
\newblock Th\'eorie de {L}ubin-{T}ate non-ab\'elienne et repr\'esentations
  elliptiques.
\newblock {\em Invent. Math.}, 169(1):75--152, 2007.

\bibitem{deJongAlterations}
A.~J. de~Jong.
\newblock Smoothness, semi-stability and alterations.
\newblock {\em Inst. Hautes \'Etudes Sci. Publ. Math.}, (83):51--93, 1996.

\bibitem{DeligneICM}
P.~Deligne.
\newblock Th\'eorie de {H}odge. {I}.
\newblock In {\em Actes du {C}ongr\`es {I}nternational des {M}ath\'ematiciens
  ({N}ice, 1970), {T}ome 1}, pages 425--430. Gauthier-Villars, Paris, 1971.

\bibitem{DeligneWeil2}
P.~Deligne.
\newblock La conjecture de {W}eil. {II}.
\newblock {\em Inst. Hautes \'Etudes Sci. Publ. Math.}, (52):137--252, 1980.

\bibitem{FaltingsPadicHodgeTheory}
G.~Faltings.
\newblock {$p$}-adic {H}odge theory.
\newblock {\em J. Amer. Math. Soc.}, 1(1):255--299, 1988.

\bibitem{FaltingsAlmostEtale}
G.~Faltings.
\newblock Almost \'etale extensions.
\newblock {\em Ast\'erisque}, (279):185--270, 2002.
\newblock Cohomologies $p$-adiques et applications arithm{\'e}tiques, II.

\bibitem{FontaineWintenberger}
J.-M. Fontaine and J.-P. Wintenberger.
\newblock Extensions alg\'ebrique et corps des normes des extensions {APF} des
  corps locaux.
\newblock {\em C. R. Acad. Sci. Paris S\'er. A-B}, 288(8):A441--A444, 1979.

\bibitem{GabberRamero}
O.~Gabber and L.~Ramero.
\newblock {\em Almost ring theory}, volume 1800 of {\em Lecture Notes in
  Mathematics}.
\newblock Springer-Verlag, Berlin, 2003.

\bibitem{HarrisTaylor}
M.~Harris and R.~Taylor.
\newblock {\em The geometry and cohomology of some simple {S}himura varieties},
  volume 151 of {\em Annals of Mathematics Studies}.
\newblock Princeton University Press, Princeton, NJ, 2001.
\newblock With an appendix by Vladimir G. Berkovich.

\bibitem{Hellmann}
E.~Hellmann.
\newblock On arithmetic families of filtered $\varphi$-modules and crystalline
  representations.
\newblock 2011.
\newblock arXiv:1010.4577.

\bibitem{Hochster}
M.~Hochster.
\newblock Prime ideal structure in commutative rings.
\newblock {\em Trans. Amer. Math. Soc.}, 142:43--60, 1969.

\bibitem{HuberContVal}
R.~Huber.
\newblock Continuous valuations.
\newblock {\em Math. Z.}, 212(3):455--477, 1993.

\bibitem{HuberDefAdic}
R.~Huber.
\newblock A generalization of formal schemes and rigid analytic varieties.
\newblock {\em Math. Z.}, 217(4):513--551, 1994.

\bibitem{Huber}
R.~Huber.
\newblock {\em \'{E}tale cohomology of rigid analytic varieties and adic
  spaces}.
\newblock Aspects of Mathematics, E30. Friedr. Vieweg \& Sohn, Braunschweig,
  1996.

\bibitem{HuberFiniteness}
R.~Huber.
\newblock A finiteness result for direct image sheaves on the \'etale site of
  rigid analytic varieties.
\newblock {\em J. Algebraic Geom.}, 7(2):359--403, 1998.

\bibitem{IllusieCotangent}
L.~Illusie.
\newblock {\em Complexe cotangent et d\'eformations. {I}}.
\newblock Lecture Notes in Mathematics, Vol. 239. Springer-Verlag, Berlin,
  1971.

\bibitem{IllusieCotangent2}
L.~Illusie.
\newblock {\em Complexe cotangent et d\'eformations. {II}}.
\newblock Lecture Notes in Mathematics, Vol. 283. Springer-Verlag, Berlin,
  1972.

\bibitem{IllusieAutourDuTML}
L.~Illusie.
\newblock Autour du th\'eor\`eme de monodromie locale.
\newblock {\em Ast\'erisque}, (223):9--57, 1994.
\newblock P{\'e}riodes $p$-adiques (Bures-sur-Yvette, 1988).

\bibitem{ItoUpperHalfPlane}
T.~Ito.
\newblock Weight-monodromy conjecture for {$p$}-adically uniformized varieties.
\newblock {\em Invent. Math.}, 159(3):607--656, 2005.

\bibitem{Ito}
T.~Ito.
\newblock Weight-monodromy conjecture over equal characteristic local fields.
\newblock {\em Amer. J. Math.}, 127(3):647--658, 2005.

\bibitem{KedlayaLiu}
K.~Kedlaya and R.~Liu.
\newblock Relative $p$-adic {H}odge theory, {I}: {F}oundations.
\newblock http://math.mit.edu/$\sim$kedlaya/papers/relative-padic-Hodge1.pdf.

\bibitem{Quillen}
D.~Quillen.
\newblock On the (co-) homology of commutative rings.
\newblock In {\em Applications of {C}ategorical {A}lgebra ({P}roc. {S}ympos.
  {P}ure {M}ath., {V}ol. {XVII}, {N}ew {Y}ork, 1968)}, pages 65--87. Amer.
  Math. Soc., Providence, R.I., 1970.

\bibitem{RapoportZink}
M.~Rapoport and T.~Zink.
\newblock \"{U}ber die lokale {Z}etafunktion von {S}himuravariet\"aten.
  {M}onodromiefiltration und verschwindende {Z}yklen in ungleicher
  {C}harakteristik.
\newblock {\em Invent. Math.}, 68(1):21--101, 1982.

\bibitem{Shin}
S.~W. Shin.
\newblock Galois representations arising from some compact {S}himura varieties.
\newblock {\em Ann. of Math. (2)}, 173(3):1645--1741, 2011.

\bibitem{TateRigidAnalytic}
J.~Tate.
\newblock Rigid analytic spaces.
\newblock {\em Invent. Math.}, 12:257--289, 1971.

\bibitem{TaylorYoshida}
R.~Taylor and T.~Yoshida.
\newblock Compatibility of local and global {L}anglands correspondences.
\newblock {\em J. Amer. Math. Soc.}, 20(2):467--493, 2007.

\bibitem{Terasoma}
T.~Terasoma.
\newblock Monodromy weight filtration is independent of $\ell$.
\newblock 1998.
\newblock arXiv:math/9802051.

\end{thebibliography}

%\newpage
%\thispagestyle{empty}
%\includegraphics[viewport = 50 0 0 800]{Lebenslauf.pdf}

\end{document}